\documentclass{amsart}
\usepackage[ngerman,english]{babel} 
\usepackage[utf8]{inputenc} 
\usepackage[T1]{fontenc} 
\usepackage{csquotes}
\usepackage[margin=2cm]{geometry}
\usepackage{blkarray}
\usepackage{nicematrix}
\usepackage{tabularray}
\usepackage{subfig}
\usepackage[dvipsnames]{xcolor}
\usepackage{amssymb}
\usepackage{amsthm}
\usepackage{amsmath}
\usepackage{pgf,tikz}
\usepackage{tikz-cd}
\usetikzlibrary{decorations.pathmorphing,shapes,calc,babel}
\usepackage{graphicx}
\usepackage{xcolor}
\usepackage{caption}
\usepackage{subcaption}
\usepackage{trimspaces}
\usepackage{svg}
\usepackage{makecell}

\usepackage{hyperref}
\hypersetup{pdfauthor={Veronika K\"orber}, pdftitle={Triangulations and maximal cross-ratio degrees}, pdfkeywords={cross-ratio, degree, maximal, tropical geometry}}

\usepackage{thmtools}
\usepackage[capitalise]{cleveref}

\usepackage[onehalfspacing]{setspace} 
\usetikzlibrary{decorations.markings}

\renewcommand{\epsilon}{\varepsilon}

\theoremstyle{plain}
\newtheorem{theorem}{Theorem}[section]
\newtheorem*{theorem*}{Theorem}
\newtheorem{remark}[theorem]{Remark}
\newtheorem{corollary}[theorem]{Corollary}
\newtheorem{lemma}[theorem]{Lemma}
\newtheorem{proposition}[theorem]{Proposition}
\newtheorem{construction}[theorem]{Construction}
\theoremstyle{definition}
\newtheorem{definition}[theorem]{Definition}
\newtheorem{example}[theorem]{Example}
\theoremstyle{remark}

\newcommand{\U}{\mathcal{U}}
\newcommand{\M}{M}
\newcommand{\R}{\mathbb{R}}
\newcommand{\cros}[4]{(#1 #2|#3 #4)}
\newcommand{\abs}[1]{\left\lvert#1\right\rvert}

\DeclareMathOperator {\mult}{mult}

\usepackage[backend=bibtex,style=alphabetic, maxnames=99]{biblatex}

\hypersetup{ colorlinks=false, linkcolor=black, filecolor=black, urlcolor=black }


\title{Triangulations and maximal cross-ratio degrees}
\author[V. Körber]{Veronika Körber}
\email{veronika.koerber@uni-tuebingen.de}
\address{Eberhard Karls Universität Tübingen,
Geschwister-Scholl-Platz,\newline
72074 Tübingen, Germany}

\begin{document}
\begin{abstract}
	The cross-ratio degree problem is about counting rational curves with $n$ marked points satisfying $n-3$ cross-ratio conditions. This problem has a tropical analogue which provides the same number, as shown by a correspondence theorem. In general, there are no closed formulas for this counting problem. In the special case of cross-ratio conditions given by triangulations, a formula was found in \cite{R24} via techniques of algebraic geometry. 
	
	We study the cross-ratio problem given by triangulations in the tropical world. In addition to computing the cross-ratio degree by tropical means, we provide concrete solutions for the counting problem in arbitrary settings, thus answering the question in \cite[Remark 1.3]{R24}. 
	
	We also use the tropical recursive algorithm by Goldner in \cite{Go21} to provide a new computational tool to compute cross-ratio degrees. With this, we can find the maximal cross-ratio degrees for $n=9,$ $n=10$ and $n=12,$ where the latter case also makes use of a theory developed in \cite{BELL25}. Previously, these numbers were only known up to $n=8.$
\end{abstract} 
\maketitle

\section{Introduction}
Given a line with $n$ marked points, then we can evaluate $n-3$ cross-ratios, i.e. $n-3$ four-tuples of the $n$ points. We denote the set of these four-tuples by $\U.$ This turns into a map $\Pi_\U: \overline{\M}_{0,n} \to (\overline{\M}_{0,4})^{n-3},$ where $\overline{\M}_{0,n}$ denotes the moduli space of $n$-marked rational stable curves. The cross-ratio degree $D_\U$ is the degree of this map. However, this number is in general hard to find and there are no universal formulas yet to determine the degree. 

The same counting problem can be phrased in the tropical world. There, a tropical cross-ratio describes a relation between four marked ends of a tropical curve, see Definition \ref{defCR}, and we consider a map $\pi_\U: \M_{0,n}^{\mathrm{trop}} \to (\M_{0,4}^{\mathrm{trop}})^{n-3}$ between tropical moduli spaces, see Definition \ref{defmod}, and its degree $d_\U,$ see Definition \ref{mappi}. It holds that $d_\U=D_\U,$ which can be seen by the tropical correspondence principles, see \cite{Ty17}, \cite{Ra17}, \cite{Gr16}, \cite{GM10}, \cite{Te07}, \cite{Mi06}, \cite{Ra16} and \cite{FS97}. 

 In this paper, we focus on the computation of $d_\U,$ i.e. on the tropical side. The degree $d_\U$ equals a sum over preimages of $\pi_\U$ (i.e. $n$-marked rational tropical curves), counted with multiplicity. Here, we do not only care about the numbers $d_\U,$ but also about the set of preimages for all possible combinatorial settings of images in $(\M_{0,4}^{\mathrm{trop}})^{n-3}.$

Progress in the counting problem was made by Goldner in \cite{Go21b} and \cite{Go21}, where a recursive algorithm was described that provides the degree for a given set of cross-ratios. With this algorithm, we are now able to find the solution for any explicit set of cross-ratios but still there is no general formula for obtaining this degree using just the combinatorics behind the selected markings (i.e. which markings appear together in which of the selected cross-ratios). To find such a general formula for cross-ratio degrees is also mentioned as an open problem in \cite{Go21} and \cite{Go20}. Then, in \cite{R24} by Silversmith, such a general formula was found for a specific choice of cross-ratios given by triangulations, see \ref{defthreeint}, as follows. The edges of an $n$-gon are marked with $1,...,n$ and we consider a triangulation $T$ of it. A triangulation is formed by $n-3$ non-intersecting diagonals of the $n$-gon. Each diagonal defines a four-tuple by taking the markings of the two pairs of edges adjacent to the vertices where the diagonal is ending. We let $\U(T)$ be the set of these four-tuples. In this setting, the degree $D_{\U(T)}$ of these $n-3$ four-tuples equals $2^d,$ where $d$ is the number of inner triangles of this triangulation, see Theorem 3.1 of \cite{R24} and Theorem \ref{allcases}. 

This paper presents an alternative approach to prove this result using tropical, instead of algebraic, methods. Additionally, we describe how all tropical curves that fulfill these cross-ratio conditions look like and what their multiplicities are, which is also mentioned as interesting question in \cite{R24}. 
\begin{theorem}[see Theorem \ref{allcases}]\label{mainint}
	Let $T$ be a triangulation of an $n$-gon with $d$ inner triangles and $n-3$ diagonals that define cross-ratio conditions $C_1,...,C_{n-3},$
	the degree $d_\U$ of the map $\pi_\U$ of Definition \ref{mappi} is $2^d,$ and we can also construct explicitly all tropical curves in $\pi_{\U(T)}^{-1}(C_1,...,C_{n-3})$ and find their multiplicities.
\end{theorem}

In the same paper \cite{R24}, Silversmith also introduced the open problem of finding the maximal degree that a set of cross-ratios using $n$ markings can have. There, this number is given for $n\leq6.$ Recently, this maximal degree was also found for $n=7$ and $n=8$ in \cite{Ma26}. Here, we list all possible degrees for any $\U,$ and thus, also the maximal degree, for $n=9$ and $n=10$ using a computer program based on Goldners algorithm, written in OSCAR \cite{OSCAR}.
\begin{theorem}[see Theorem \ref{max9}]
	For a set of 6 cross-ratios using the markings in $\{1,...,9\},$ their cross-ratio degree can be every number in $\{0,1,2,3,4,5,6\}$ and for a set of 7 cross-ratios using the markings in $\{1,...,10\},$ their cross-ratio degree can be every number in $\{0,1,2,3,4,5,6,7,8,9,10\}.$
\end{theorem}

For $n=12,$ we make use of Corollary 4.4 in \cite{BELL25}, where an upper bound for this maximal degree in some cases was shown. The remaining cases can now also be calculated with a computer program such that we can also provide the maximal degree for $n=12.$
\begin{theorem}[see Theorem \ref{n12}]
	For a set of 9 cross-ratios using the markings in $\{1,...,12\}$, their cross-ratio degree can be at most 32.
\end{theorem}

Cross-ratios also appear in physics, in relation to the positive geometry of $\M_{0,n}$ and scattering amplitudes, see \cite{ABL17}, \cite{Brown09}, \cite{L22}, \cite{AHHL23}, \cite{CSY14}, \cite{Br21}, \cite{BBBDF15} and \cite{DG14}. Also, cross-ratio degrees are also a special case of Kapranov degrees which are introduced in \cite{BELL25}. Other similar counting problems are found in \cite{CI24}, \cite{Eb25} and \cite{Go22}.

Connected to this counting problem is also the problem of finding the number of embeddings of a Laman graph in the sphere. A triangulation of an $2n$-gon can viewed as a graph where all the diagonals and edges form the edges of the graph and their adjacent corners the vertices. A Laman graph is a minimally rigid graph, for a precise definition see \cite{GGS20}. It is easy to check that graphs given by triangulations are also Laman graphs. 

In \cite{GGS20}, it is described how a set of $2n-3$ cross-ratios with markings in $\U=\{S_1,...,S_{2n-3}\}$ can be assigned to such a Laman graph and that the degree of the corresponding map $\pi_{\U}$, see Definition \ref{mappi} is equal to the number of ways of embedding a Laman graph in the sphere. The connection of Laman graphs and constructing tropical curves that fulfill the cross-ratio conditions defined by a Laman graph is work of an ongoing project.

\subsection{Structure of the Paper}
In the Preliminaries \ref{prelim}, we introduce the counting problem with some basic notations. Section \ref{general} provides general lemmas that are later used to prove the main theorem. Then, the main theorem is proved in Section \ref{2ndint} after considering a special case that hopefully helps in understanding the proof methods. In Section \ref{max}, we provide a new result concerning the problem of finding maximal cross-ratio degrees. A table of case distinctions \ref{cases} and another section \ref{1stint} about a different approach to the main theorem are found in the appendix.

\subsection{Acknowledgments}
This work is supported by the DFG project ID 286237555, TRR 195. I want to thank Daniele Agostini, Shelby Cox, Matt Larson, Hannah Markwig and Tobias Schnieders for proofreading and inspiring discussions.

\section{Preliminaries}\label{prelim}
In section \ref{21}, we introduce tropical curves and in section \ref{23} their moduli spaces and tropical cross-ratios. These moduli spaces naturally have the structure of tropical fans, which we introduce before, in section \ref{22}. Then, in section \ref{24}, the central problem of this paper, called the counting problem is defined.  Finally, section \ref{25} explains how we can obtain cross-ratios from a triangulation of an $n$-gon, which is the main case that we consider in this paper.
\subsection{Abstract Rational Tropical Curves}\label{21}
\begin{definition}[Abstract rational tropical curve]
	An \emph{abstract rational tropical curve} is a connected graph $\Gamma$ that is a tree without two-valent vertices, where the set of the vertices and the set of the edges of $\Gamma$ are denoted by $V(\Gamma)$ and $E(\Gamma)$, respectively, and that satisfies the following properties.
	\begin{itemize}
		\item $E(\Gamma)$ is equipped with a length function $l:E(\Gamma)\to \R_{> 0} \cup \{\infty\}$, such that all edges with $l(e)=\infty$ are exactly the ones adjacent to a one-valent vertex. These edges are called \emph{ends}. The other edges $e$, for which it then holds that $l(e)<\infty$ are called \emph{bounded edges}.
		\item Vertices that are adjacent to a bounded edge are called \emph{inner vertices}.
	\end{itemize}
	
	An $n$-\textit{marked abstract rational tropical curve} $(\Gamma,x_{[n]})$ is an abstract rational tropical curve that has exactly $n$ ends that are marked with pairwise different $x_1,\dots,x_n\in\mathbb{N}$. Two $n$-marked tropical curves $(\Gamma,x_{[n]})$ and $(\tilde{\Gamma},\tilde{x}_{[n]})$ are isomorphic if there is a homeomorphism $\Gamma\to \tilde{\Gamma}$ mapping $x_i$ to $\tilde{x}_i$ for all $i$ and each edge of $\Gamma$ is mapped onto an edge of $\tilde{\Gamma}$ by an affine linear map of slope $\pm 1$. Often, we also denote it as $\Gamma$ when the markings are clear from the context.
	
	Forgetting all lengths of an $n$-marked abstract rational tropical curve gives us its \textit{combinatorial type}.
\end{definition}
\begin{remark}
	Because we are not working with other tropical curves in this paper, we will refer to $n$-marked abstract rational tropical curves as tropical curves if the number $n$ is already known or irrelevant for this case.
\end{remark}
\begin{example}\label{ex4}
	For a $4$-marked tropical curve, there are only four combinatorial types, see Figure \ref{m04}.
	\begin{figure}
		\centering
		\includegraphics[scale=0.5]{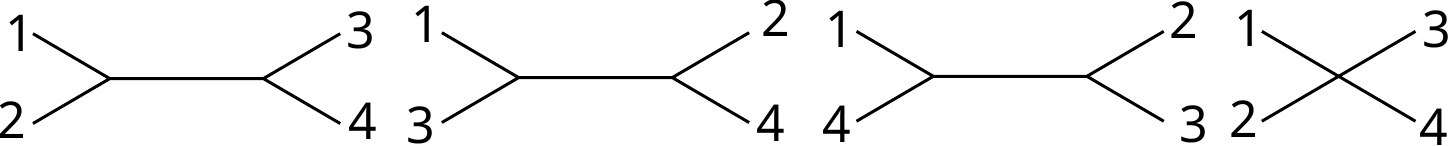}
		\caption{All combinatorial types of tropical curves with the markings 1, 2, 3 and 4.}
		\label{m04}
	\end{figure}
\end{example}
\subsection{Tropical Fans and Morphisms}\label{22}
We now introduce further concepts of tropical geometry that we will need in this paper.  More on these topics can be found in section 2 of \cite{GKM09}. 

\begin{definition}[Fan]
	A \emph{fan} $X$ in $\R^m$ is a nonempty finite set of cones in $\R^m$ such that different cones intersect each other along faces and all faces of the cones are also part of the fan.
	
	The \emph{dimension} of a fan is the largest dimension of a cone in it. A fan is called \emph{pure dimensional} of dimension $n$ if all \emph{maximal cones}, that are cones that are not faces of any other cones, are of dimension $n$.
	
	The set of all $k$-dimensional cones of $X$ is denoted as $X^{(k)}.$
	
	The \emph{support} $|X|$ is the union of all cones.
\end{definition}

\begin{remark}[Lattice]\label{lattice}
	In tropical geometry, it is necessary to keep track of integral structure. For that reason, we need to fix a lattice inside the vector spaces we work with. We will usually work in coordinates and fix $\mathbb{Z}^m$ inside $\R^m.$ For a cone $\sigma,$ its lattice is the intersection $\mathrm{Span}(\sigma)\cap \mathbb{Z}^m$ with the lattice $\mathbb{Z}^m.$
\end{remark}
\begin{definition}[Normal vector]
	For a cone $\sigma$ of dimension $n$ and a face $\tau$ of it of dimension $n-1,$ we define the \emph{normal vector $u_{\sigma/\tau}$ of $\sigma$ relative to $\tau$}  to be a vector in $\R^m$ that descends to the unique generator of the lattice $\mathrm{Span}(\sigma)/\mathrm{Span}(\tau).$ Note, that there are infinitely many choices of $u_{\sigma/\tau}$ in $\R^m$ but they all descend to the same vector in $\mathrm{Span}(\sigma)/\mathrm{Span}(\tau).$
\end{definition}
\begin{definition}[Tropical fan]\label{deftropfan}
	A \emph{tropical fan} is a pair $(X,\omega_X),$ where $X$ in $\R^m$ is a fan of pure dimension $n$ and $\omega_X:X^{(n)}\to \mathbb{Z}_{\ge0}$ is a weight function that assigns a weight $\omega_X(\sigma)$ to each maximal cone $\sigma$ such that for each cone $\tau$ of dimension $n-1,$ the following \emph{balancing condition} holds:
	$$\sum \omega_X(\sigma) u_{\sigma/\tau}=0,\text{ in $\R^m /\mathrm{Span}(\tau)$},$$
	where the sum goes over all maximal cones $\sigma$ of which $\tau$ is a face.
	
	Often, a tropical fan is just denoted as $X$, when the weight function is clear from the context. A tropical fan $X$ is called \emph{irreducible} if there is no other tropical fan $Y$ of the same dimension with $|Y|\subsetneq |X|.$
\end{definition}
\begin{example}\label{extropfan}
	The tropical line consisting of the three rays from the origin generated by the vectors $\binom{1}{1},$ $\binom{-1}{0}$ and $\binom{0}{-1}$ with weight 1 each form a tropical fan, see Figure \ref{embed}. This fan is also irreducible.
	
	Our main examples of tropical fans are the moduli spaces of $n$-marked tropical curves, see later in Definition \ref{defmod} and Remark \ref{M0nisfan}.
\end{example}
\begin{definition}[Product of tropical fans]\label{defprod}
	For two tropical fans $X$ in $\R^m$ and $Y$ in $\R^{m'}$ we can define their \emph{product}. The cones are all products $\sigma \times \sigma'$ in $\R^{m+m'}$ of cones $\sigma \in X$ and $\sigma'\in Y$ and for the maximal cones $\sigma \times \sigma',$ their weight is the product of the weights of $\sigma$ and $\sigma'$.
\end{definition}
\begin{remark}
	It can easily be checked that the product of tropical fans is a tropical fan. 
\end{remark}
\begin{definition}[Morphism of fans]\label{defmorph}
	For two fans $X$ in $\R^m$ and $Y$ in $\R^{m'}$, a \emph{morphism} from $X$ to $Y$ is a map $f:X\to Y$ that is induced by a $\mathbb{Z}$-linear map of the surrounding spaces. A morphism of tropical fans is defined in the same way, there are no additional conditions on the weights.
\end{definition}

Using the following Proposition, we obtain that for some morphisms of tropical fans, the degree, which is the number of preimages of a general point, counted with multiplicity is finite, well defined and thus, does not depend on the choice of the point. The multiplicity arises from the integral structure, see Remark \ref{lattice}, and the weights of the cones.

This will later be used for counting tropical curves that fulfill some given cross-ratio conditions.
\begin{proposition}See \cite[Corollary 2.26]{GKM09}.\label{2.26}
	Let $X$ and $Y$ be tropical fans of the same dimension $n\ge1$ and let $ f: X
	\to Y $ be a morphism. Assume that $Y$ is irreducible. Then there is a fan
	$ Y_0 $ of smaller dimension with $ |Y_0| \subset |Y| $ such that
	\begin {enumerate}
	\item 
	each point $ Q \in |Y| \backslash |Y_0| $ lies in the interior of a
	cone $ \sigma_Q' \in Y $ of dimension $n$;
	\item
	each point $ P \in f^{-1} (|Y| \backslash |Y_0|) $ lies in the interior
	of a cone $ \sigma_P \in X $ of dimension $n$;
	\item 
	for $ Q \in |Y| \backslash |Y_0| $ the sum
	\[ \sum_{P \in |X|: f(P)=Q} \mult_P f \]
	is called the \emph{degree} of $f,$ and does not depend on $Q$, where the multiplicity $ \mult_P f $ of $f$ at $P$
	is defined to be
	\[ \mult_P f := \frac {\omega_X(\sigma_P)}{\omega_Y(\sigma'_Q)}
	\cdot |\Lambda'_{\sigma'_Q}/f(\Lambda_{\sigma_P})|, \]
	\end {enumerate}
	where $|\Lambda'_{\sigma'_Q}/f(\Lambda_{\sigma_P})|$ denotes the index of the image of the lattice of the cone $\sigma_P$ inside the lattice of $\sigma'_Q,$ see also Construction 2.24 of \cite{GKM09}. 
\end{proposition}
\begin{remark}
	In the cases relevant to this paper, the multiplicity of Proposition \ref{2.26} can also be calculated differently, as we will see in Remark \ref{mult}.
\end{remark}
\subsection{Tropical Moduli Spaces}\label{23}

	All tropical curves of a specific combinatorial type with $k$ bounded edges can be parameterized by a cone $\R_{>0}^k$, where each coordinate represents one of the bounded edges. A point $(\ell_1,...,\ell_k)$ in $\R_{>0}^k$ corresponds to the tropical curve with length $\ell_i$ for the corresponding bounded edge $e_i.$
	
	We now consider the case, where one of the bounded edges shrinks, so that in the end, the length is zero. This is now a tropical curve of another combinatorial type. There are also other combinatorial types of tropical curves that produce this one when shrinking an edge. 
	
	For example, in Figure \ref{m04} and Example \ref{ex4}, if we shrink the bounded edge of any of the first three combinatorial types, we obtain the tropical curve of the last combinatorial type. So, the three open cones isomorphic to $\R_{>0}$ can be glued together at the point corresponding to zero to obtain a space parametrizing all $4$-marked tropical curves.
	
	Analogously, we can also glue cones together along their boundary, where at least one coordinate is 0 to obtain a parameter space, called the \emph{moduli space}, of all $n$-marked tropical curves. With this, the moduli space of $n$-marked tropical curves obtains the structure of an abstract cone complex. In Remark \ref{M0nisfan}, we will see that it is also a tropical fan.
	
	An $n$-marked tropical curve has the most bounded edges if all inner vertices are three-valent. Then, there are $n-3$ bounded edges, so the moduli space of $n$-marked tropical curves is of dimension $n-3.$

\begin{definition}[Moduli space] 	\label{defmod}
	The space $\M_{0,S}^{\mathrm{trop}}$ is called the moduli space of abstract tropical curves, where $S$ is the set of the markings of the ends of the tropical curve and it parameterizes all these tropical curves up to isomorphism. For $S=\{1,...,n\},$ we write it as $\M_{0,n}^{\mathrm{trop}}.$
\end{definition}
\begin{remark}[$\M_{0,n}^{\mathrm{trop}}$ is a tropical fan]\label{M0nisfan}
	By renaming the markings, every moduli space is isomorphic to some $\M_{0,n}^{\mathrm{trop}}.$
	 Moduli spaces of tropical curves were studied regarding different aspects and applications like phylogenetic trees or compactifications in
	 \cite{Mi07}, \cite{BHV01}, \cite{GM10}, \cite{SS04}, and \cite{Te07}.
	  In \cite[Theorem 3.7]{GKM09}, it is shown that the moduli space $\M_{0,n}^{\mathrm{trop}}$ can be made into a tropical fan as defined in Definition \ref{deftropfan}, embedded into a $\R^t$ for some $t\in\mathbb{N},$ of pure dimension $n-3$, where each of its cones corresponds to the set of all abstract rational tropical curves of a specific combinatorial type and all its weights are 1.
\end{remark}

\begin{example}\label{exM04}	
	The moduli space $\M_{0,4}^{\mathrm{trop}}$ can be embedded into $\R^2$ as depicted in Figure \ref{embed}, with the direction vectors of the rays being $\binom{1}{1},$ $\binom{-1}{0}$ and $\binom{0}{-1}$ and we can easily see that it is indeed an irreducible tropical fan when taking weight 1 for each of the three maximal cones, see Example \ref{extropfan}.
	\begin{figure}
		\centering
		\includegraphics[scale=0.4]{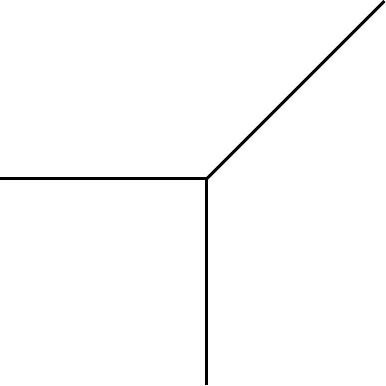}
		\caption{Embedding of the moduli space $\M_{0,4}^{\mathrm{trop}}$ into $\R^2,$ see Example \ref{exM04}}
		\label{embed}
	\end{figure}
\end{example}
\begin{definition}[Tropical cross-ratio]\label{defCR}
 A \emph{tropical cross-ratio} $C=(\lambda,\cros{a}{b}{c}{d}),$ is a tuple of an element $\lambda\in\R_{\ge0}$ and two distinct tuples $(a,b)$ and $(c,d)$ of markings. 
	
	Let $\Gamma$ be an $n$-marked abstract rational tropical curve that has ends marked with $a,$ $b,$ $c,$ and $d.$ We say $\Gamma$ \emph{fulfills} a cross-ratio condition  $C=(\lambda,\cros{a}{b}{c}{d})$ if the unique paths from the end $a$ to the end $c$ and the one from the end $b$ to the end $d$ meet in at least one vertex and the sum of the lengths of the edges they have in common is $\lambda,$ see Figure \ref{CR}.
	
	We say that an edge \emph{contributes} to a cross-ratio if it is one of the edges, that those two paths have in common.
		\begin{figure}
		\centering
		\includegraphics[scale=0.4]{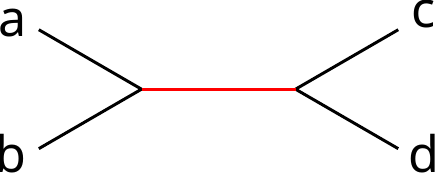}
		\caption{The pictures shows a 4-marked tropical curve $\Gamma.$ We assume that the length of the bounded edge is $\lambda.$ Then $\Gamma$ fulfills the cross-ratio $C=(\lambda, \cros{a}{b}{c}{d}),$ see Definition \ref{defCR}}
		\label{CR}
	\end{figure}
\end{definition}
\begin{example}\label{excr}
	Consider the 7-marked tropical curve $\Gamma$ depicted in the left of Figure \ref{excrim}. This tropical curve fulfills, among others, the cross-ratio conditions $C_1=(\ell_1+\ell_2, \cros{1}{2}{3}{4}),$ $C_2=(\ell_4, \cros{2}{5}{6}{7}),$ and $C_3=(\ell_1+\ell_3, \cros{1}{2}{5}{7}).$ 

\end{example}
\begin{remark}[Tropical cross-ratio]\label{remcr}

	The previous definition implies that for $\lambda>0$, we have a set of consecutive edges $e_1,...,e_k$ that are the common part of the paths from $a$ to $c$ and the one from $b$ to $d.$	The path from $a$ to $b$ meets these edges at a vertex adjacent to $e_1$ and the path from $c$ to $d$ meets them at a vertex adjacent to $e_k.$ Thus, we can see that $\cros{a}{b}{c}{d}=\cros{a}{b}{d}{c}=\cros{b}{a}{c}{d}=\cros{b}{a}{d}{c}=\cros{c}{d}{a}{b}=\cros{c}{d}{b}{a}=\cros{d}{c}{a}{b}=\cros{d}{c}{b}{a}.$ So, for a choice of four markings, there are three different possibilities of tropical cross-ratios. We can also interpret these three possibilities as the three different open cones of $\M_{0,4}^{\mathrm{trop}},$ as each cone corresponds to a combinatorial type and these three possibilities define different combinatorial types regarding the ends $a,$ $b,$ $c$ and $d,$ see the following Lemma \ref{crasft}.
\end{remark}
\begin{definition}[Forgetful map]\label{forget}
   Let $n \in \mathbb{N}$ with $n\ge4$. For every $i\in[n]$ we have a \emph{forgetful map} $$ft_i:\M_{0,n}^{\mathrm{trop}}\to \M_{0, [n]\setminus\{i\}}^{\mathrm{trop}},$$ where an $n$-marked tropical curve is mapped to an $n-1$-marked tropical curve by removing the end marked with $i$ and stabilizing the resulting tropical curve. This means that if a two-valent vertex is created when removing that end, the two adjacent edges get merged, and their lengths are added up if both are bounded edges and if one of them is an end, this new edge is assigned to be an end and the other length is omitted.
   
   \begin{figure}
   	\centering
   	\includegraphics[scale=0.5]{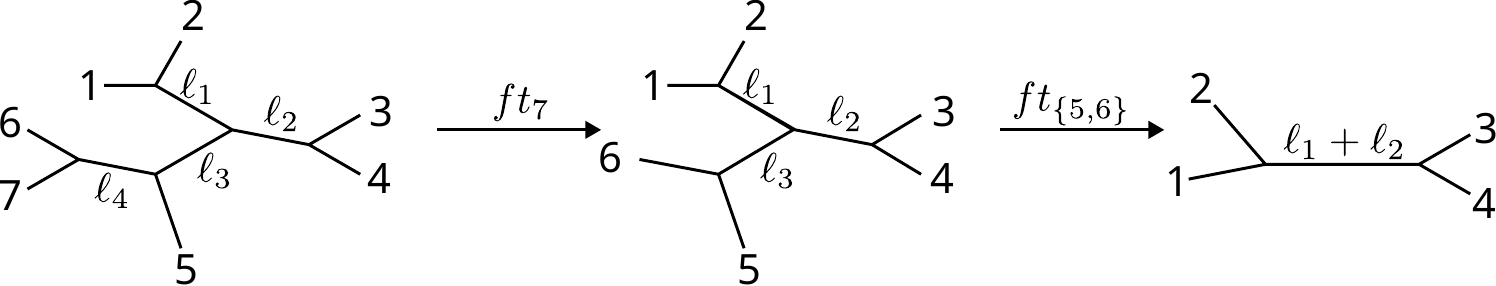}
   	\caption{Here, we see a tropical curve $\Gamma$ and how $\Gamma$ gets transformed under some forgetful morphisms, see Definition  \ref{forget} as well as Example \ref{excr} and \ref{excr2}}
   	\label{excrim}
   \end{figure}
\end{definition}
\begin{proposition}
	The forgetful map of Definition \ref{forget} is a morphism of tropical fans as in Definition \ref{defmorph} with the embedding of $\M_{0,n}^{\mathrm{trop}}$ mentioned in Remark \ref{M0nisfan}.
\end{proposition}
\begin{proof}
	A proof is found in Proposition 3.12 of \cite{GKM09}.
\end{proof}
\begin{remark}\label{remforget}
	By concatenating several forgetful maps, we obtain maps from $\M_{0,n}^{\mathrm{trop}}$ to $\M_{0,k}^{\mathrm{trop}}$ for all $n>k\ge3.$ In particular, we can obtain a map $$ft_{[n]\setminus S}:\M_{0,n}^{\mathrm{trop}}\to \M_{0, S}^{\mathrm{trop}}$$ for all $S\subset[n]$ with $|S|\geq3.$
	As compositions of morphisms of tropical fans are again morphisms of tropical fans, $ft_{[n]\setminus S}$ is one as well.
\end{remark}
\begin{lemma}\label{crasft}
	Given the map 
	$$ft_{[n]\setminus\{a,b,c,d\}}:\M_{0,n}^{\mathrm{trop}}\to\M_{0,\{a,b,c,d\}}^{\mathrm{trop}},$$
	which is a forgetful map that forgets all ends but $\{a,b,c,d\}.$ Then,  a tropical curve fulfills the cross-ratio $C=(\lambda,\cros{a}{b}{c}{d})$ exactly if this tropical curve is in the preimage under the map $ft_{[n]\setminus\{a,b,c,d\}}$ of the 4-marked tropical curve, where the bounded edge has length $\lambda$ (which can also be 0) and the ends $a$ and $b$ share a vertex. 
\end{lemma}
\begin{proof}
	Given a point $\Gamma$ in $\M_{0,n}^{\mathrm{trop}}$ that corresponds to a tropical curve fulfilling this cross-ratio $C,$ then applying the forgetful map does not change the lengths of the common edges in the paths from $a$ to $c$ and from $b$ to $d$.
	As also the ends $a,$ $b,$ $c$ and $d$ are preserved, $ft_{[n]\setminus\{a,b,c,d\}}(\Gamma)$ corresponds to the 4-marked tropical curve with $a$ and $b$ adjacent to the same vertex and with length of the bounded edge equal to $\lambda.$ 
	
	Now, we take a tropical curve that corresponds to a point in the preimage of the given point. We can obtain it by adding the former ends to the ends or the bounded edges of the tropical curve corresponding to the given point in $\M_{0,\{a,b,c,d\}}^{\mathrm{trop}}$ again, which trivially fulfills $(\lambda,\cros{a}{b}{c}{d})$. When adding these ends, the cross-ratio condition is not affected and thus, all tropical curves in the preimage fulfill this cross-ratio condition.
\end{proof}
\begin{example}[Revisited Example \ref{excr}]\label{excr2}
	Again, we consider the tropical curve $\Gamma$ depicted in the left of Figure \ref{excrim} and the cross-ratio $C_1=(\ell_1+\ell_2, \cros{1}{2}{3}{4}).$  In the same Figure, we also see how the tropical curve gets transformed when applying the forgetful morphisms $ft_7$ and $ft_{\{5,6\}}.$ As $ft_7\circ ft_{\{5,6\}}=ft_{\{5,6,7\}},$ we also see that $\Gamma$ is in the preimage under the map $ft_{\{5,6,7\}}$ of the 4-marked tropical curve depicted on the right, as described in Lemma \ref{crasft} and thus, we see again that $\Gamma$ fulfills $C_1.$
\end{example}
\subsection{The Counting Problem}\label{24}
In this section, we define the counting problem, which is the central definition of this paper. Given a general set of cross-ratio conditions, the counting problem is about finding the number of tropical curves that fulfill all of them.
\begin{definition}\label{mappi}
	Let $\U=\{S_1,...,S_{n-3}\}$ be a set of four element subsets of $[n]$ such that there exists a maximal cone of $\M_{0,n}^{\mathrm{trop}}$ on which the following map is injective. The map is $$\pi_\U=ft_{[n]\setminus S_1}\times...\times ft_{[n]\setminus S_{n-3}}:\M_{0,n}^{\mathrm{trop}}\to \prod_{S\in\U}\M_{0,S}^{\mathrm{trop}},$$ 
	where $ft_{[n]\setminus S_i}$ is as described in \ref{remforget}.
	The degree of this map is denoted as $d_\U$ and is finite and well-defined, see the following remark. If for $\U$ there is no maximal cone of $\M_{0,n}^{\mathrm{trop}}$ on which $\pi_\U$ is injective, we set $d_\U=0.$ This degree is the solution to the counting problem.
\end{definition}
\begin{remark}\label{genpos}
	The map $\pi_{\U}$ fulfills all the conditions of Proposition \ref{2.26} as $\M_{0,n}^{\mathrm{trop}}$ and $\prod_{S\in\U}\M_{0,S}^{\mathrm{trop}},$ are fans of the same dimension and the latter is an irreducible fan as a product of irreducible fans, see Definition \ref{defprod} and Proposition 2.20 of \cite{GKM09}. Thus, the degree of this map is well-defined and we can define a \emph{multiplicity} of an $n$-marked abstract rational tropical curve fulfilling the cross-ratio conditions $\{C_1,...,C_{n-3}\}$ that use the markings $\U=\{S_1,...,S_{n-3}\},$ respectively, as in Proposition \ref{2.26} applied to the map $\pi_{\U}.$
	
	Furthermore, we say that a set of cross-ratios $\{C_1,...,C_{n-3}\}$ with markings taken from $\U$ is in general position if every preimage under the map $\pi_\U$ is in the interior of a maximal cone of $\M_{0,n}^{\mathrm{trop}}.$  In other words, this means that the degree $d_\U$ does not depend on the exact lengths of the cross-ratios or which of the four markings are grouped together as long as they are in general position.
	
	Thus, what is described as the fan $Y_0$ in Proposition \ref{2.26}, are all the cross-ratios $\{C_1,...,C_{n-3}\}$ in non-general position in the setting of the map $\pi_{\U}.$ For the four markings of the sets $S_i$, the different ways how to divide them into two tuples for the cross-ratios $C_i$ correspond to different maximal cones of the moduli space $\prod_{S\in\U} \M_{0,S}.$
	
	As cross-ratios in general position must have their preimage in a maximal cone, all tropical curves described by such cross-ratios have to have only three-valent inner vertices.
	
	If we consider some inequalities for the lengths of the cross-ratios such as
	$\lambda_3>\lambda_1+\lambda_2$, it is always possible to find cross-ratios in general position satisfying these inequalities because all the ones in non-general position are of lower dimension in $Y_0,$ whereas such inequalities define a top-dimensional cone in ${(\M_{0,4}^{\mathrm{trop}})}^{n-3}.$	
\end{remark}

\begin{remark}[Multiplicity of a tropical curve]\label{mult}
	 We consider the cross-ratio conditions $C_1,...,C_{n-3}$ and the corresponding map $\pi_{\U}$ as in Definition \ref{mappi}.  In Corollary 2.27 of \cite{GKM09}, it is shown that the multiplicity of an $n$-marked tropical curve with respect to the map $\pi_{\U}$ can also be computed directly using the local coordinates of a cone. This is easier to calculate and is described now.
	This \emph{multiplicity} equals the absolute value of the determinant of the following $(n-3)\times (n-3)$-matrix, called \emph{multiplicity matrix}. The columns correspond to the bounded edges $e_1,...,e_{n-3}$ of the tropical curve and the rows correspond to the  cross-ratios $C_1,...,C_{n-3}$. The entries $a_{i,j}$ of the matrix are 1, when the edge $e_j$ contributes to the cross-ratio $C_i$ and 0 elsewhere.
	
	This matrix is the expression of the map $\pi_\U$ in local coordinates and is unique up to permutation of the rows and columns, but as permutations only change the sign of a determinant, this does not affect the multiplicity.
\end{remark}
\begin{example}

Consider the tropical curve depicted in the middle of Figure \ref{excrim} with the cross-ratios $C_1=(\ell_1+\ell_2,\cros{1}{2}{3}{4}),$ $C_2=(\ell_2+\ell_3,\cros{3}{4}{5}{6})$ and $C_3=(\ell_2,\cros{2}{5}{3}{4}).$ The corresponding multiplicity matrix is given by 
$$\begin{blockarray}{cccc}
	e_1& e_2&e_3 &  \\
	\begin{block}{(ccc)c}
		1 & 1& 0 & C_1 \\
		0 & 1 & 1 & C_2 \\
		0 & 1 & 0  &C_3  \\
	\end{block}
\end{blockarray},$$ where $e_i$ is the bounded edge with length $\ell_i.$
The absolute value of this determinant is 1, so the multiplicity of this tropical curve with these cross-ratios is 1.

\end{example}

\begin{example}\label{exnongen}
	Let $\U=\{\{3,4,5,6\},\{1,3,4,5\},\{2,4,5,6\}\}.$ Consider the 6-marked tropical curves depicted in Figure \ref{nongeneral} and the three cross-ratios $C_1=(\lambda_1,\cros{3}{4}{5}{6})$, $C_2=(\lambda_2,\cros{1}{3}{4}{5})$ and $C_3=(\lambda_3,\cros{2}{4}{5}{6})$ and consider the map $$\pi_\U:\M_{0,6}^{\mathrm{trop}}\to \M_{0,\{3,4,5,6\}}^{\mathrm{trop}}\times\M_{0,\{1,3,4,5\}}^{\mathrm{trop}}\times\M_{0,\{2,4,5,6\}}^{\mathrm{trop}}.$$
	
	For $\lambda_1=2,$ $\lambda_2=\lambda_3=1,$ we can check that the only tropical curve fulfilling these conditions is depicted in the right of Figure \ref{nongeneral} with all bounded edges having length 1. Thus, the cross-ratios with these lengths are in general position.
	For $\lambda_1=\lambda_2=\lambda_3=0,$ the tropical curve on the left fulfills these cross-ratios for any length $\ell$ of the bounded edge, so in the preimage of the point in $\M_{0,\{3,4,5,6\}}^{\mathrm{trop}}\times\M_{0,\{1,3,4,5\}}^{\mathrm{trop}}\times\M_{0,\{2,4,5,6\}}^{\mathrm{trop}}$ corresponding to these cross-ratio conditions there are infinitely many tropical curves. This happens because these cross-ratios are in non-general position.
	\begin{figure}
		\centering
		\includegraphics[scale=0.4]{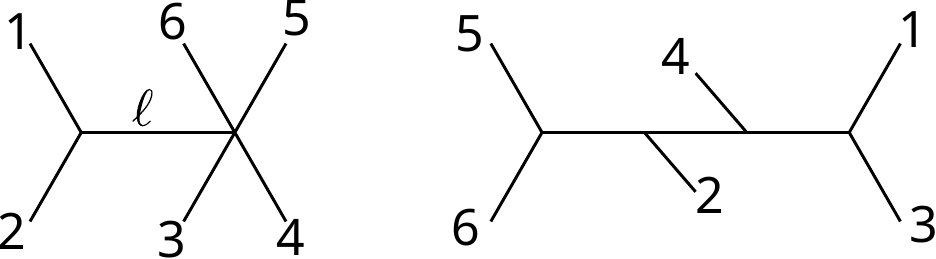}
		\caption{Two tropical curves that fulfill  the same cross-ratios but with different lengths, see Example \ref{exnongen}}
		\label{nongeneral}
	\end{figure}
\end{example}

\begin{definition}
Let $\Gamma$ be a tropical curve. Because the underlying graph of $\Gamma$ is a connected tree, for a $k$-valent vertex $v$ we can partition all edges of the graph into $k$ sets, according to which of the $k$ adjacent edges of $v$ is part of the path from $v$ to them. We say that these  sets are different \emph{branches} of $v$. 

\end{definition}

\subsection{Triangulations and Cross-Ratios} \label{25}
Let $T$ be a triangulation of an $n$-gon. In the introduction of \cite{R24} it is defined how we can associate a set $\U(T)$ of the markings of $n-3$ cross-ratios to it. It is found in \cite{R24} that using such sets $\U(T),$ the solution to the counting problem $d_{\U(T)}$ can then be determined using just the combinatorics of the triangulation with methods from algebraic geometry. In this paper, we enhance this definition by not only associating a set $\U(T)$ but also a set of $n-3$ concrete tropical cross-ratios to a triangulation (so, we also consider a length and the various choices of dividing a four-tuple). Our main result Theorem \ref{allcases}, see also in the introduction Theorem \ref{mainint}, describes the concrete set of tropical curves in the preimage of $d_{\U(T)}$ and their multiplicities for any choice of lengths and dividing the four-tuple of the cross-ratios.

\begin{definition}[Triangulations and cross-ratios] \label{defthreeint}
	Consider a regular $n$-gon with $n\ge4$ and edges marked clockwise as $1,...,n $ and a triangulation, denoted as $T,$ of it, where a length $\lambda\in\R_{\ge 0}$ is assigned to each diagonal of $T$. The triangles in this triangulation can either be adjacent to 0, 1 or 2 edges of the $n$-gon. If a triangle is adjacent to 0 edges, we call it an \emph{inner triangle}, if it is adjacent to 1, then we call it an \emph{outer triangle} and if it is adjacent to 2 edges, then we call it a \emph{border triangle}.
	
	Every diagonal of $T$ is adjacent to two vertices that are each adjacent to two edges of the $n$-gon. W.l.o.g, let us say, these four edges are $k,$ $k+1,$ $l$ and $l+1$. These four markings together with the length of the diagonal define a cross-ratio.
	For this cross-ratio we have three different choices of dividing these four markings into two tuples. We say that the \emph{dual  interpretation} of the cross-ratio is $(\lambda,(k,l+1|k+1,l)) $, the \emph{neighboring one} is $(\lambda,(k,k+1|l,l+1)) $ and the \emph{intersecting one} is $(\lambda,(k,l|k+1,l+1)),$ see Figure \ref{threeint}, where on the top left, the considered part of the $n$-gon is depicted. By Remark \ref{remcr}, these three interpretations are well-defined.
	
		\begin{figure}
		\centering
		\includegraphics[scale=0.4]{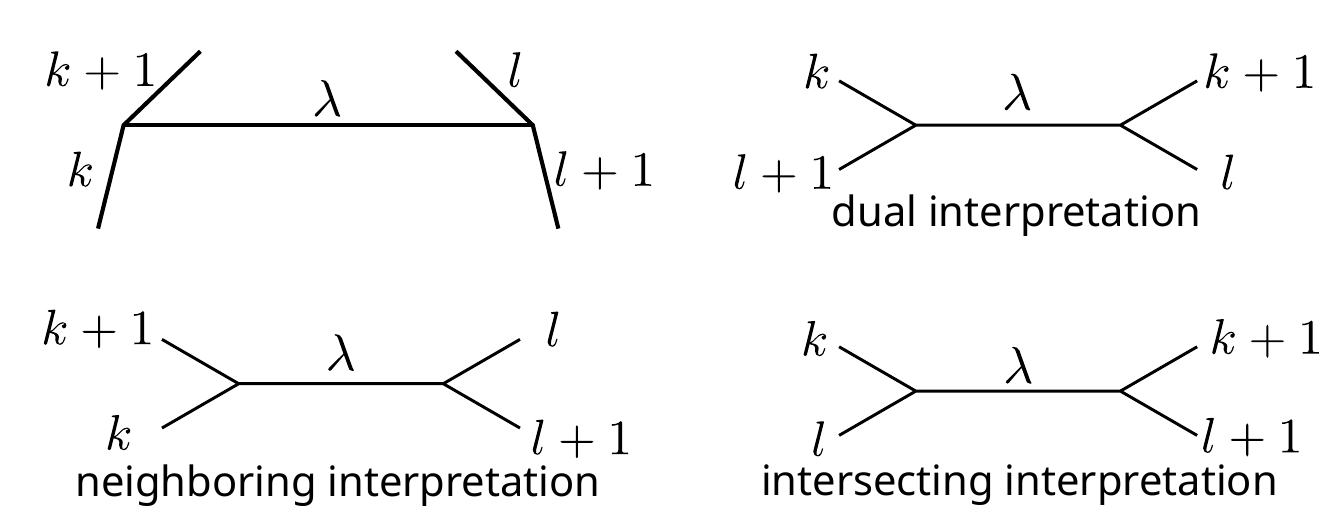}
		\caption{On the upper left, we see a sketch of a diagonal of the $n$-gon. Further, there are the three interpretations of a tropical cross-ratio, see Definition \ref{defthreeint}}
		\label{threeint}
	\end{figure}
	We denote the corresponding set of markings of the cross-ratios defined by $T$ as $\U(T).$
\end{definition}

\begin{remark}
	In section \ref{2ndint}, we will see that for cross-ratios defined by triangulations, the conditions for the map $\pi_{\U(T)}$ of Definition \ref{mappi} are always fulfilled, and thus, the counting problem to determine $d_{\U(T)}$ always leads to a finite number of tropical curves, counted with multiplicity, that fulfill the demanded conditions. Depending on the choices of dividing each set of 4 markings (i.e. interpreting a cross-ratio in the dual, neighboring or intersecting way) and the lengths of the cross-ratios, we obtain different (and also a different number of) tropical curves in $\pi_{\U(T)}^{-1}(C_1,...,C_{n-3}).$ In Theorem \ref{allcases}, we describe the tropical curves in the preimage of $d_{\U(T)}$ for any such choice.
\end{remark}
 In section \ref{2ndint}, starting with the case where all cross-ratios are interpreted in the neighboring way, we prove this Theorem \ref{allcases}, mentioned as Theorem \ref{mainint} in the introduction. Additionally, for the case where all cross-ratios are interpreted in the dual way, another method of obtaining the demanded tropical curves is presented in the appendix in section \ref{1stint}.

\section{A Reduction Step Used in the Proof of the Main Theorem}\label{general}
In this section, some general Lemmas are introduced that are later used in proving the main Theorem \ref{allcases}, mentioned as Theorem \ref{mainint} in the introduction.
Lemma \ref{trop25} is stated and proven algebraically in \cite[Lemma 2.5]{R24}. Here, we prove it using tropical methods and a part of it is outsourced to Lemma \ref{bijection}. 

\begin{lemma}\label{bijection}
	Let $\U=\{S_1,...,S_{n-3}\}$ be a set of four element subsets of $[n].$ Suppose there exists a partition $[n]=\{i_1,i_2,i_3\}\sqcup X\sqcup Y$ such that $S_i\subseteq\{i_1,i_2,i_3\}\cup X$ for all $i\in \{1,...,k\}$ and $S_i\subseteq\{i_1,i_2,i_3\}\cup Y$ for all $i\in\{k+1,...,n-3\}$. Define $\U_X=\{S\in\U:S\subseteq\{i_1,i_2,i_3\}\cup X\}$, and similarly $\U_Y$. Let $\{C_1,...,C_{n-3}\}$ be a set of cross-ratios in general position, where $C_i$ uses the markings of $S_i.$	
	Then, there is a bijection between tropical curves $\Gamma$ in $\pi_{\U}^{-1}(C_1,...,C_{n-3})$ and pairs of tropical curves $(\Gamma_X,\Gamma_Y)$ in $\pi_{\U_X}^{-1}\times\pi_{\U_Y}^{-1}((C_1,...,C_k),(C_{k+1},...,C_{n-3}))$ and the multiplicity $m$ of $\Gamma$ with respect to $\pi_\U$ is the product of the multiplicities $m_X$ and $m_Y$ of $\Gamma_X$ and $\Gamma_Y$ with respect to $\pi_{\U_X}$ and $\pi_{\U_Y},$ respectively, where $\pi_\U,$ $\pi_{\U_X}$ and $\pi_{\U_Y}$ are as in Definition \ref{mappi}.
\end{lemma}
\begin{proof}
	First, we show that each tropical curve $\Gamma$ in $\pi_{\U}^{-1}(C_1,...,C_{n-3})$ is in bijection with a pair of two tropical curves $(\Gamma_X,\Gamma_Y)$ in $\pi_{\U_X}^{-1}\times\pi_{\U_Y}^{-1}((C_1,...,C_k),(C_{k+1},...,C_{n-3})),$ where $\Gamma_X$ uses the markings of $\{i_1,i_2,i_3\}\cup X$ and $\Gamma_Y$ uses the markings of $\{i_1,i_2,i_3\} \cup Y.$ Afterwards, we consider the multiplicities.
	
	The idea of this proof is kept in the following commuting diagram. The diagram commutes because none of the maps affect the cross-ratios.
	\begin{figure}[h]
		\begin{tikzcd}
			\M_{0,n}^{\mathrm{trop}} \arrow[r, "ft_Y\times ft_X", leftrightarrow] \arrow[d,"\pi_{\U}"]
			& \M_{0,\{i_1,i_2,i_3\}\cup X}^{\mathrm{trop}}\times\M_{0,\{i_1,i_2,i_3\}\cup Y}^{\mathrm{trop}} \arrow[d, "\pi_{\U_X}\times\pi_{\U_Y}" ] \\
			\prod_{S\in\U}\M_{0,S}^{\mathrm{trop}} \arrow[r,equal]
			& |[]| \prod_{S\in\U_X}\M_{0,S}^{\mathrm{trop}}\times\prod_{S\in\U_Y}\M_{0,S}^{\mathrm{trop}}
		\end{tikzcd}
	\end{figure}
	
	For obtaining the bijection, we first take a tropical curve $\Gamma$ in $\pi_\U^{-1}(C_1,...,C_{n-3}).$ Now, we apply a forgetful morphism as in Definition \ref{forget} and Remark \ref{remforget}:
	$$ft_Y\times ft_X:\M_{0,n}^{\mathrm{trop}}\to\M_{0,\{i_1,i_2,i_3\}\cup X}^{\mathrm{trop}}\times\M_{0,\{i_1,i_2,i_3\}\cup Y}^{\mathrm{trop}}.$$
	Thus, a tropical curve $\Gamma$ gets mapped to a pair of tropical curves $(\Gamma_X,\Gamma_Y),$ where $\Gamma_X:=ft_Y(\Gamma)$ is a tropical curve in $\M_{0,\{i_1,i_2,i_3\}\cup X}^{\mathrm{trop}},$  and $\Gamma_Y:=ft_X(\Gamma)$ is a tropical curve in $\M_{0,\{i_1,i_2,i_3\}\cup Y}^{\mathrm{trop}}.$ Also, $\Gamma_X$ and $\Gamma_Y$ are in $\pi_{\U_X}^{-1}(C_1,...,C_k)$ and $\pi_{\U_Y}^{-1}(C_{k+1},...,C_{n-3})$, respectively as $ft_Y \times ft_X$ does not affect the cross-ratios whose markings are not forgotten.
	
	Now, we construct the inverse function to $ft_Y\times ft_X.$ Let $(\Gamma_X,\Gamma_Y)$ be in $\pi_{\U_X}^{-1}\times\pi_{\U_Y}^{-1}((C_1,...,C_k),(C_{k+1},...,C_{n-3})).$ For $A\in\{X,Y\},$ let $v_A$ be the unique vertex in $\Gamma_A$ that the three paths from $i_1$ to $i_2,$ the one from $i_2$ to $i_3$ and the one from $i_3$ to $i_1$ have in common and denote by $P_{j,A}$ the path from $v_A$ to the end $i_j.$ In the disjoint union of $\Gamma_X$ and $\Gamma_Y,$ identify the vertices $v_X=v_Y$ and the paths $P_{j,X}=P_{j,Y}$ for $j\in\{1,2,3\}.$ We obtain a tropical curve $\Gamma$ in $\M_{0,n}^{\mathrm{trop}},$ which is the preimage of $(\Gamma_X,\Gamma_Y)$ under $ft_X \times ft_Y.$ A sketch of this gluing of $\Gamma_X$ and $\Gamma_Y$ can be seen in Figure \ref{glueinggraphs}. This gluing is inverse to applying the forgetful morphisms as in the preceding paragraph, because the tropical curve $\Gamma$ gets reconstructed, and thus, $ft_Y\times ft_X$ is indeed a bijection.

	Now, we want to show that the multiplicity $m$ of $\Gamma$ with respect to $\pi_{\U}$ is the same as the product of the multiplicities $m_X$ and $m_Y$ of $\Gamma_X$ and $\Gamma_Y$ with respect to $\pi_{\U_X}$ and $\pi_{\U_Y},$ where $\pi_\U,$ $\pi_{\U_X}$ and $\pi_{\U_Y}$ are as in Definition \ref{mappi}. 
	
	We observe that the bounded edges of $\Gamma$ that are on the path from $v_A$ to an end $i_j$ can contribute to cross-ratios in $\U_X$ and $\U_Y. $ Every other bounded edge can only contribute to either $\U_X$ or $\U_Y.$ For a bounded edge $e$ that contributes to cross-ratios in both of  $\U_X$ and $\U_Y,$ we consider the adjacent vertex $v'$, that is further away from $v_A.$ Out of the two other edges (which also can be ends) adjacent to this vertex $v',$ one of them is also on the path from $v_A$ to $i_j.$ We call this edge $e'.$ The other edge $e''$ either just contributes to cross-ratios in $\U_X$ or in $\U_Y$ or is an end with a marking in $X$ or $Y$. W.l.o.g, we say that either the edge contributes to cross-ratios in $\U_X$ or is an end in $X.$ A sketch can be seen in Figure \ref{glueinggraphs}. Thus, for any cross-ratio in $\U_Y,$ either both or none of $e$ and $e'$ contribute to it. In the tropical curve $\Gamma_X,$ we also have the edge $e''.$ The vertex adjacent to $e''$ that is nearer to $v_X$ is adjacent to two further edges. We call them $\tilde{e}$ and $\tilde{e}',$ where $\tilde{e}$ is the one closer to $v_X,$ see Figure \ref{glueinggraphs}.
	
	\begin{figure}
		\centering
		\includegraphics[scale=0.4]{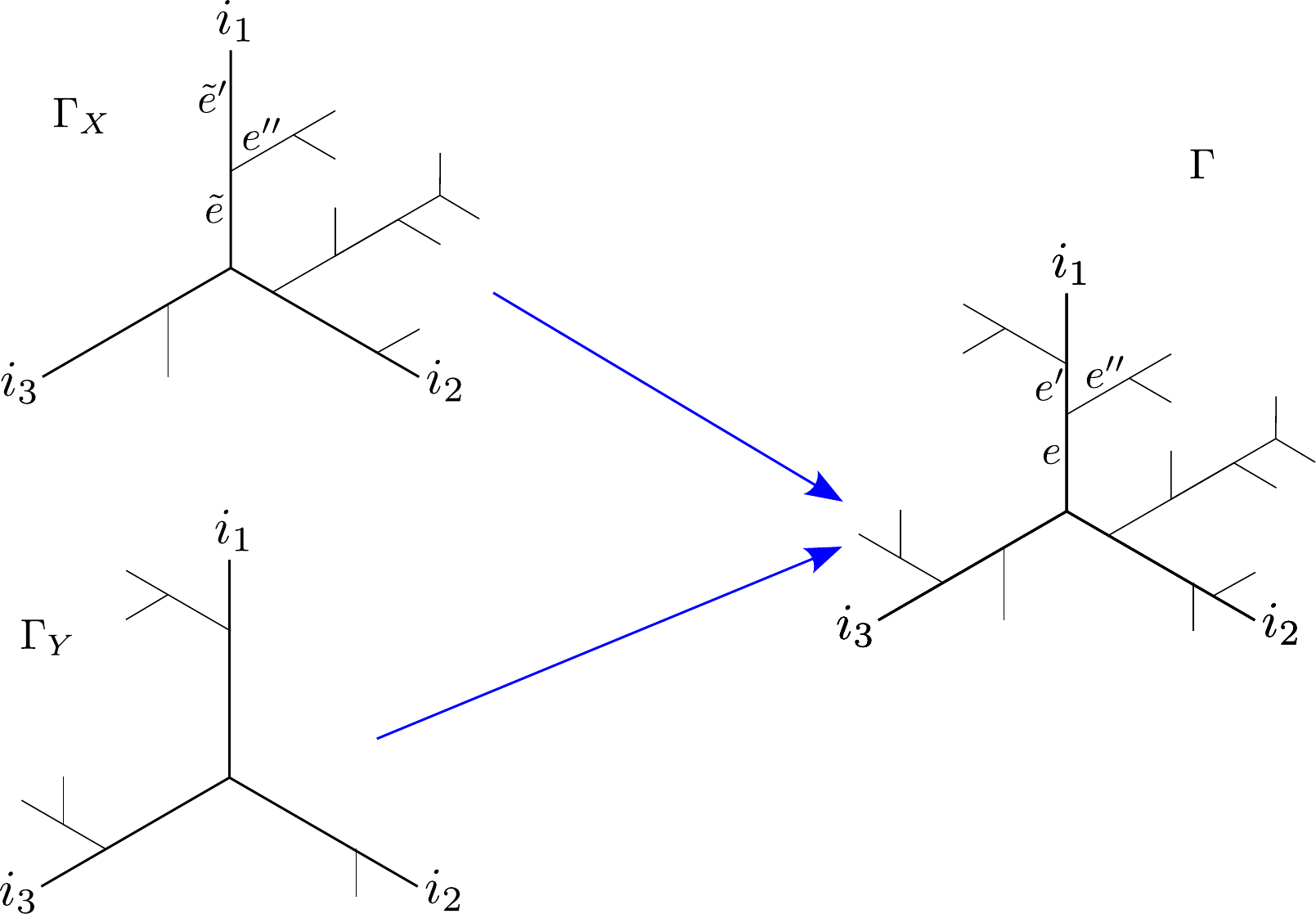}
		\caption{The gluing of $\Gamma_X$ and $\Gamma_Y$ as described in the proof of Lemma \ref{bijection}}
		\label{glueinggraphs}
	\end{figure}
	
	Now, we consider the multiplicity matrix $M_\U$ of $\Gamma$ with respect to $\pi_\U.$ Consider the two columns corresponding to the edges $e$ and $e'$. If we subtract the column of $e'$ from the column of $e,$ the resulting column has only entries in rows that correspond to cross-ratios in $\U_X,$ because for all cross-ratios in $\U_Y$, as either both or none of these two edges contribute to them, their entries get deleted with this matrix operation.  Thus, the entries in this column in the rows corresponding to cross-ratios in $\U_X$ are the same as the ones in the column of $\tilde{e}$ in the multiplicity matrix $M_{\U_X}$ of $\Gamma_X$ with respect to $\pi_{\U_X}$ after subtracting the column of $\tilde{e}'$ from the column of $\tilde{e}.$
	$$\begin{pNiceArray}[first-row,last-col]{cccccc}
			e_1&\dots&e_k&e_{k+1}&\dots&e_{n-3}&\\
			\Block{6-6}<\Huge>{M_{\U}}&& & &  & & C_1 \\
			&&&&&&\vdots \\
			&& && &  & C_k \\
			&&& & & & C_{k+1} \\
			&&& & &&\vdots  \\
			&&&&&&C_{n-3}
	\end{pNiceArray}
	\sim
	\begin{pNiceArray}[first-row,last-col]{cccccc}
		e_1&\dots&e_k&e_{k+1}&\dots&e_{n-3}&\\
		\Block{3-3}<\Huge>{M_{\U_X}}&& & \Block{3-3}<\Huge>{0} &  & & C_1 \\
		&&&&&&\vdots \\
		&& && &  & C_k \\
		\Block{3-3}<\Huge>{0} &&& \Block{3-3}<\Huge>{M_{\U_Y}}& & & C_{k+1} \\
		&&& & &&\vdots  \\
		&&&&&&C_{n-3}
	\end{pNiceArray}$$
	This matrix operation is now applied to every column that corresponds to an edge that contributes to cross-ratios in both $\U_X$ and $\U_Y$. We apply it first at the columns corresponding to bounded edges that are closer to $v_A$ because only then the above arguments still hold as the columns needed for an operation are not edited until then.
	Now, we can sort the rows and columns of the $M_\U$ such that the resulting matrix has block diagonal form with one block only having entries in the rows corresponding to cross-ratios in $\U_X$ and the columns corresponding to edges $e$, where the adjacent edge $e''$ only contributes to cross-ratios in $\U_X$ or is an end with	marking in $X$ and analogously for the other block and $Y.$ These two blocks are identical to the matrices that we obtain from the multiplicity matrices $M_{\U_X}$ and $M_{\U_Y}$ of $\Gamma_X$ and $\Gamma_Y$ with respect to $\pi_{\U_X}$ and $\pi_{\U_Y}$ after applying the corresponding matrix operations, which were described before. Thus, the absolute value of the determinant of $M_\U$ is the product of the absolute values of the determinants of $M_{\U_X}$ and $M_{\U_Y},$ which means that $m=m_X\cdot m_Y.$
\end{proof}
\begin{lemma}\label{trop25}
(See Lemma 2.5 of \cite{R24}.)	Let $\U=\{S_1,...,S_{n-3}\}$ be a set of four-tuples with distinct elements of $[n].$ Suppose there exists a partition $[n]=\{i_1,i_2,i_3\}\sqcup X\sqcup Y$ such that for all $S\in\U$, we have either $S\subseteq\{i_1,i_2,i_3\}\cup X$ or $S\subseteq\{i_1,i_2,i_3\}\cup Y$. Define $\U_X=\{S\in\U:S\subseteq\{i_1,i_2,i_3\}\cup X\}$, and similarly $\U_Y$. Then, we find that 
	 $$d_{\U}=\begin{cases}
		d_{\U_X}\cdot d_{\U_Y}&\abs{\U_X}=\abs{X}\\
		0&\text{otherwise}.
	\end{cases}$$
	
	Here, $d_{\U},$ $d_{\U_X}$ and $d_{\U_Y}$ are the degrees of the respective maps as defined in Definition \ref{mappi}.
\end{lemma}
\begin{proof}

First, we consider the case, where $\abs{\U_X}\ne\abs{X}.$ Then either $\U_X$ or $\U_Y$ contains less cross-ratios than there are markings in $X$ or $Y$, respectively. Let us say, $\abs{\U_X}>\abs{X}.$ Then, the dimension of the moduli space $\M_{0,\{i_1,i_2,i_3\}\cup X}^\mathrm{trop}$ is smaller than the number of cross-ratios in $\U_X.$ 
This means that a tropical curve in $\M_{0,\{i_1,i_2,i_3\}\cup X}^{\mathrm{trop}}$ would have to fulfill more cross-ratio conditions than it has bounded edges which yields an overdetermined system of equalities if the cross-ratios are in general position and thus, the degree $d_{\U_X}$ is 0. Then, the degree of the map $\pi_\U$ also is 0 because if there are no tropical curves fulfilling all cross-ratio conditions in $\U_X,$ then there cannot be any tropical curves fulfilling all cross-ratio conditions in $\U=\U_X\cup\U_Y.$ Thus, we can now assume that $\abs{\U_X}=\abs{X}$.		

Remember, the degree $d_\U$ is the sum over all tropical curves with their multiplicities in $\pi_\U^{-1}(p)$ for a general positioned $p\in \prod_{S\in\U}\M_{0,S}^{\mathrm{trop}},$ see Definition \ref{mappi} and Remark \ref{genpos}. 

We recall that we know from Lemma \ref{bijection} that each such tropical curve $\Gamma$ in $\pi_{\U}^{-1}$ is in bijection with two tropical curves $\Gamma_X$ and $\Gamma_Y$ in $\M_{0,\{i_1,i_2,i_3\}\cup X}^{\mathrm{trop}}$ and $\M_{0,\{i_1,i_2,i_3\}\cup Y}^{\mathrm{trop}},$ where the $\Gamma_X$ uses the markings of $\{i_1,i_2,i_3\}\cup X$ and $\Gamma_Y$ uses the markings of $\{i_1,i_2,i_3\}\cup Y.$ Thus, we can count pairs of tropical curves $(\Gamma_X,\Gamma_Y)$ in $\M_{0,\{i_1,i_2,i_3\}\cup X}^{\mathrm{trop}}\times
\M_{0,\{i_1,i_2,i_3\}\cup Y}^{\mathrm{trop}}$ instead of tropical curves $\Gamma$ in $\M_{0,n}^{\mathrm{trop}}.$

	Now, we again make use of Lemma \ref{bijection}, from where we also know that $m=m_X\cdot m_Y, $ where $m$ is the multiplicity of $\Gamma$ with respect to $\pi_{\U}$ and $m_X$ and $m_Y$ are the multiplicities of $\Gamma_X$ and $\Gamma_Y$ with respect to $\pi_{\U_X}$ and $\pi_{\U_Y}.$ This implies that we can count pairs of tropical curves $(\Gamma_X,\Gamma_Y)$ with multiplicity $m_X\cdot m_Y$ instead of tropical curves $\Gamma$ with multiplicity $m.$
	
	This sum can now be separated into a product of two sums, where the first sum goes over all tropical curves $\Gamma_X$ counted with multiplicity $m_X$ and the second one over all tropical curves $\Gamma_Y$ counted with multiplicity $m_Y.$ By definition of the degree, this is just the product of $d_{\U_X}$ and $d_{\U_Y}$ and thus, $d_\U=d_{\U_X}\cdot d_{\U_Y}.$	
\end{proof}

\begin{corollary}[Outer Triangles]\label{trop26}
	Let us consider a triangulation $T$, that defines a set of cross-ratios, as in Definition \ref{defthreeint}, and assume that there is an outer triangle. The conditions mentioned in Lemma \ref{bijection} and Lemma \ref{trop25} are fulfilled, if we take $i_1$ to be the marking of the edge of this triangle that is also an edge of the $n$-gon and $i_2$ and $i_3$ be the ones of the edges of the $n$-gon adjacent to the opposite vertex, such that the marking $i_2$ appears first when reading all markings in clockwise direction starting at $i_1$. We further define $X$ as the set of all markings between $i_1$ and $i_2$ when looking at the edges of the $n$-gon in clockwise direction, and $Y$ as the set of all markings between $i_3$ and $i_1.$ 	
	
	By applying Lemma \ref{bijection} several times, we obtain a bijection between the set $\U(T)$ of markings of the cross-ratios defined by $T$ and the union of several sets  $\U(T_1),...,\U(T_k)$ of markings of cross-ratios defined by triangulations $T_1,...,T_k$ without outer triangles and the degree $d_{\U(T)}$ equals the product $d_{\U(T_1)}\cdot...\cdot d_{\U(T_k)},$ see Definition \ref{mappi}.
\end{corollary}
\begin{proof}
	First, we fix an outer triangle of $T.$ With this partition of the markings into $X$ and $Y,$ all diagonals now only touch vertices adjacent to edges with markings all in $\{i_1,i_2,i_3\}\cup X$ or all in $\{i_1,i_2,i_3\} \cup Y,$ see Figure \ref{outer-tri}.
	Then, there exists an $k\in \{1,...,n-4\}$ such that we can decompose the $n$-gon into a $(n-k)$-gon with the edges corresponding to the markings in $\{i_1,i_2,i_3\}\cup X$ and an $(k+3)$-gon with the edges corresponding to the markings in $\{i_1,i_2,i_3\}\cup Y,$ see Figure \ref{outer-tri}. Thus, the conditions of Lemma \ref{bijection} and Lemma \ref{trop25} are fulfilled and we can apply them to this setting.
	
	Then, two tropical curves with multiplicities $m_X$ and $m_Y$ fulfilling the cross-ratio conditions of these smaller polygons are in bijection to a tropical curves with multiplicity $m=m_X\cdot m_Y$ that fulfills the cross-ratio conditions defined by $T.$
	
	The number of outer triangles of the $(n-k)$-gon plus the number of outer triangles of the $(k+3)$-gon is one less than the number of outer triangles of the $n$-gon, because the outer triangle along which we decomposed $T$ is now a border triangle in the triangulations of the $(n-k)$-gon and the $(k+3)$-gon. Also, the total number of inner triangles stays the same.  A sketch can be seen in Figure \ref{outer-tri}. 
	
	By repeatedly applying Lemma \ref{bijection}, we can obtain a set $\{T_1,...,T_k\}$ of triangulations of smaller polygons such that there are no outer triangles in any of these triangulations. As this is a bijection as described in the proof of Lemma \ref{bijection} and the multiplicity also is preserved, we obtain the demanded bijection between a set $\U(T)$ of markings of cross-ratios defined by $T$ and the union of several sets $\U(T_1),...,\U(T_k)$ of markings of cross-ratios defined by triangulations $T_1,...,T_k$ without outer triangles. If we now also apply Lemma \ref{trop25}, we see that for the degrees it holds that $d_{\U(T)}=d_{\U(T_1)}\cdot...\cdot d_{\U(T_k)}.$
	\begin{figure}[h]
		\centering
		\includegraphics[scale=0.4]{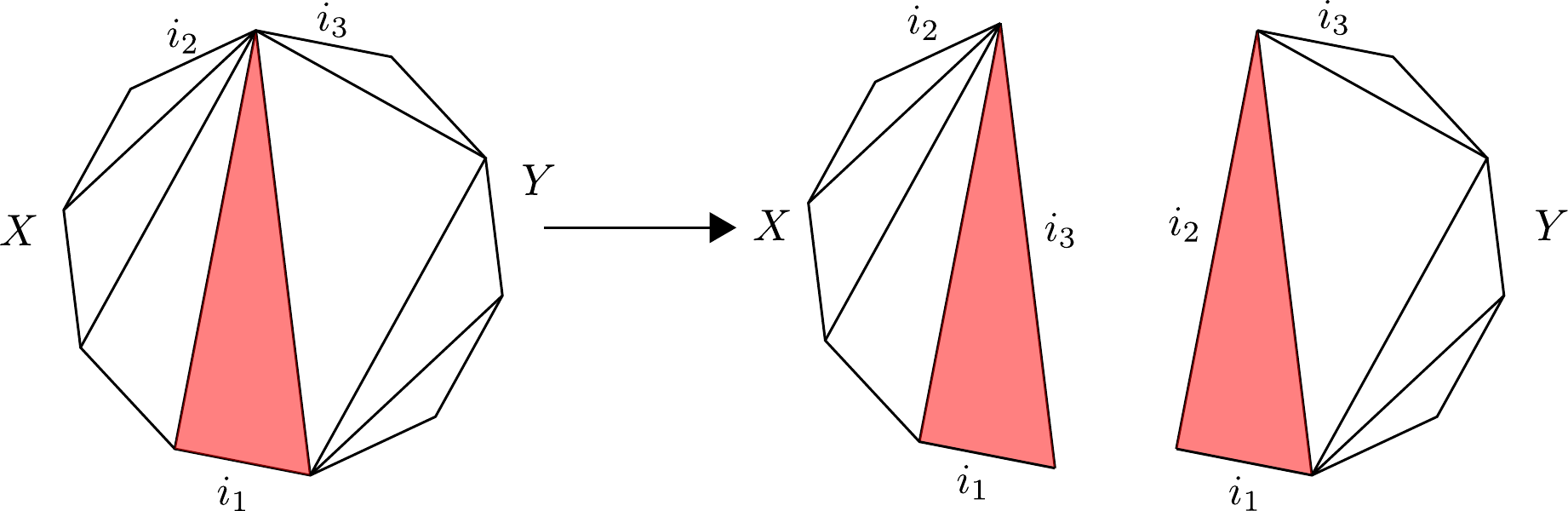}
		\caption{Splitting a triangulation at an outer triangle as described in Corollary \ref{trop26}}
		\label{outer-tri}
	\end{figure}
\end{proof}
\begin{corollary}[No inner triangles]
	If a set of cross-ratio conditions $C_1,...,C_{n-3}$ defined by a triangulation $T$ of an $n$-gon as in Definition \ref{defthreeint} has no inner triangles then there is just one tropical curve of multiplicity 1 in $\pi_{\U(T)}^{-1}(C_1,...,C_{n-3})$ and thus, $d_{\U(T)}=1.$
\end{corollary}
\begin{proof}
	If there are no inner triangles, all triangles must be outer triangles or border triangles. Because the $n$-gon has $n$ edges, there are $n-2$ triangles in total. As a border triangle is adjacent to two and an outer triangle is adjacent to one boundary edge of the polygon, there have to be two border triangles and $n-4$ outer triangles for $n\ge4$. We can now decompose the $n$-gon inductively into smaller polygons according to Corollary \ref{trop26} such that we end up with $n-3$ quadrilaterals and this process also does not create any new inner triangles. The quadrilaterals all correspond to unique tropical curves of multiplicity 1 each. As we know from Lemma \ref{bijection} that the multiplicity does not change by gluing these tropical curves together again, we obtain that the multiplicity of a tropical curve in $\pi_{\U(T)}^{-1}(C_1,...,C_{n-3})$ always is 1 and there is only one such tropical curve, so also the degree $d_{\U(T)}$ of the map $\pi_{\U(T)}$ is 1.
\end{proof}

\section{The Main Theorem}\label{2ndint}
In this section, we want to prove Theorem \ref{allcases}, which is also mentioned in the introduction as Theorem \ref{mainint}. We first consider the case, where all cross-ratios of a triangulation are interpreted in the neighboring way, as in Definition \ref{defthreeint} and prove the theorem for this case. Then, in section \ref{3rdint}, we consider all interpretations and prove the remaining part of the theorem using a case distinction for all for all possibilities of interpretations.
\subsection{The Neighboring Interpretation}
In this section, all cross-ratios are interpreted in the neighboring way.
\begin{definition}\label{deftriineq}
	Let $T$ be a triangulation of an $n$-gon as in Definition \ref{defthreeint} that defines cross-ratios $C_1,...,C_{n-3}.$ We say that three cross-ratios $C_i,$ $C_j$ and $C_k$ with lengths $\lambda_i,$ $\lambda_j$ and $\lambda_k$ fulfill the triangle inequalities if $\lambda_i+\lambda_j>\lambda_k$, $\lambda_i+\lambda_k>\lambda_j$ and $\lambda_j+\lambda_k>\lambda_i.$ Then, we also say that an inner triangle of $T$ that is adjacent to the diagonals defining $C_i,$ $C_j$ and $C_k$ fulfills the triangle inequalities.
\end{definition}

\begin{example}\label{ex6b}
	Consider the triangulation $T$ depicted in Figure \ref{ex6bim}. The three cross-ratios defined by $T$ in the neighboring interpretation are $C_1=(\lambda_1,\cros{1}{6}{2}{3})$, $C_2=(\lambda_2,\cros{2}{3}{4}{5})$ and $C_3=(\lambda_3,\cros{1}{6}{4}{5})$. If the three cross-ratios fulfill the triangle inequalities, then we obtain that there is one unique tropical curve in $\pi_{\U(T)}^{-1}(C_1,C_2,C_3)$ that looks like the one in the middle of Figure \ref{ex6bim} and with multiplicity 2, as the multiplicity matrix is $$\begin{pmatrix}
		1&1&0\\
		0&1&1\\
		1&0&1
	\end{pmatrix},$$
or a permutation of its rows and columns. If the triangle inequalities are not fulfilled, we can say that w.l.o.g the green cross-ratio $C_1=(\lambda_1,\cros{1}{6}{2}{3})$ has the largest length. Then we obtain that there are exactly two different tropical curves with multiplicity 1 each in $\pi_{\U(T)}^{-1}(C_1,C_2,C_3)$, see the right of Figure \ref{ex6bim}, as their multiplicity matrices are $$\begin{pmatrix}
	1&1&0\\
	1&0&1\\
	1&0&0
	\end{pmatrix},$$ or a permutation of it.
\end{example}
\begin{figure}
	\centering
	\includegraphics[scale=0.4]{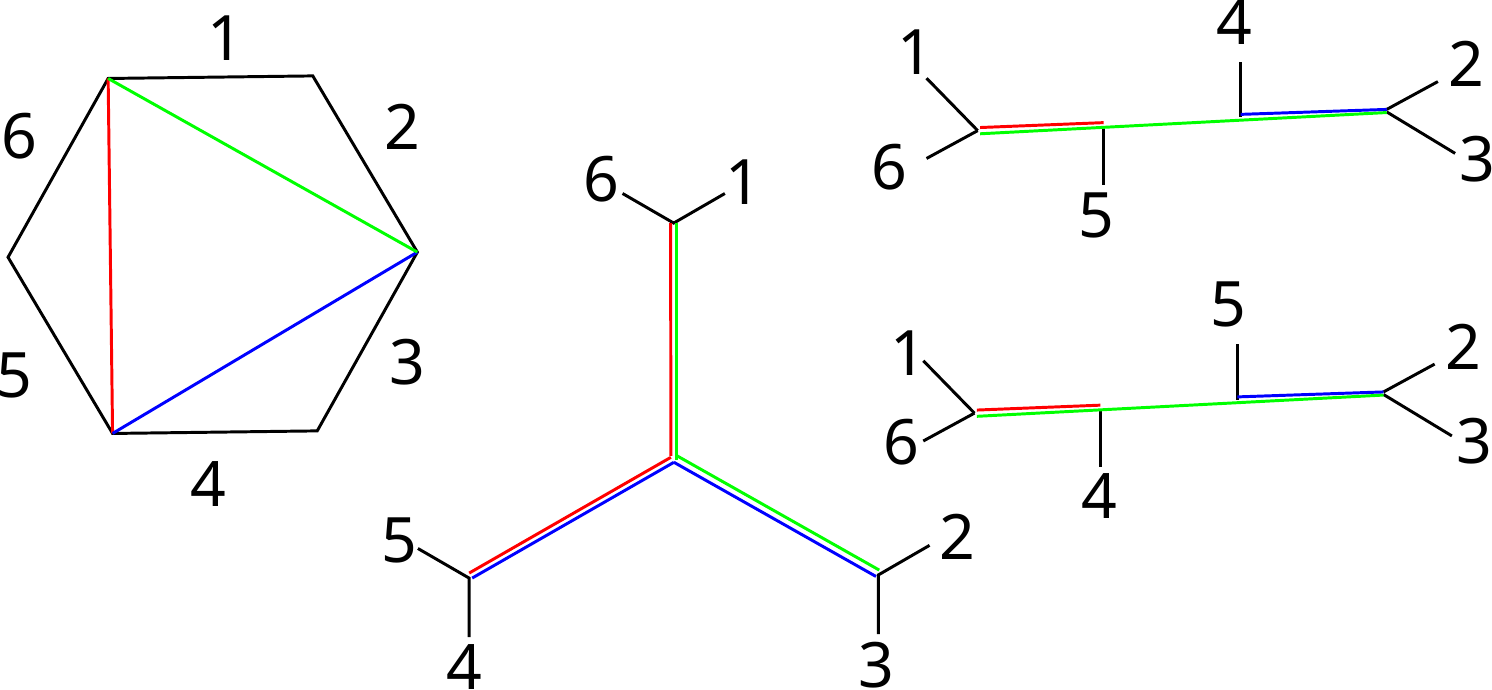} 
	\caption{The tropical curves that fulfill the cross-ratio conditions given by that triangulation of a hexagon in the neighboring interpretation, see Example \ref{ex6b}}
	\label{ex6bim}
\end{figure}

\begin{construction}\label{addingtri}
In the following three lemmas, we consider a triangulation $T$ without outer triangles and with diagonals of lengths $\lambda_3,...,\lambda_{n-1}$ of an $n$-gon. We now explain a way to add two new edges, diagonals and cross-ratios $C_1$ and $C_2$ with lengths $\lambda_1$ ans $\lambda_2$ to $T.$ Out of all the inner triangles of $T,$ we can choose one that is adjacent to a border triangle and define $v$ to be one of the vertices that is adjacent to the diagonal between these two triangles. W.l.o.g, we can say that this diagonal is the one with length $\lambda_3.$  If $n=4$ and thus, there are no inner triangles, we choose $v$ to be a vertex, where the unique diagonal (which then must have length $\lambda_3$) is ending. Now, we consider the border triangle (or for $n=4$ one of the border triangles) that is adjacent to this diagonal. We Then can construct a triangulated $(n+2)$-gon without outer triangles by defining the two other edges of this border triangle to be new diagonals that define cross-ratios $C_1$ and $C_2$ and adjoining a border triangle on the other side of both of them. This construction can be seen in Figure \ref{add-triangle}.
\begin{figure}[h]
	\centering
	\includegraphics[scale=0.5]{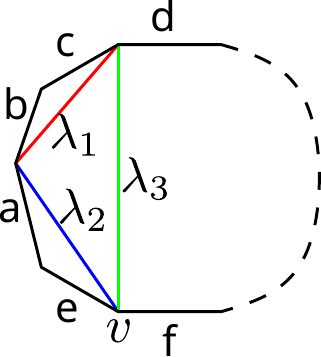}
	\caption{Adding a new inner triangle to a given triangulation as described in Construction \ref{addingtri}}
	\label{add-triangle}
\end{figure}
\end{construction}
\begin{lemma}\label{lemaddtri}
	Let $T$ be a triangulation of an $n$-gon without outer triangles as in Definition \ref{defthreeint}. Then $T$ can be constructed by applying Construction \ref{addingtri} to a quadrilateral $\frac{n}{2}-2$ times.
\end{lemma}
\begin{proof}
		We prove the statement by induction. This trivially true for $n=4.$ Now, we assume the statement holds for $n-2$ and we consider a triangulation $T$ of an $n$-gon without outer triangles. As $T$ has $n-2$ triangles in total, a border triangle is adjacent to two edges of the $n$-gon and an inner triangle to none, $n$ must be even and there must be $\frac{n}{2}$ border triangles and $\frac{n}{2}-2$ inner triangles. As border triangles cannot share a diagonal for $n>4,$ there must be an inner triangle that is adjacent to two border triangles. By removing these two border triangles, we obtain a triangulation of an $n-2$-gon without outer triangles for which the statement holds. As removing the two border triangles is inverse to Construction \ref{addingtri}, we also obtain the statement for $n.$
\end{proof}
\begin{remark}
	In Construction \ref{addingtri}, we also can rename the edges so that $C_1=(\lambda_1,(ab|cd)),$ $C_2=(\lambda_2,(ab|ef))$ and $C_3=(\lambda_3,(cd|ef))$ are these three cross-ratios of the described new inner triangle. In the following Figures \ref{Im21}, \ref{Im22} and \ref{Im23}, $C_1$ will be depicted in red, $C_2$ in blue and $C_3$ in green.
	
	Now, we describe what happens to a tropical curve in $\pi_{\U(T)}^{-1}(C_3,...,C_{n-3})$ when we add these new triangles. 
\end{remark}
\begin{theorem}\label{thm2}
	Let $T$ be a triangulation of an $n$-gon with $d$ inner triangles, where all cross-ratios $C_1,...,C_{n-3}$ are interpreted in the neighboring way as in Definition \ref{defthreeint}. Let $k\in\{1...,d\},$ be the number of inner triangles that fulfill the triangle inequalities, see Definition \ref{deftriineq}, then there are $2^{d-k}$ tropical curves of multiplicity $2^{k}$ each in $\pi_{\U(T)}^{-1}(C_1,...,C_{n-3})$ and thus, $d_{\U(T)}=2^d,$ see Definition \ref{mappi}.
\end{theorem}
To prove this theorem, we use induction on the number of inner triangles and use the following three Lemmas. 
In these, a triangulation of an $n$-gon without outer triangles is considered. For any such triangulation it holds that $n$ must be even and there are $\frac{n}{2}-2$ inner triangles.
\begin{lemma}\label{case21}
	Let $T$ be a triangulation of an $n$-gon without outer triangles that defines the cross-ratios $C_3,...,C_{n-1},$ as in Definition \ref{defthreeint}. We interpret all cross-ratios in the neighboring way, and assume there are $2^{n/2-k-2}$ tropical curves of multiplicity $2^k$ each, for $k$ in $\{1,...,\frac{n}{2}-2\},$ in $\pi_{\U(T)}^{-1}(C_3,...,C_{n-1}),$  see Definition \ref{mappi}. By adding another inner triangle at $C_3$, as in Construction \ref{addingtri}, with the new cross-ratios $C_1$ and $C_2$ also interpreted in the neighboring way, that fulfills the triangle inequalities, we obtain a triangulation $T'$ of an $(n+2)$-gon for which it holds that there are $2^{n/2-k-2}$ tropical curves of multiplicity $2^{k+1}$ each in $\pi_{\U(T')}^{-1}(C_1,...,C_{n-1}).$
\end{lemma}
\begin{proof}
	Let $\U_3=\{S_1,S_2,S_3\}$ be the set of the markings of the cross-ratios $C_1,$ $C_2$ and $C_3.$ Fulfilling the three cross-ratios $C_1,$ $C_2$ and $C_3$ implies that in $\pi_{\U_3}^{-1}(C_1,C_2,C_3)$ there is just one tropical curve $\Gamma_3$ and it looks like described in Example \ref{ex6b}. Also, if we fix an $n$-marked tropical curve $\Gamma_n$ in $\pi_{\U(T)}^{-1}(C_3,...,C_{n-1}),$ there is just one way of gluing $\Gamma_3$ and $\Gamma_n$ together so that all conditions are still fulfilled, which means that the resulting tropical curve $\Gamma$ is in $\pi_{\U(T)}^{-1}(C_1,...,C_{n-1}).$
	This is because both tropical curves $\Gamma_3$ and $\Gamma_n$ share a common cross-ratio $C_3$ with its markings and thus, the edges that contribute to $C_3$ and the paths to the markings of $C_3$ have to be identified. As $C_3$ is fulfilled in both tropical curves, the lengths of these edges add up to the same number $\lambda_3,$ and so, they can be glued on each other. In this process, some of the edges of $\Gamma_3$ and $\Gamma_n$ can get subdivided. We now analyze how this glued tropical curve $\Gamma$ in $\pi_{\U(T')}^{-1}(C_1,...,C_{n-1})$ has to look like and what happens to its multiplicity.
	
	 As the ends $a$ and $b$ do not appear in $C_3,...,C_{n-1},$ there is a bounded edge $e_1$ of the underlying graph of $\Gamma$ adjacent to these two ends  that only contributes to the two cross-ratios $C_1$ and $C_2,$ see sketch in Figure \ref{Im21}. This implies that in the corresponding column of $e_1$ in the multiplicity matrix, in the rows corresponding to $C_1$ and $C_2$, the entry is 1 and 0 in the rest of this column.
	\begin{figure}
		\centering
		\includegraphics[scale=0.5]{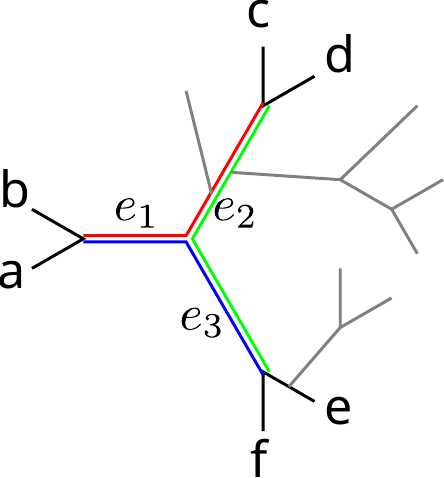}
		\caption{A tropical curve with edges corresponding to an inner triangle that fulfills the triangle inequalities, see Lemma \ref{case21}}
		\label{Im21}
	\end{figure}
	
	Furthermore, the vertex, where this edge $e_1$ ends, has two other adjacent edges, we call them $e_2$ and $e_3$, see Figure \ref{Im21}. As all other cross-ratios $C_3,...,C_{n-1}$ use other markings than $a$ and $b,$ either both or none of $e_2$ and $e_3$ contribute to them. However, one of these two edges, w.l.o.g. let us say it is $e_2,$ contributes to $C_1$ and the other one to $C_2$, which leads to the two columns of $e_2$ and $e_3$ differing in only the rows of $C_1$ and $C_2$. So, when we subtract the column $e_3$ from the column $e_2$, we end up with a block diagonal matrix of the form 
	
	$$\begin{blockarray}{cccccc}
		e_1& e_2-e_3 &  &  &  \\
		\begin{block}{(ccccc)c}
			1 & 1 & * & \dots& * & C_1 \\
			1 & -1 & * & \dots& * & C_2 \\
			0 & 0 & * & \dots& * &  \\
			\vdots & \vdots & \vdots & \ddots& \vdots  &  \\
			0 & 0 & * & \dots& *  &  \\
		\end{block}
	\end{blockarray}.$$
	
	 The $(2\times 2)$-block on the left upper side has determinant $ \pm 2$. The block on the right lower side is identical to the $n\times n$ multiplicity matrix of $\pi_{\U(T)}$  of the tropical curve with the ends $a$ and $b$, and thus also the cross-ratios $C_1$ and $C_2$, removed, because then the edge $e_1$ just is removed and the edge $e_2$ and $e_3$ merge. 
	 
	Thus, we obtain that there is only this one possibility to obtain a tropical curve with the ends $a$ and $b$ added that lives in $\pi_{\U(T)}^{-1}(C_3,...,C_{n-1})$ from an $n$-marked tropical curve in $\pi_{\U(T)}^{-1}(C_3,...,C_{n-1})$ by adding the ends $a$ and $b,$ and the multiplicity of it increases by a factor of 2.
\end{proof}
\begin{lemma} \label{case22}
	 Let $T$ be a triangulation of an $n$-gon without outer triangles that defines the cross-ratios $C_3,...,C_{n-1},$ as in Definition \ref{defthreeint}. We interpret all cross-ratios in the neighboring way, and assume there are $2^{n/2-k-2}$ tropical curves of multiplicity $2^k$ each, for $k$ in $\{1,...,\frac{n}{2}-2\},$ in $\pi_{\U(T)}^{-1}(C_3,...,C_{n-1}),$ see Definition \ref{mappi}. By adding another inner triangle at $C_3$, as in Construction \ref{addingtri}, with the new cross-ratios $C_1$ and $C_2$ with lengths $\lambda_1$ and $\lambda_2$ also interpreted in the neighboring way, such that $\lambda_3>\lambda_1+\lambda_2$, we obtain a triangulation $T'$ of an $(n+2)$-gon for which it holds that there are $2^{n/2-k-1}$ tropical curves of multiplicity $2^{k}$ each in $\pi_{\U(T')}^{-1}(C_1,...,C_{n-1}).$
\end{lemma}
\begin{proof}
	Let $\U_3=\{S_1,S_2,S_3\}$ be the set of the markings of the cross-ratios $C_1,$ $C_2$ and $C_3.$ Fulfilling the three cross-ratios $C_1,$ $C_2$ and $C_3$ implies that in $\pi_{\U_3}^{-1}(C_1,C_2,C_3)$ there are exactly two tropical curves and they look like described in Example \ref{ex6b}. We choose one of them and call it $\Gamma_3.$ The rest of the proof can be done analogously when choosing the other one.
	Also, if we fix an $n$-marked tropical curve $\Gamma_n$ in $\pi_{\U(T)}^{-1}(C_3,...,C_{n-1}),$ there is just one way of gluing $\Gamma_3$ and $\Gamma_n$ together so that all conditions are still fulfilled, which means that the resulting tropical curve $\Gamma$ is in $\pi_{\U(T)}^{-1}(C_1,...,C_{n-1}).$
	This is because both tropical curves $\Gamma_3$ and $\Gamma_n$ share a common cross-ratio $C_3$ with its markings and thus, the edges that contribute to $C_3$ and the paths to the markings of $C_3$ have to be identified. As $C_3$ is fulfilled in both tropical curves, the lengths of these edges add up to the same number $\lambda_3,$ and so, they can be glued on each other. In this process, some of the edges of $\Gamma_3$ and $\Gamma_n$ can get subdivided. We now analyze how this glued tropical curve $\Gamma$ in $\pi_{\U(T')}^{-1}(C_1,...,C_{n-1})$ has to look like and what happens to its multiplicity.
	
	As $\lambda_3$ is the largest length, in both cases, the vertex adjacent to the end $a$ of $\Gamma$ is adjacent to two bounded edges (and no other ends) and the same holds for the end $b$ because they are not adjacent to other ends under the image of the map $ft_{[n]\setminus \{a,b,c,d,e,f\}}$. We call these edges $e_1$, $e_2$, $e_3$ and $e_4$, where $e_1$ is contributing to $C_1$ and $e_2$ is contributing to $C_2$, $e_3$ and the end $a$ share a vertex and the same holds for $e_4$ and $b$, see Figure \ref{Im22}. It can happen that $e_3$ and $e_4$ coincide but this does not affect the rest of the proof. A sketch of this can be found in Figure \ref{Im22}, where the gray edges are an example of possible further edges in the tropical curve.
	
	\begin{figure}
		\centering
		\includegraphics[scale=0.5]{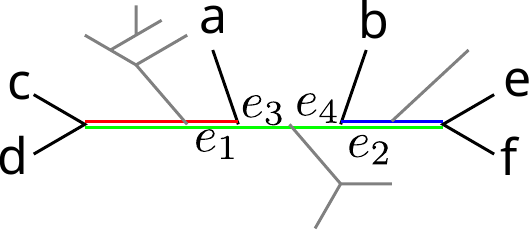}
		\caption{A tropical curve where $\lambda_3$ (in green) is larger than the sum of $\lambda_1$ (in red) and $\lambda_2$ (in blue), see Lemma \ref{case22}}
		\label{Im22}
	\end{figure}
	
	 As there are no other cross-ratios where the markings $a$ and $b$ appear, for all other cross-ratios $C_3,..,C_{n-1}$ the edges $e_1$ and $e_3$ either both or none of them contribute to them and the same holds for the edges $e_2$ and $e_4.$ Thus, the columns of $e_1$ and $e_3$ and also of $e_2$ and $e_4$ in the multiplicity matrix only differ in one entry which is the one in the row of $C_1$ or $C_2$, respectively, see Figure \ref{Im22}. So, if we subtract column $e_3$ from column $e_1$ and column $e_4$ from column $e_2$, the columns $e_1$ and $e_2$ both only have one entry, which is a 1 in the row corresponding to $C_1$ or $C_2$, respectively. Now, this matrix has a block diagonal form, where the upper left block has determinant 1, and the block on the right lower side is identical to the multiplicity matrix of the tropical curve with the ends $a$ and $b$, and thus also the cross-ratios $C_1$ and $C_2$ removed, because then the edges $e_1$ and $e_3$ as well as the edges $e_2$ and $e_4$ merge. 
	 
	 $$\begin{blockarray}{cccccc}
	 	e_1-e_3& e_2-e_4 &  &  &  \\
	 	\begin{block}{(ccccc)c}
	 		1 & 0 & * & \dots& * & C_1 \\
	 		0 & 1 & * & \dots& * & C_2 \\
	 		0 & 0 & * & \dots& * &  \\
	 		\vdots & \vdots & \vdots & \ddots& \vdots  &  \\
	 		0 & 0 & * & \dots& *  &  \\
	 	\end{block}
	 \end{blockarray}$$
	 
	Thus, we obtain that for the two possibilities of how the 6-marked tropical curve in the image of the map $ft_{[n]\setminus \{a,b,c,d,e,f\}}$ can look like, there is only this one possibility to obtain an $(n+2)$-marked tropical curve that is in $\pi_{\U(T')}^{-1}(C_1,...,C_{n-1})$ from an $n$-marked tropical curve in $\pi_{\U(T)}^{-1}(C_3,...,C_{n-1})$ by adding the ends $a$ and $b$, and the multiplicities of them are the same.
\end{proof}
\begin{lemma}	\label{case23}
	Let $T$ be a triangulation of an $n$-gon without outer triangles that defines the cross-ratios $C_3,...,C_{n-1},$ as in Definition \ref{defthreeint}. We interpret all cross-ratios in the neighboring way, and assume there are $2^{n/2-k-2}$ tropical curves of multiplicity $2^k$ each, for $k$ in $\{1,...,\frac{n}{2}-2\},$ in $\pi_{\U(T)}^{-1}(C_3,...,C_{n-1}),$ see Definition \ref{mappi}. By adding another inner triangle at $C_3$, as in Construction \ref{addingtri}, with the new cross-ratios $C_1$ and $C_2$ with lengths $\lambda_1$ and $\lambda_2$ also interpreted in the neighboring way, such that $\lambda_1>\lambda_2+\lambda_3$, we obtain a triangulation $T'$ of an $(n+2)$-gon for which it holds that there are $2^{n/2-k-1}$ tropical curves of multiplicity $2^{k}$ each in $\pi_{\U(T')}^{-1}(C_1,...,C_{n-1}).$
\end{lemma}
\begin{proof}
	Let $\U_3=\{S_1,S_2,S_3\}$ be the set of the markings of the cross-ratios $C_1,$ $C_2$ and $C_3.$ Fulfilling the three cross-ratios $C_1,$ $C_2$ and $C_3$ implies that in $\pi_{\U_3}^{-1}(C_1,C_2,C_3)$ there are exactly two tropical curves and they look like described in Example \ref{ex6b}.We choose one of them and call it $\Gamma_3.$ The rest of the proof can be done analogously when choosing the other one.
	Also, if we fix an $n$-marked tropical curve $\Gamma_n$ in $\pi_{\U(T)}^{-1}(C_3,...,C_{n-1}),$ there is just one way of gluing $\Gamma_3$ and $\Gamma_n$ together so that all conditions are still fulfilled, which means that the resulting tropical curve $\Gamma$ is in $\pi_{\U(T)}^{-1}(C_1,...,C_{n-1}).$
	This is because both tropical curves $\Gamma_3$ and $\Gamma_n$ share a common cross-ratio $C_3$ with its markings and thus, the edges that contribute to $C_3$ and the paths to the markings of $C_3$ have to be identified. As $C_3$ is fulfilled in both tropical curves, the lengths of these edges add up to the same number $\lambda_3,$ and so, they can be glued on each other. In this process, some of the edges of $\Gamma_3$ and $\Gamma_n$ can get subdivided. We now analyze how this glued tropical curve $\Gamma$ in $\pi_{\U(T')}^{-1}(C_1,...,C_{n-1})$ has to look like and what happens to its multiplicity.

	W.l.o.g. we say that $\lambda_1$ (and not $\lambda_2$) is the largest length of these cross-ratios, so $\lambda_1>\lambda_2+\lambda_3.$ In both cases of which tropical curve $\Gamma_3$ was chosen, there is a vertex adjacent to both ends $a$ and $b$ in $\Gamma$, which also is adjacent to an bounded edge $e_1$ that only contributes to the two cross-ratios $C_1$ and $C_2,$ see Figure \ref{Im23}, because for $\Gamma_3$ this is the case and no other cross-ratios than $C_1$ and $C_2$ use the markings $a$ and $b$. So, the corresponding column of $e_1$ in the multiplicity matrix has the entry 1 in the rows corresponding to $C_1$ and $C_2$ and zeroes elsewhere.
	
	Depending on which of the two possibilities of choosing $\Gamma_3$ we are looking at, either the end $e$ or the end $f$ is nearer to the ends $a$ and $b$. W.l.o.g., we say it is $e$. All other ends of $\Gamma$ must be somewhere in between $f$ and $c$ or $d$, because again the other cross-ratios do not use the markings $a$ and $b$, which can also be seen inductively by first considering the triangle on the other side of the diagonal with length $\lambda_3$. Thus, the vertex adjacent to $e$ is adjacent to the edge $e_1$ and another bounded edge, we call it $e_2.$ The edge $e_2$ is only contributing to $C_1,$ because no other cross-ratios use the markings $a$ and $b$ and all that use the marking $e$ also use the marking $f$ (as we are in the neighboring case). A sketch of this can be found in Figure \ref{Im23}, where the gray edges are an example of possible further edges in the tropical curve.
	
	\begin{figure}
		\centering
		\includegraphics[scale=0.5]{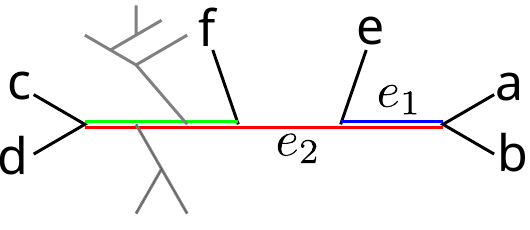}
		\caption{A tropical curve where $\lambda_1$>$\lambda_2$+$\lambda_3$, see Lemma \ref{case23}}
		\label{Im23}
	\end{figure}
	
	$$\begin{blockarray}{cccccc}
		e_1& e_2 &  &  &  \\
		\begin{block}{(ccccc)c}
			1 & 1 & * & \dots& * & C_1 \\
			1 & 0 & * & \dots& * & C_2 \\
			0 & 0 & * & \dots& * &  \\
			\vdots & \vdots & \vdots & \ddots& \vdots  &  \\
			0 & 0 & * & \dots& *  &  \\
		\end{block}
	\end{blockarray}$$
	
	So, the column in the multiplicity matrix corresponding to $e_2$ has a 1 in the row corresponding to $C_1$ and zeroes elsewhere. Hence, the multiplicity matrix of $\Gamma$ with respect to $C_1,...,C_{n-1}$ has block diagonal form, where the upper left block has determinant -1 and the lower right block is identical to the multiplicity matrix of the tropical curve with the ends $a$ and $b$, and thus also the cross-ratios $C_1$ and $C_2,$ removed, because when removing those ends, the edge $e_2$ gets contracted and the edge $e_1$ gets removed.
	
	Thus, we obtain that for the two possibilities of how the 6-marked tropical curve in the image of the map $ft_{[n]\setminus \{a,b,c,d,e,f\}}$ can look like, there is only this one possibility to obtain an $(n+2)$-marked tropical curve that lives in $\pi_{\U(T')}^{-1}(C_1,...,C_{n-1})$ from an $n$-marked tropical curve in $\pi_{\U(T)}^{-1}(C_3,...,C_{n-1})$ by adding the ends $a$ and $b$, and the multiplicities of them are the same.	
\end{proof}

\begin{proof}[Proof of Theorem \ref{thm2}]
	
	Corollary \ref{trop26} does not use any of these three interpretations of cross-ratios so we can also apply it in this case. 
	
	For any triangulation of an $n$-gon, we can split this $n$-gon up into smaller polygons as described in the proof of Corollary \ref{trop26} until there are no outer triangles left any more. Then, we can use the Lemmas \ref{case21}, \ref{case22} and \ref{case23} to find the number of tropical curves and their multiplicities.
\end{proof}
\begin{remark}[Splitting up]
	Let $T$ be a triangulation of an $n$-gon as in Definition \ref{defthreeint}, but where the lengths of the diagonals are not determined yet. Then, for any choice of which triangles should fulfill the triangle inequalities and which not, we can find fitting cross-ratios lengths as also mentioned in Remark \ref{genpos}. This means that for a fixed triangulation with $d$ inner triangles where all cross-ratios are interpreted in the neighboring way, see Definition \ref{defthreeint}, when varying the lengths of the cross-ratios, we can obtain from 1 up to $2^d$ different tropical curves in $\pi_{\U(T)}^{-1}(C_1,...,C_{n-3}).$
	
	 The combinatorics between the number of tropical curves depending on the lengths of the cross-ratios can be viewed as the the graph of a $d$-dimensional cube, also called boolean graph, as follows. The $2^d$ vertices of the graph correspond to the $2^d$ corners of the unit cube and vertices are connected if the coordinates of the corresponding corners only differ in a single entry. Each inner triangle represents a coordinate and has the entry 1 if the triangle inequalities are fulfilled and 0 if not. Thus, the vertex corresponding to the origin represents a set of $2^d$ different tropical curves of multiplicity 1 each and the vertex corresponding to the opposite corner represents a single tropical curve of multiplicity $2^d.$ When varying the lengths of one or more cross-ratios, we change the number of inner triangles that fulfill the inequalities, so we wander around in this graph. 
\end{remark}

\subsection{A Complete List of Cases}\label{3rdint}
Now, we consider all interpretations of Definition \ref{defthreeint} in order to prove the rest of Theorem \ref{allcases}, mentioned as Theorem \ref{mainint} in the introduction. We observe how to construct all tropical curves fulfilling cross-ratio conditions defined by a triangulation in any interpretation.
\begin{example}\label{ex6c}
	We consider the triangulation of a hexagon depicted in Figure \ref{ex6cim} with the three cross-ratios $C_1=(\lambda_1,(26|13))$, $C_2=(\lambda_2,(24|35))$ and $C_3=(\lambda_3,(46|15)).$ Here, we find that no matter how $\lambda_1$, $\lambda_2$ and $\lambda_3$ are sorted, there are always two tropical curves of multiplicity 1 that fulfill these conditions. For the case that $\lambda_2$ is the smallest out of these three lengths, the two possible tropical curves are depicted in Figure \ref{ex6cim}. The other cases look similar.
	\begin{figure}
		\centering
		\includegraphics[scale=0.4]{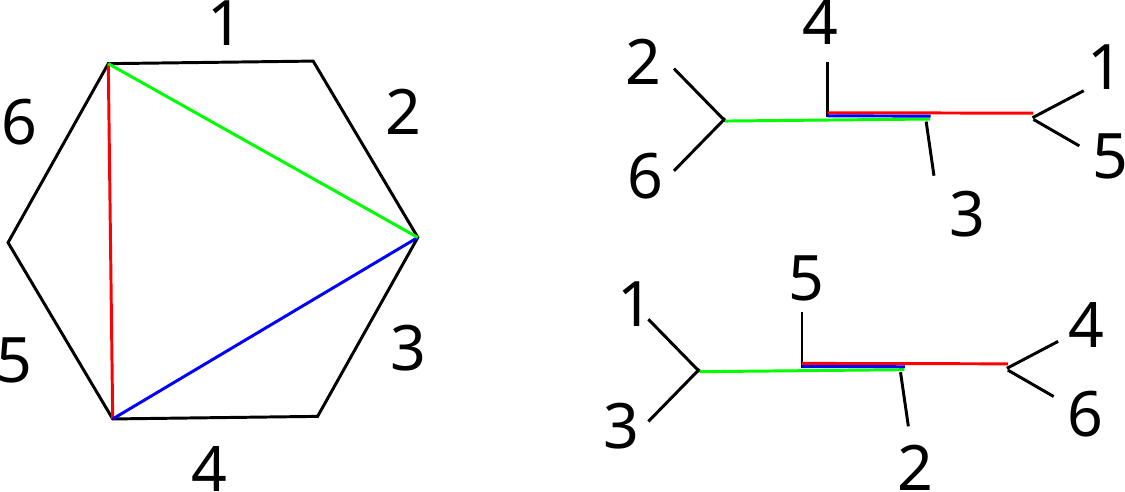}
		\caption{The two graphs that fulfill the cross-ratio conditions given by that triangulation of a hexagon in the intersecting interpretation, see Example \ref{ex6c}}
		\label{ex6cim}
	\end{figure}
\end{example}
\begin{theorem}\label{allcases}
	Let $T$ be a triangulation of an $n$-gon with $d$ inner triangles and $n-3$ diagonals that define cross-ratio conditions $C_1,...,C_{n-3}$ in general position and fix an interpretation for each cross-ratio as in Definition \ref{defthreeint}. Let $k$ be the number of inner triangles, where all three cross-ratios are interpreted in the neighboring way and that fulfill the triangle inequalities, see Definition \ref{deftriineq}. Then there are $2^{d-k}$ tropical curves of different combinatorial types and multiplicity $2^k$ each in $\pi_{\U(T)}^{-1}(C_1,...,C_{n-3})$, which means that the degree $d_{\U(T)}$ of the map $\pi_{\U(T)}$ of Definition \ref{mappi} is $2^d.$
\end{theorem}
\begin{proof}
	Because of Lemma \ref{bijection}, Lemma \ref{trop25} and Corollary \ref{trop26}, we can assume that there are no outer triangles and because of Lemma \ref{lemaddtri} we know that $T$ can be constructed as in Construction \ref{addingtri}.

	The rest is proven by induction on $d$. For no inner triangle, so we just have a quadrilateral with a diagonal, there is just one possible tropical curve in the preimage of the map $\pi_{\U(T)}$, see Definition \ref{mappi} for each of the three interpretations. These three tropical curves are depicted in Figure \ref{m04}. The tropical curve where all four ends meet in a vertex cannot appear since we only consider cross-ratios in general position. Now, we analyze how many tropical curves fulfill all cross-ratio conditions, when we add an inner triangle, as in Construction \ref{addingtri}, so $d$ increases by 1. 

	For the induction step, we take any tropical curve in $\pi_{\U(T)}^{-1}(C_3,...,C_{n-1})$ and adjoin another inner triangle at $C_3$ to this $n$-gon to obtain a triangulation $T'$, as described before in Construction \ref{addingtri}. Let $a$ and $b$ be the new markings that we obtain when adjoining this new inner triangle and we call the edges next to them $c,d,e$ and $f$, as it can be seen in Figure \ref{add-triangle}, where the old cross-ratio $C_3$ consisting of the four markings $c,d,e,f$ is depicted in green, and the new cross-ratios $C_1$, with the markings $a,b,c,d$ and $C_2$, with the markings $a,b,e,f$ are depicted in red and blue. The interpretations for these three cross-ratios are not chosen yet.
	
	Now, we consider every case on how these three cross-ratios can be interpreted. All these cases can be found in Table \ref{table} of Appendix \ref{cases}, where in the first three columns, the interpretations of the three cross-ratios $C_3,$ $C_1$ and $C_2$ are given. The letters "d", "n" and "i" in the table stand for the dual, neighboring and intersecting interpretation. Because everything is done analogously if we change the interpretations of the cross-ratios $C_1$ and $C_2,$ always only one of these is considered. Then, there is a sketch of where these ends are added to (a part of) the tropical curve in $\pi_{\U(T)}^{-1}(C_3,...,C_{n-1})$ so that all three cross-ratios are fulfilled in the demanded interpretations. Of course, the tropical curve can also have more ends if there are already several inner triangles but we do not need to consider these as no other cross-ratios involve the markings $a$ and $b.$ So the sketch is reduced to the edges that contribute to the cross-ratio $C_3$ and the paths to the ends $c,d,e,f$. Several times, these sketches look differently depending on the order of these three cross-ratios lengths. In the sketch, we always see on which branches the new ends $a$ and $b$ have to be added - the exact position depends on the exact values of the lengths. There are no other ways of adding the new ends apart from the ones depicted. We look at one of these cases more closely.
	
	In the first case where all three cross-ratios are interpreted in the dual way, we have the cross-ratios $C_1=(\lambda_1,\cros{b}{c}{a}{d})$ $C_2=(\lambda_2,\cros{a}{e}{b}{f})$ and $C_3=(\lambda_3,\cros{c}{e}{d}{f})$. We already have a tropical curve that fulfills the cross-ratio $C_3$ and we consider the edges that contribute to it and the branches of $c,$ $d,$ $e,$ and $f.$  The new ends $a$ and $b$ must be added either on these branches or on these edges between them because they cannot be added on other edges as no other cross-ratios involve the markings $a$ and $b.$  The cross-ratios $C_1$ and $C_2$ now tell us that the end $a$ should be closer to the branches of $d$ and $e$ and the end $b$ should be closer to the branches of $c$ and $f$. If one of the ends was added at the edges contributing to the cross-ratio $C_3,$ this condition could not be fulfilled. The same holds if both ends were added at the same branch. So, they have to be added on different branches. If those branches were $c$ and $d,$ or $e$ and $f$, respectively, only one of the cross-ratios $C_1$ and $C_2$ could be fulfilled. If those branches were $c$ and $f$ or $d$ and $e$, none of these two cross-ratios could be fulfilled, so the only possibilities are to add the ends $a$ and $b$ at the branches $c$ and $e,$ or $d$ and $f,$ respectively and the cross-ratio conditions then can all be fulfilled. A sketch can be found in Figure \ref{case111} and also in the first line of Table \ref{table}. For the other cases, we argue similarly. Sometimes, we obtain different combinatorial types of tropical curves depending on the order of $\lambda_1,$ $\lambda_2$ and $\lambda_3.$
	
		\begin{figure}
		\centering
		\includegraphics[scale=0.4]{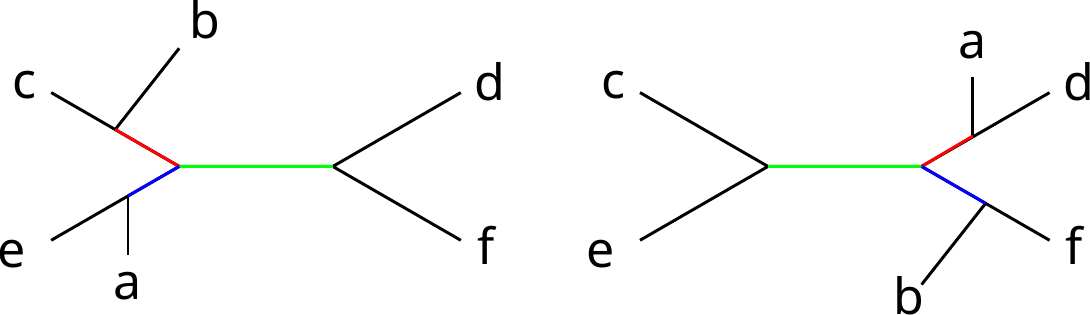}
		\caption{The first case in the proof of Theorem \ref{allcases}}
		\label{case111}
	\end{figure}
	
	 We see that if at least one cross-ratio of an inner triangle is not interpreted in the neighboring way, or all are interpreted in the neighboring way and do not fulfill the triangle inequalities, there are exactly two ways on how the ends $a$ and $b$ can be added to the existing tropical curve. If all cross-ratios are interpreted in the neighboring way and fulfill the triangle inequalities, there is just one way of how to add the new ends.
	 
	 Now, we consider the multiplicities in each case. We already saw in Lemma \ref{case21} that the multiplicity of the new tropical curve  with respect to $\pi_{\U(T')}$ is twice the multiplicity of the old tropical curve with respect to $\pi_{\U(T)},$ when we add an inner triangle with all cross-ratios interpreted the neighboring way that fulfills the triangle inequalities. We show that in every other case, the multiplicity stays the same by considering the multiplicity matrix.
	 
	Analogously to Lemma \ref{case22} and Lemma \ref{case23}, we can make column operations with the columns corresponding to the edges adjacent to the ends $a$ and $b$ or the one edge adjacent to both of these ends in some cases and obtain a matrix of block diagonal form, where the upper left block has determinant $\pm1$ and the lower right block is the multiplicity matrix of the tropical curve with the ends $a$ and $b$ removed, with respect to $\pi_{\U(T)}.$ Thus, the multiplicity does not change in all other cases.
	
	So, in total, when adjoining another inner triangle, we can construct the new tropical curves inductively and either the number of tropical curves fulfilling the corresponding cross-ratio conditions stays the same and the multiplicity increases by a factor of 2 for all already existing tropical curves with the new ends attached or the number of tropical curves doubles and the multiplicities stay the same. Hence, the degree $d_{\U(T)}$ of the map $\pi_{\U(T)}$ defined in Definition \ref{mappi} doubles when adding another inner triangle.
\end{proof}
\section{Maximal Cross-Ratio Degrees}\label{max}
\begin{remark}
	In the paper \cite{R24}, the open problem of finding the maximal cross-ratio degree for a certain number $n$ of markings, denoted $C(n)$ is introduced. There, this maximal degree is given for $n$ up to $n=6,$ and a lower bound for this maximal degree is given for all $n$ up to $n=14$. In \cite{Ma26}, it was recently shown that this lower bound for $n=7$ and $n=8$ is indeed sharp and all degrees from 0 up to 2 or up to 4, respectively, can appear. These new results are marked in red in the following table.
	
	\begin{center}
		\begin{tabular}{c||c|c|c|c|c|c|c|c|c|c|c|c}\label{maxtable}
			$n$&3&4&5&6&7&8&9&10&11&12&13&14\\\hline
			$C(n)$&1&1&1&2&\color{red}2&\color{red}4&\color{blue}6&\color{blue}10&$\ge$13&\color{OliveGreen}32&\color{violet}$\ge$32&\color{violet}$\ge$64
		\end{tabular}
	\end{center}
	The next Theorem provides a maximal degree for $n=9$ and $n=10,$ marked in blue in the preceding table.
\end{remark}
\begin{theorem}\label{max9}
	For a set of 6 cross-ratios $\{C_1,...,C_6\}$ using the markings $\U_9=\{S_1,...,S_6\}$, the map $\pi_{\U_9}$ from Definition \ref{mappi} can have every degree in $\{0,1,2,3,4,5,6\}$ and for a set of 7 cross-ratios $\{C_1,...,C_7\}$ using the markings $\U_{10}=\{S_1,...,S_7\},$ the map $\pi_{\U_{10}}$ can have every degree in $\{0,1,2,3,4,5,6,7,8,9,10\}.$
\end{theorem}
\begin{proof}
	The proof is done by computer using an program based on the algorithm of Goldner in \cite{Go21}, written in Oscar \cite{OSCAR}. The code is found in \cite{code}. 
	
	The function \texttt{deg\_by\_Goldner} takes a set $\U$ of four-element subsets of $[n]$ and computes its degree by the algorithm presented in \cite{Go21}. A part of the algorithm has been outsourced in the additional function \texttt{subsets\_checker}.  For faster calculation, a four-element subset $S$ of $[n]$ is presented as an integer, where in its bitwise representation, the $i$-th bit is a 1 if the number $i$ is contained in $S$ and 0 otherwise.
	
	The function \texttt{calculation} takes a value $n\in\mathbb{N}$ and returns a list of all possible cross-ratio degrees that can occur for cross-ratios using these $n$ markings. First, all four element subsets of $[n]$ are determined and put into a set  \texttt{set\_of\_four\_sets}. In the most efficient way, the next step would be to take all $n-3$ element subsets of \texttt{set\_of\_four\_sets} up to permuting the markings $\{1,...,n\}.$ Until now, we were not able to make use of all these symmetries. Instead, a first element \texttt{S1}$=\{1,2,3,4\}$ is fixed and four cases are now distinguished: we take one other set and have four possibilities on how this intersects \texttt{S1}. These sets are called \texttt{S2}$=\{1,2,3,n\},$ \texttt{S3}$=\{1,2,n-1,n\},$ \texttt{S4}$=\{1,n-2,n-1,n\}$ and \texttt{S5}$=\{n-3,n-2,n-1,n\}.$ So, the four cases are that we fix \texttt{S1} and one set of \texttt{S2}, \texttt{S3}, \texttt{S4} and \texttt{S5} and now take all $n-5$ element subsets of \texttt{set\_of\_four\_sets} without these two sets. For these, their degree is now calculated and all occurring degrees are put into the return value.  This is done for $n=9$ and $n=10.$ 
\end{proof}
\begin{remark}
	As described, this method is not optimal yet and it would be a future project to optimize the algorithm by making use of all symmetries. Then it probably would also be possible to calculate further results in a reasonable runtime. On a standard laptop the case $n=9$ took about 20 minutes and the case $n=10$ took about a week. In comparison, the case $n=8$ takes about 4 seconds and $n=11$ is estimated to take over 20 years.
\end{remark}
\begin{remark}
	In \cite[Theorem 5.3]{GGS20} it is shown that the count of embeddings of so-called Laman graphs in the sphere corresponds to cross-ratio degrees with even $n$. Some examples of high degrees are found in Table 1 of \cite{GGS20}. Hence, we can extend the known lower bounds from \cite{R24} in the cases of $n=12$ and $n=14,$ and thus also for $n=13.$ For $n=12,$ this bound is also found to be sharp, see the next Theorem \ref{n12}. This number is marked in green and the other new bounds are marked in purple in table \ref{maxtable}.
\end{remark}

\begin{lemma}[\cite{BELL25}]\label{Kapranovlemma}
	For a set of cross-ratios $\{C_1,...,C_9\}$ using the markings $\U_{12}=\{S_1,...,S_9\},$ the degree of the map $\pi_{\U_{12}}$ can not be larger than 32 if  there is a number $m \in [12]$ such that $m$ occurs in at least four of the sets in $\U_{12}.$
\end{lemma}
\begin{proof}
	This is a conclusion of Remark 4.4 in \cite{BELL25}.
\end{proof}
\begin{lemma}\label{intersect}
	Up to relabeling, there is only one possibility for a set $\U_{12}=\{S_1,...,S_9\}$ of four-element subsets of $[12]$ such that for any $i,j\in\{1,...,9\},$ the intersection $S_i\cap S_j$ is non-empty and each number in $[12]$ is contained in exactly three of the sets in $\U_{12}.$
\end{lemma}
\begin{proof}
	W.l.o.g, we can say that $S_1=\{1,2,3,4\}.$  As any two of the sets $S_1,...,S_9$ have to intersect and all numbers of $1,...,12$ are contained in exactly three of these $S_i,$ we find that there is no set $S_j$ with $j\in\{2,...,9\}$ such that there are two numbers if $1,2,3,4$ contained in it. Otherwise, there would be a set $S_i$ that has empty intersection with $S_1.$ Thus, we can assume w.l.o.g. that 1 is contained in $S_1,$ $S_2$ and $S_3,$ 2 is contained in $S_1,$ $S_4$ and $S_5,$ 3 is contained in $S_1,$ $S_6$ and $S_7,$ and 4 is contained in $S_1,$ $S_8$ and $S_9.$ So, up to the symmetry of relabeling, 
	$$S_1=\{1,2,3,4\}, S_2\supset\{1\}, S_3\supset\{1\}, S_4\supset\{2\}, S_5\supset\{2\}, S_6\supset\{3\}, S_7\supset\{3\}, S_8\supset\{4\}, S_9\supset\{4\}.$$
	
	Due to symmetry reasons, we find that for all $i,j \in \{1,...,9\},$ the intersection of $S_i$ and $S_j$ contains exactly one element. Thus, we can say that w.l.o.g (up to symmetry) $S_2=\{1,5,6,7\}$ and $S_3=\{1,8,9,10\}.$
	
	Up to now, we do not know in which sets the numbers 11 and 12 can be. As any two sets $S_i$ and $S_j$ must have an intersection of size 1, none of these numbers can be contained in both of $S_4$ and $S_5,$ $S_6$ and $S_7$ or $S_8$ and $S_9.$ Thus, we can say that w.l.o.g, 11 is contained in $S_4,$ $S_6$ and $S_8.$ As 12 can only be contained in 0 or 1 of the sets that also contain 11, we can say that w.l.o.g $12\in S_5$ and $12\in S_7.$ Assume, 12 is also contained in a set that contains 11, we can say it is $S_4.$
	
	Then, up to the symmetry of relabeling, 
	$$S_1=\{1,2,3,4\}, S_2=\{1,5,6,7\}, S_3=\{1,8,9,10\}, $$ $$S_4\supset\{2,11,12\}, S_5\supset\{2,12\}, S_6\supset\{3,11\}, S_7\supset\{3,12\}, S_8\supset\{4,11\}, S_9\supset\{4\}.$$
	
	The set $S_4$ also has to contain a fourth number which can only be one of $5,...,10.$ This now is a contradiction because if this fourth number was 5, 6 or 7, the intersection of $S_4$ with $S_3$ would be empty and if this number was 8, 9 or 10, the intersection of $S_4$ with $S_2$ would be empty. Thus, we find that $12\notin S_4$ and $12\in S_9.$
	
	Thus, we know now that up to the symmetry of relabeling, 
	$$S_1=\{1,2,3,4\}, S_2=\{1,5,6,7\}, S_3=\{1,8,9,10\}, $$ $$S_4\supset\{2,11\}, S_5\supset\{2,12\}, S_6\supset\{3,11\}, S_7\supset\{3,12\}, S_8\supset\{4,11\}, S_9\supset\{4,12\}.$$
	
	Two of the sets $S_4,...,S_9$ also have to contain the number 5 and they cannot intersect in the numbers 1, 2, 3, 4, 11, or 12 because we always need an intersection size of one as mentioned above. As the sets $S_4,...,S_9$ just differ by relabeling, we can say that $5\in S_4$ and $5\in S_7.$ Analogously we can think about 6  and 7 but they cannot be in one of the sets that already contains 5 because these numbers already appear together in the set $S_2$. W.l.o.g, we say that $6\in S_8$, then it follows that $7\in S_9,$ $6\in S_5$ and $7\in S_6.$ 
	
	So, we know $$S_1=\{1,2,3,4\}, S_2=\{1,5,6,7\}, S_3=\{1,8,9,10\}, $$ $$S_4\supset\{2,11,5\}, S_5\supset\{2,12,6\}, S_6\supset\{3,11,7\}, S_7\supset\{3,12,5\}, S_8\supset\{4,11,6\}, S_9\supset\{4,12,7\}.$$
	
	Now, just the numbers 8, 9 and 10 are not completely determined yet. W.l.o.g, we say that $8\in S_4$. Then, the only set left over where 8 can also be in is $S_9.$ In the same way, we can say that 9 is contained in $S_5$ and $S_6$ and 10 is contained in $S_7$ and $S_8.$ Thus, we are finished with $$S_1=\{1,2,3,4\}, S_2=\{1,5,6,7\}, S_3=\{1,8,9,10\}, $$ $$S_4\supset\{2,11,5,8\}, S_5\supset\{2,12,6,9\}, S_6\supset\{3,11,7,9\}, S_7\supset\{3,12,5,10\}, S_8\supset\{4,11,6,10\}, S_9\supset\{4,12,7,8\}.$$
	
	As we only had one possibility to choose the numbers for the sets up to symmetry such that all demanded conditions still hold, we have proven the lemma.
\end{proof}
\begin{theorem}\label{n12}
	For a set of 9 cross-ratios $\{C_1,...,C_9\}$ using the markings $\U_{12}=\{S_1,...,S_9\}$, the degree of the map $\pi_{\U_{12}}$ from Definition \ref{mappi} is at most 32.
\end{theorem}
\begin{proof}
	Using Lemma \ref{Kapranovlemma}, we find that the degree of such a set $\U_{12}=\{S_1,...,S_9\}$ can only be larger than 32 if all numbers $n\in[12]$ appear in at most 3 sets of $\U_{12}.$ By the pigeonhole principle we find that they then must appear in exactly 3 of these sets.
	
	We know from Lemma \ref{intersect} that up to symmetry there is just one possibility of how the set $\U_{12}=\{S_1,...,S_9\}$ of cross-ratio markings can look like if we assume that all sets $S_i$ and $S_j$ intersect nonempty and every number in $[12]$ appears exactly three times. For this set $\U_{12},$ we can calculate with Goldners algorithm that the degree of the map $\pi_{\U_{12}}$ is 19. Thus, for the rest of the proof, we can assume that all numbers in $[12]$ appear exactly in three of the sets in $\U_{12}$ and there are two sets $S_i$ and $S_j$ such that their intersection is empty and show that then, the degree cannot be larger than 32. W.l.o.g, we can say that $S_1=\{1,2,3,4\}$ and $S_2=\{5,6,7,8\}.$
	
	Now, we consider all possibilities up to symmetry on how to put each two of the numbers 1, 2, 3, and 4 in the sets $S_3,...,S_9.$ We find that there are only 17 possibilities to do so and they are found in Table \ref{1234}. An explanation on how to read these diagrams is found in section \ref{Venn}. The same can be done with each two of the numbers 5, 6, 7 and 8. We now also consider all possibilities up to symmetry on how to put each three of the numbers 9, 10, 11 and 12 into the sets $S_3,...,S_9.$ As there do not have to be three pairwise distinct sets, we also consider the cases where there is no set $S_i=\{9,10,11,12\}.$ These can be found in Table \ref{no-four-set}.
	The rest of the proof is done with the help of a computer program that can be found in \cite{code}.
	
	First, the program takes every of the $\binom{17}{2}$ possibilities to choose 2 of the each 17 cases for the numbers 1, 2, 3 and 4 and 5, 6, 7 and 8 and tries every possibility to combine them. Then it is counted how many spaces are left in each set $S_3,...,S_9$ according to these numbers, only the ones with the fitting sizes of the cases for the numbers 9, 10, 11 and 12 are considered. These numbers are then put into the sets $S_3,...,S_9$ in all possible ways and for the resulting set $\U_{12}=\{S_1,...,S_9\}, $ its degree is then calculated using the code with Goldners algorithm used in Theorem \ref{max9}.
	
	All possible degrees are put into the return value and we find that for these cases, all degrees from 0 up to 32 but 30 and 31 can appear. Thus, we have proven the Theorem.
	
\end{proof}

\begin{remark}
	So far, we do not have found any example of a set $\U_{12}=\{S_1,...,S_9\}$ of cross-ratio markings that has degree 30 or 31 but this could possibly exist. This is because, in this proof, we did not need to consider every possibility of how such a set of markings can look like and thus, a degree of 30 or 31 could possibly appear in the cases where there is a number in $\{1,...,12\}$ that appears in at least four of the sets in $\U_{12}.$
\end{remark}

\appendix 

\section{A Complete List of the Cases in the Proof of Theorem  \ref{allcases}}\label{cases}
In the following table, we list all the cases of how the cross-ratios corresponding to a newly added inner triangle can be interpreted and how the two new ends then are added to the existing tropical curve as described in the proof of Theorem \ref{allcases}, which completes this proof. There, the table is also explained further.
\label{Tabelle}
	\begin{longtblr}[
		caption = {All cases of how the ends $a$ and $b$ can be added},
		label = {table},
		]{	colspec = {|ccccc|},
		}
	\hline
	Int. of \color{green} $C_3$& Int. of \color{red}$C_1$& Int. of \color{blue}$C_2$&sketch&  order of $\lambda_1,$ $\lambda_2$ and $\lambda_3$\\
	\hline \hline
	d & d&d  & \includegraphics[scale=0.4]{111.pdf}&any order \\
	\hline
	d & d&n  & \includegraphics[scale=0.4]{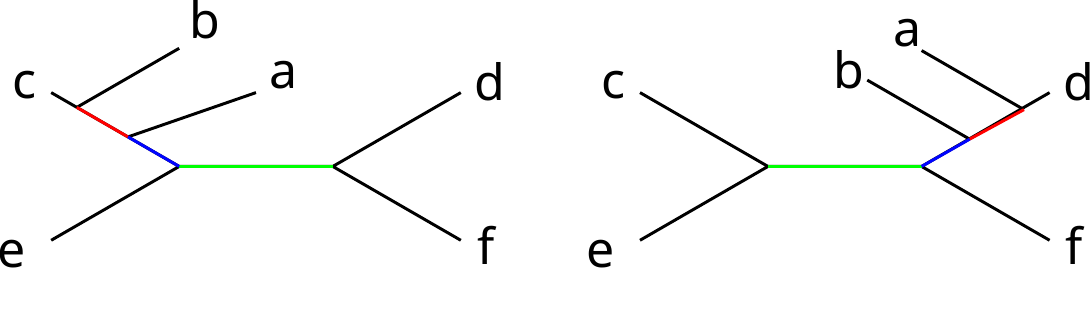}&any order \\
	\hline
	d & d&i  & \includegraphics[scale=0.4]{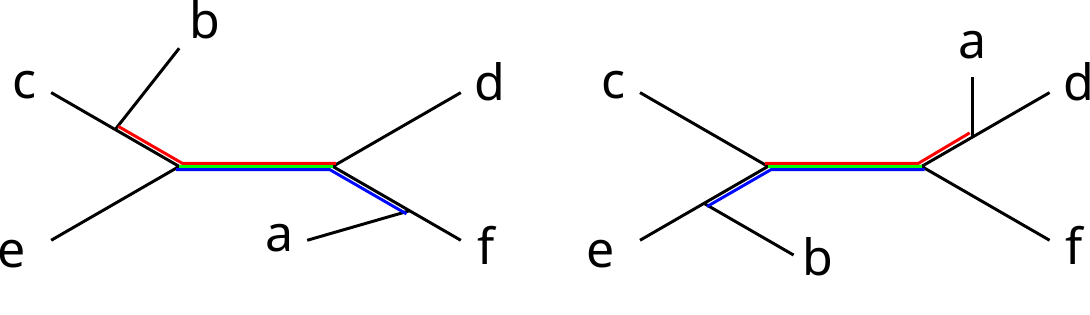}&$\lambda_3$ smallest \\
	&&&\includegraphics[scale=0.4]{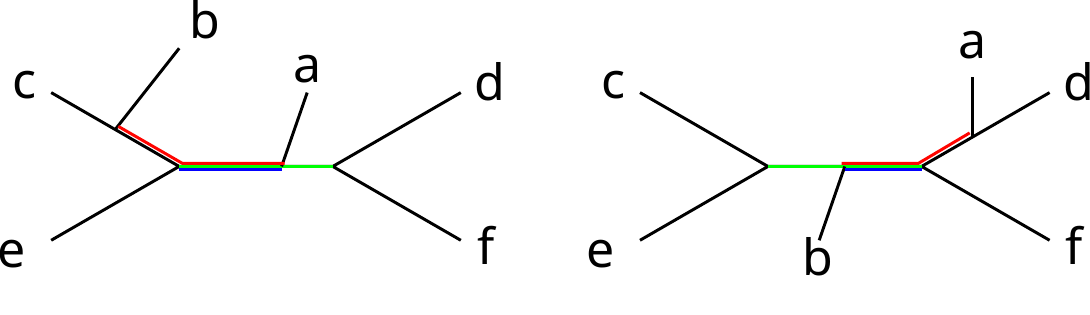}& $\lambda_2$ smallest\\
	&&&\includegraphics[scale=0.4]{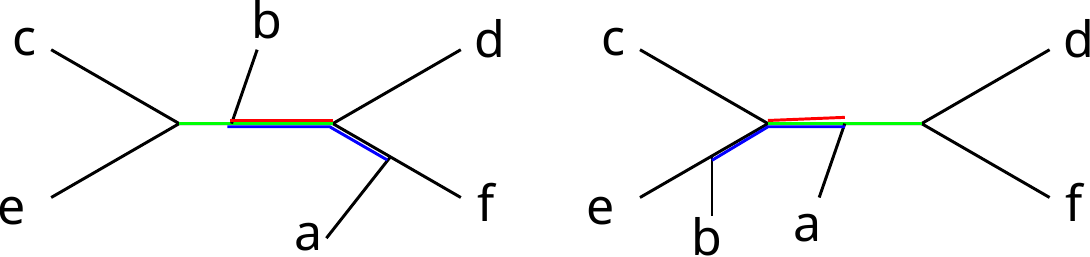}& $\lambda_1$ smallest\\
	\hline
	d & n&n  & \includegraphics[scale=0.4]{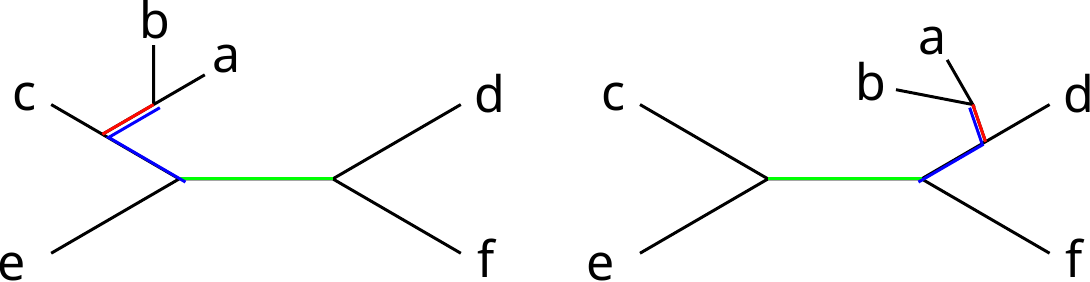}&$\lambda_1<\lambda_2$ \\
	& & &\includegraphics[scale=0.4]{122a.pdf}&$\lambda_2<\lambda_1$ \\	\hline
	d & n&i  & \includegraphics[scale=0.4]{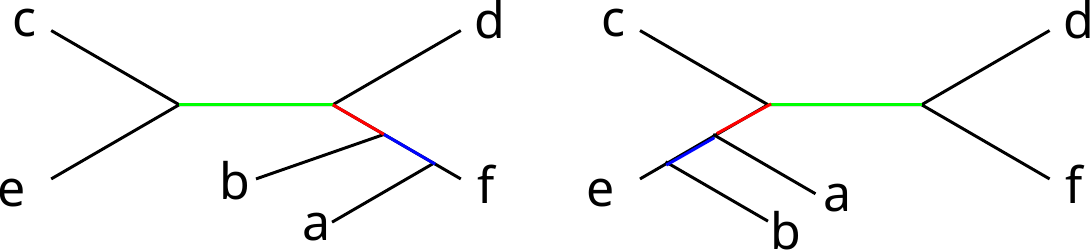}&any order \\
	\hline
	d & i&i  & \includegraphics[scale=0.4]{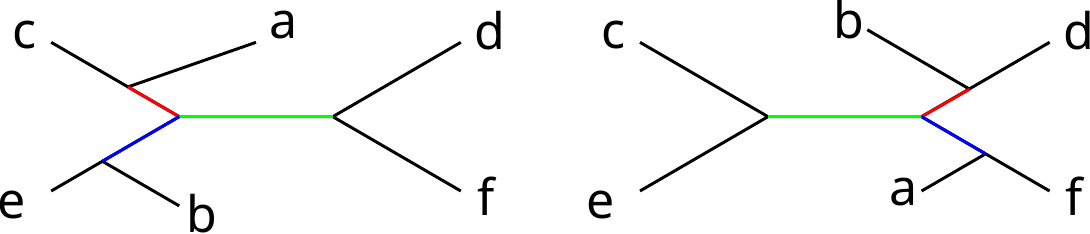}&any order \\
	\hline
	n & d&d  & \includegraphics[scale=0.4]{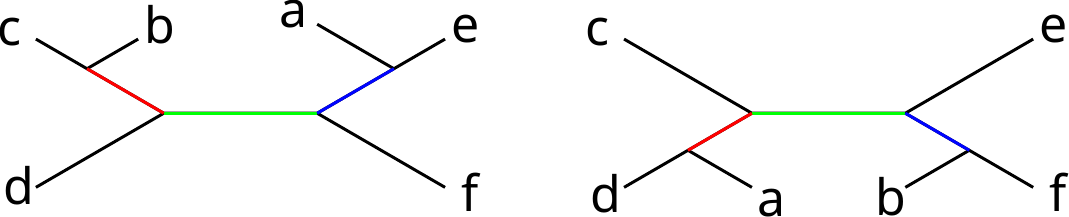}&any order \\
	\hline
	n & d&n  & \includegraphics[scale=0.4]{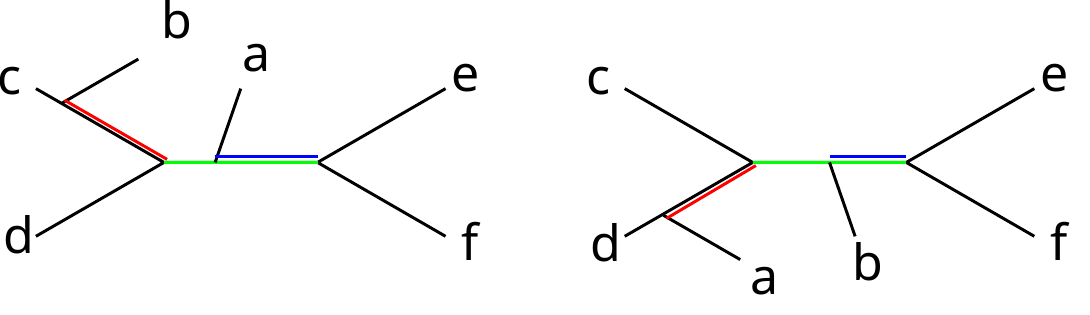}&$\lambda_2<\lambda_3$\\
	&&&\includegraphics[scale=0.4]{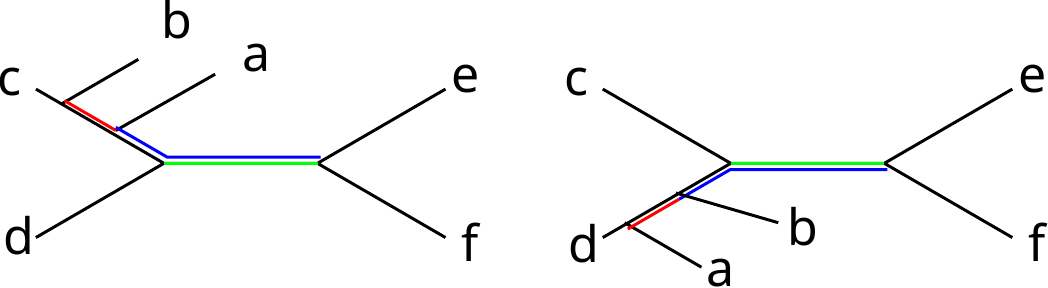}&$\lambda_2>\lambda_3$\\
	\hline
	n & d&i  & \includegraphics[scale=0.4]{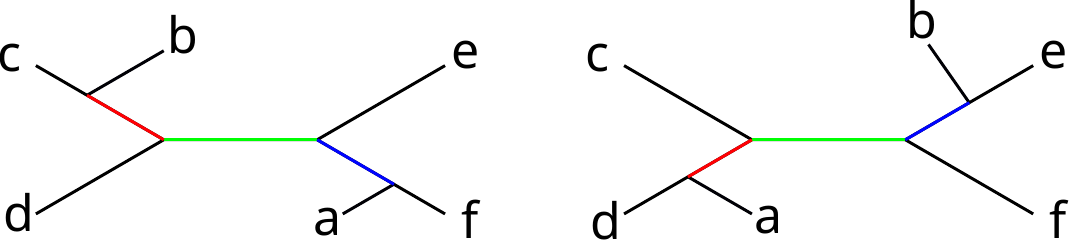}&any order \\
	\hline
	n & n&n  & \includegraphics[scale=0.4]{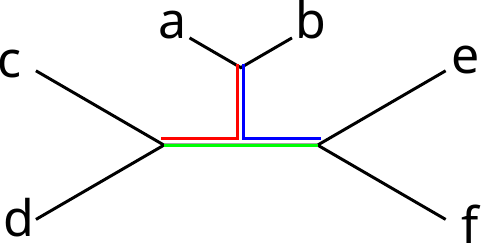}&\makecell{fulfill \\ triangle inequalities}\\
	&&&\includegraphics[scale=0.4]{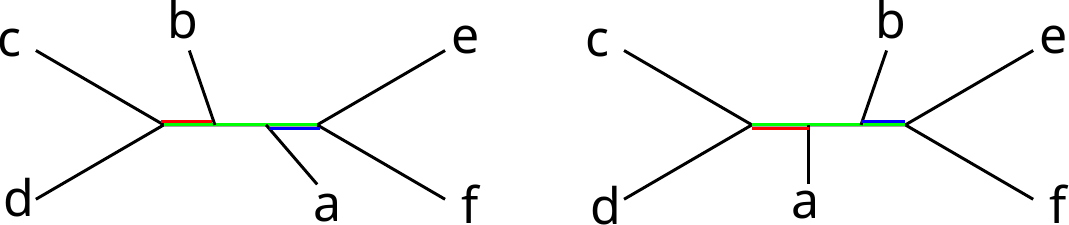}&$\lambda_3>\lambda_1+\lambda_2$\\
	&&&\includegraphics[scale=0.4]{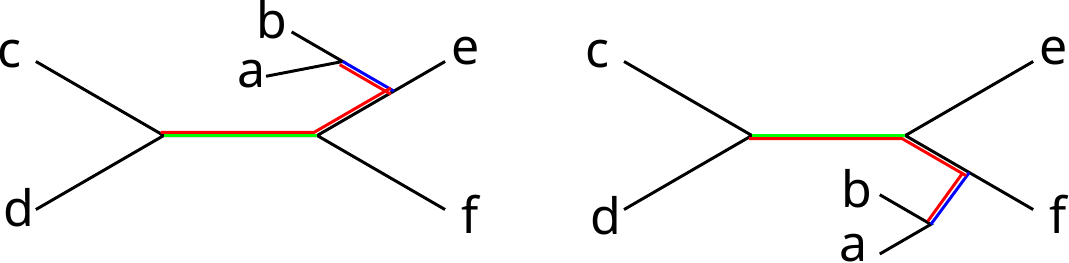}&$\lambda_1>\lambda_2+\lambda_3$\\
	&&&\includegraphics[scale=0.4]{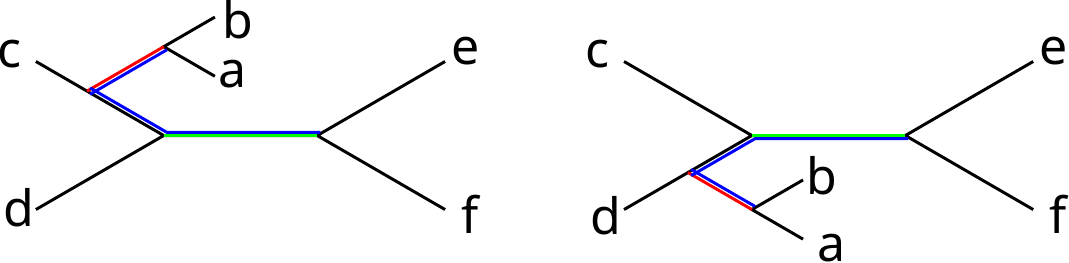}&$\lambda_2>\lambda_1+\lambda_3$\\
	\hline
	n & n&i  & \includegraphics[scale=0.4]{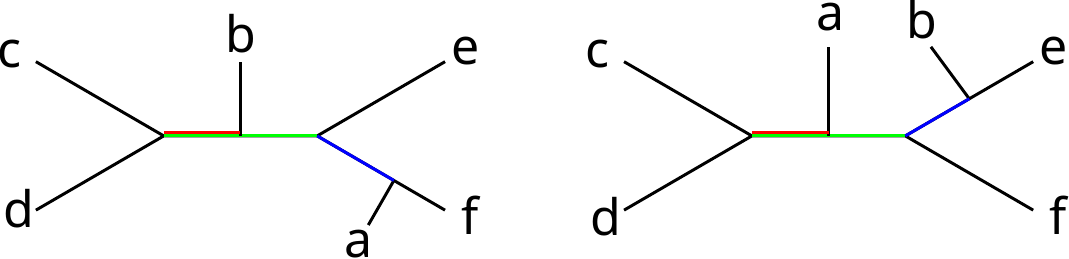}&$\lambda_1<\lambda_3$ \\
	&  && \includegraphics[scale=0.4]{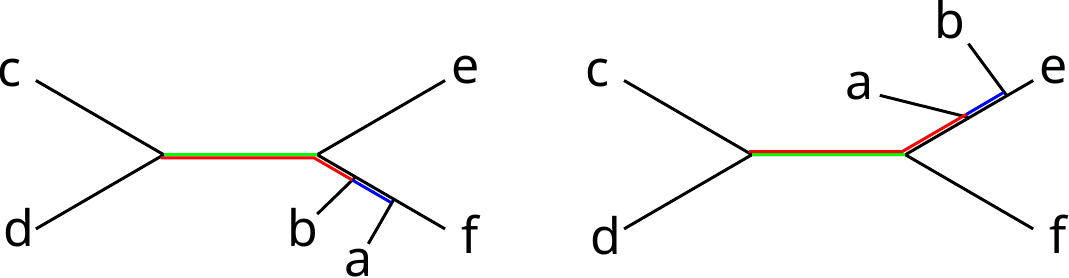}&$\lambda_1>\lambda_3$\\
	\hline
	n & i&i  & \includegraphics[scale=0.4]{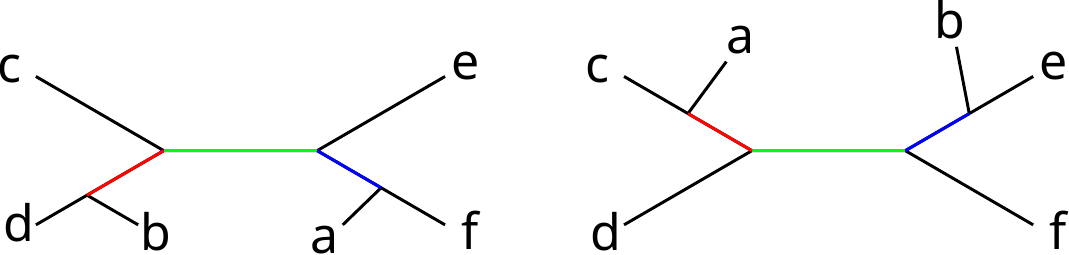}&any order \\
	\hline
	i & d&d  & \includegraphics[scale=0.4]{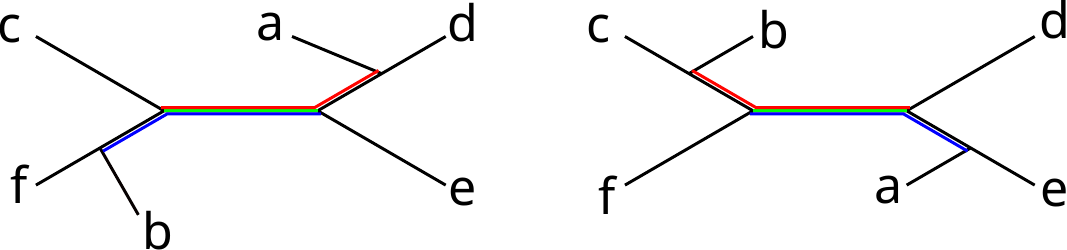}&$\lambda_3$ smallest\\
	&&&\includegraphics[scale=0.4]{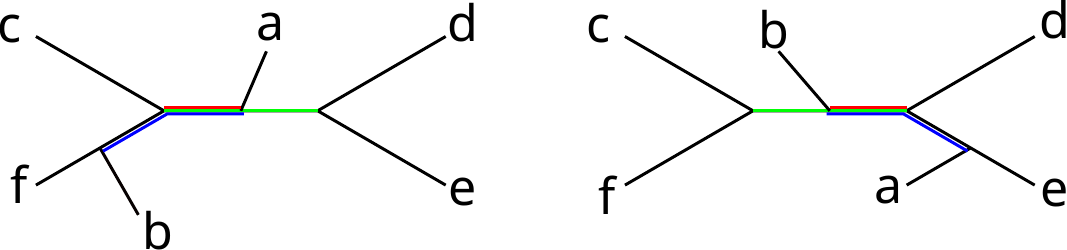}&$\lambda_1$ smallest\\
	&&&\includegraphics[scale=0.4]{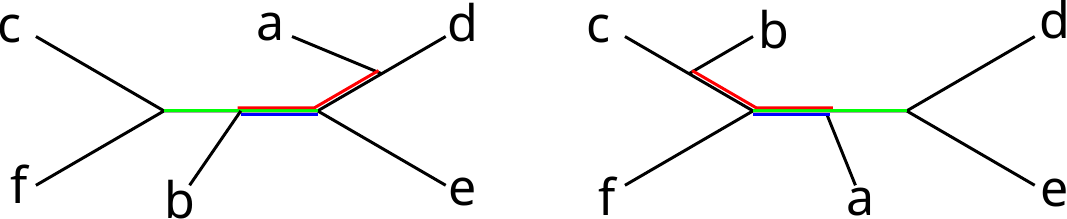}&$\lambda_2$ smallest\\
	\hline
	i & d&n & \includegraphics[scale=0.4]{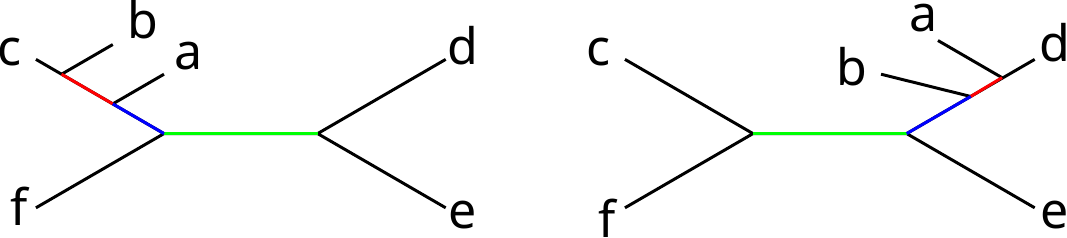}&any order \\
	\hline
	i & d&i  & \includegraphics[scale=0.4]{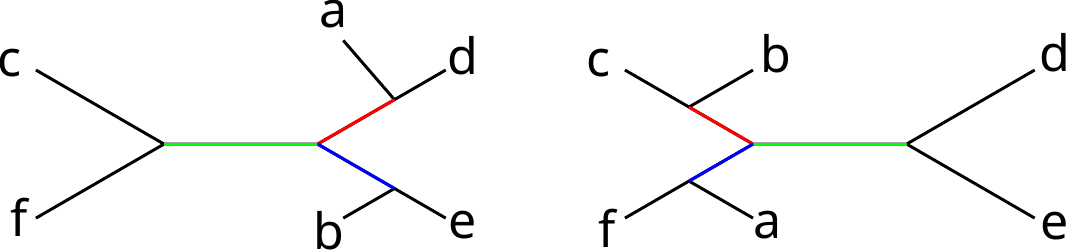}&any order \\
	\hline
	i & n&n  & \includegraphics[scale=0.4]{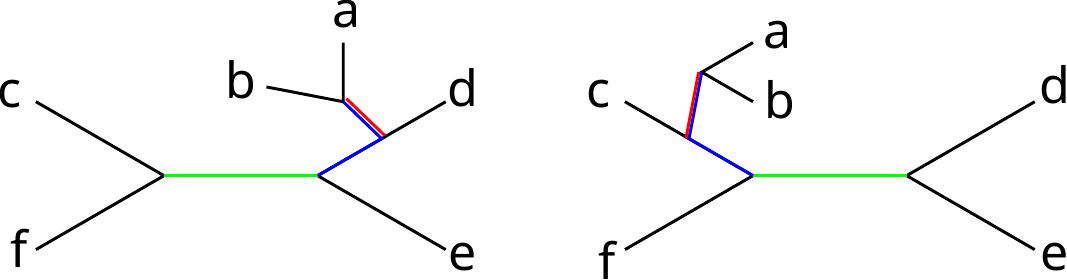}&$\lambda_1<\lambda_2$ \\
	& & &\includegraphics[scale=0.4]{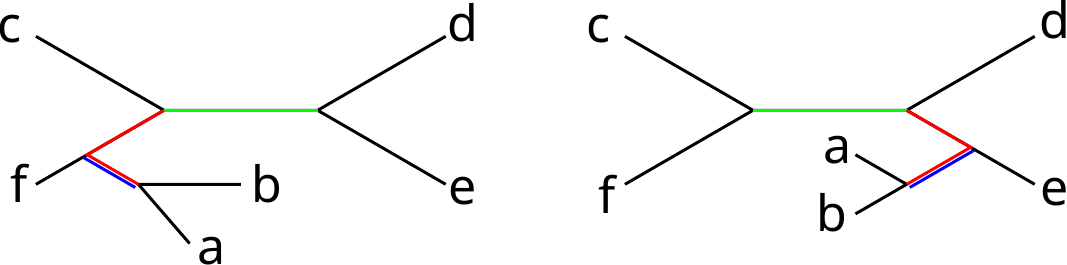}&$\lambda_2<\lambda_1$ \\
	\hline
	i & n&i  & \includegraphics[scale=0.4]{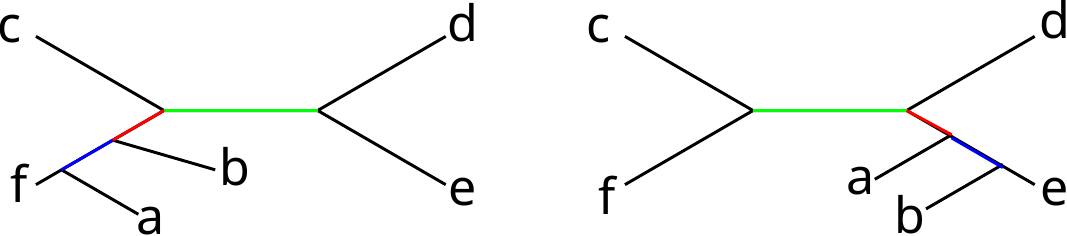}&any order \\
	\hline
	i & i&i  & \includegraphics[scale=0.4]{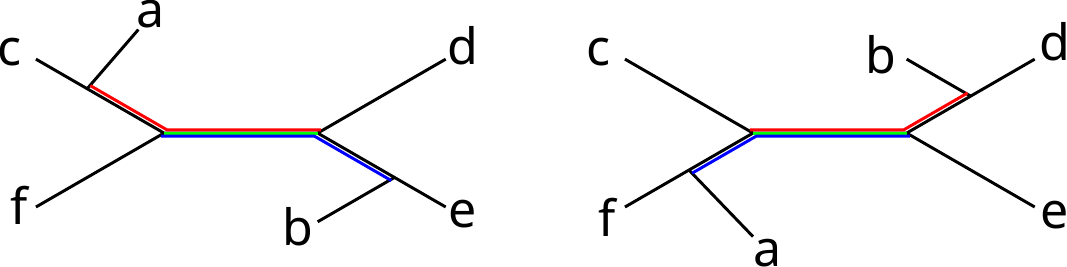}&$\lambda_3$ smallest \\
	&&&\includegraphics[scale=0.4]{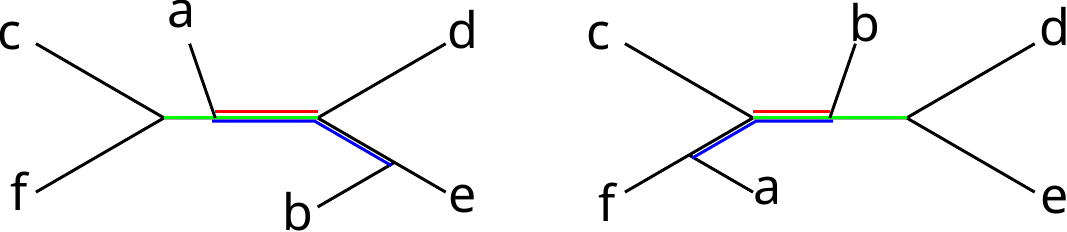}&$\lambda_1$ smallest \\
	&&&\includegraphics[scale=0.4]{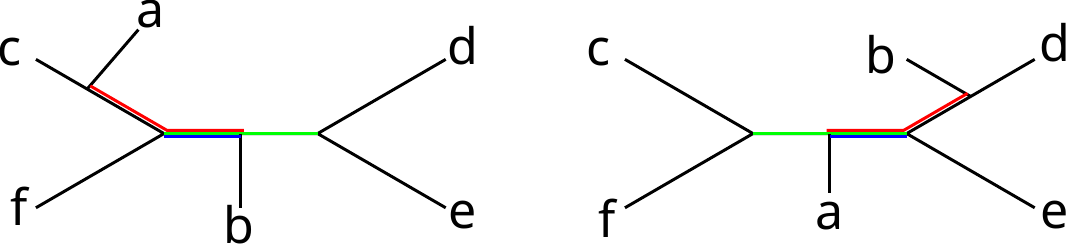}&$\lambda_2$ smallest \\ \hline
\end{longtblr}
\section{Venn Diagrams}\label{Venn}
In this section, we provide a list of all the cases of how the numbers 1,...,12 can be distributed in the proof of Theorem \ref{n12}. This is helpful for understanding the computer program used in the proof. First, we explain how to read the diagrams by which these cases are denoted.

\begin{remark}\label{Explain_Venn}
	A Venn diagram is a diagram that shows the logical relation between several sets. Most commonly used are Venn diagrams of three sets represented by three intersecting circles. Here, we use a generalization to four sets. 
	
	To depict all the relations that four sets can have to each other, we use a $4\times4$-square with one corner removed. Inside this shape, we can find four different $2\times4$- or $4\times2$-rectangles. These four rectangles represent our four sets. A sketch is depicted in Figure \ref{Venn_numbers}, where these four rectangles are drawn in yellow, magenta, cyan and gray and where they intersect, their colors are mixed. Now, we find that for each possibility of choosing one or more of the four sets, there is exactly one $1\times1$-square that represents this choice meaning that it is contained in the rectangles corresponding to the chosen sets.
	
	Another sketch for this can be seen in Figure \ref{Exp_venn}. The square labeled with "a" for example represents the possibility of just choosing the set of the yellow rectangle. The square labeled with "j" however represents choosing the sets of the magenta, gray and cyan rectangle. Thus, the letters "a", "c", "g" and "o" correspond to just choosing one set, the letters "b", "d", "f", "k", "l" and "n" represent choosing two of the sets, "e", "h", "j" and "m" represent choosing three sets and "i" represents the possibility of choosing all four. This can also be seen on the right of Figure \ref{Exp_venn}.

\begin{figure}[h]
	\centering
	\includegraphics[scale=0.5]{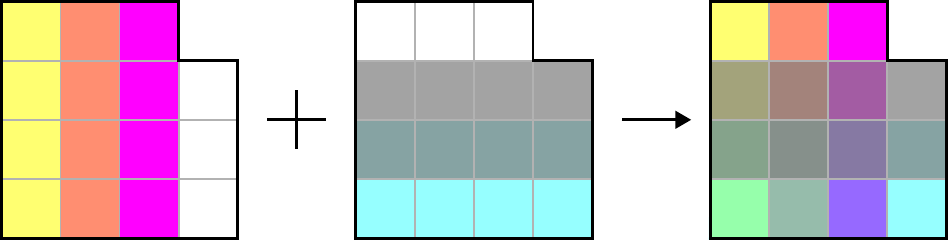}
	\caption{How to find the different sets in the Venn diagram of four sets, see also Remark \ref{Explain_Venn}.}
	\label{Venn_numbers}
\end{figure}

\begin{figure}[h]
	\centering
	\begin{subfloat}{\includegraphics[scale=0.7]{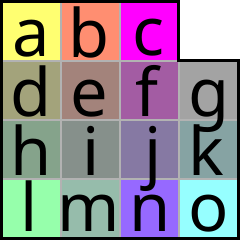}}
	\end{subfloat}
	\begin{subfloat}{\includegraphics[scale=0.7]{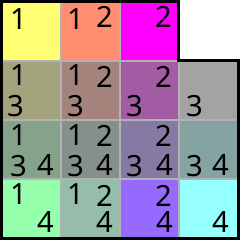}}
	\end{subfloat}
	\caption{A sketch for showing which squares correspond to choosing which sets in a Venn diagram of four elements, see also Remark \ref{Explain_Venn}.}
	\label{Exp_venn}
\end{figure}
\end{remark}

\begin{remark}
	Here, we use the Venn diagrams a little differently. We explain it first, using the numbers 1, 2, 3 and 4 but analogously this can be done for any other choice of four numbers. In the proof of Theorem \ref{n12}, we also use is for the numbers 5, 6, 7 and 8 as well as the numbers 9, 10, 11 and 12.
	
	In Table \ref{1234} and Table \ref{no-four-set}, the yellow rectangle corresponds to the number 1, the magenta rectangle to the number 2, the gray rectangle to the number 3 and the cyan rectangle to the number 4, see the right of Figure \ref{Exp_venn}. Then, there are numbers written in the $1\times1$-squares. Zeros are not written down for better overview. If in such a $1\times1$-square, there is the number $n$ and this square is part of the rectangles corresponding to the numbers $I\subset \{1,2,3,4\}$ it means that there are $n$ sets $S_{i_1},...,S_{i_n}$ that contain the numbers in $I$ but not the numbers in $\{1,2,3,4\}\setminus I.$
	
	For example, the first diagram in Table \ref{1234} represents the case that there is exactly one set that contains all the numbers 1, 2, 3 and 4, then there are two sets that contain 1, 2 and 3 but not 4 and 2 sets that contain the number 4 but no other of these numbers. So, we have that there are $S_i=\{1,2,3,4\},$ $S_j\supset\{1,2,3\},$ $S_k\supset\{1,2,3\},$ $S_l\supset\{4\}$ and $S_m\supset\{4\}.$
\end{remark}

\begin{remark}\label{n12list}
	The next table \ref{1234} provides a list of all possible cases up to symmetry of distributing the numbers 1, 2, 3 and 4 to the sets $S_3,...,S_9,$ when fixing $S_1=\{1,2,3,4\}$ and all numbers appear exactly three times. In terms of the diagrams that means that there is always a 1 in the square "i", the sum over all numbers in a $2\times4$- or $4\times2$-rectangle is always 3. As there are only eight sets $S_1,S_3...,S_9,$ the sum over all numbers in a diagram has to be at most eight.
	
	 Analogously, this list can also be used for all cases up to symmetry of distributing the numbers 5, 6, 7 and 8 to the sets $S_3,...,S_9,$ when fixing $S_2=\{5,6,7,8\}$ and all numbers appear exactly three times.
	
	For the case of finding all possibilities up to symmetry of distributing the numbers 9,10,11 and 12 to the sets $S_3,...,S_9,$ and all numbers appear exactly three times, we consider all diagrams od Table \ref{1234} where the sum over all numbers is at most seven and additionally all diagrams in Table \ref{no-four-set}. These correspond to the possibility that there is no set that contains all of these four numbers. For the latter, it thus also holds that the sum over all $2\times4$- or $4\times2$-rectangle is always 3 and the sum over all numbers of the diagram is at most seven but now, there always is a zero (no entry) for the square "i". We do not consider the cases, where the square "i" has a number larger than 1 because this means that we would have more identical sets $S_i=S_j$ that would lead to a degree $d=0$ in Theorem \ref{n12}. The markings a-g and 1-41 agree with the one used in the code.
\end{remark}
\begin{longtblr}[
	caption = {All cases of distributing four numbers in the sets $S_1,...,S_9$, when fixing one set, see Remark \ref{n12list} and Theorem \ref{n12}.},
	label = {1234},
	]{	colspec = {|cc|cc|cc|cc|cc|cc|},
	}
	\hline
	a&\includegraphics[scale=0.45]{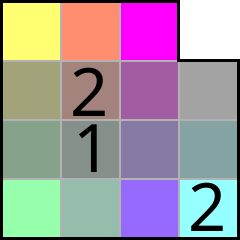}&b&\includegraphics[scale=0.45]{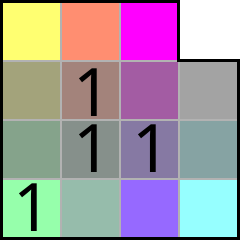}&c&\includegraphics[scale=0.45]{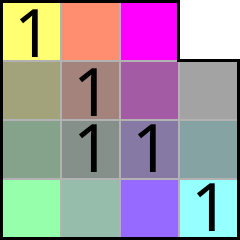}&d&\includegraphics[scale=0.45]{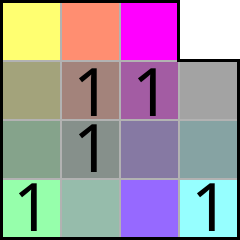}&e&\includegraphics[scale=0.45]{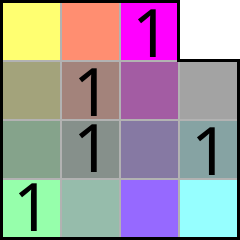}&f&\includegraphics[scale=0.45]{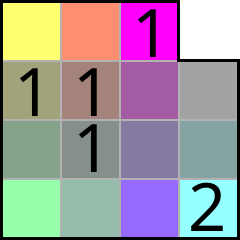}\\ \hline
	g&\includegraphics[scale=0.45]{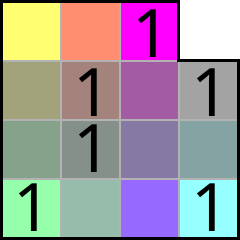}&h&\includegraphics[scale=0.45]{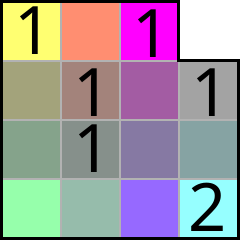}&i&\includegraphics[scale=0.45]{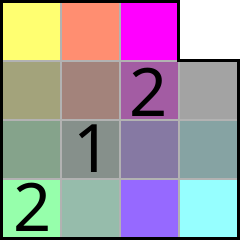}&j&\includegraphics[scale=0.45]{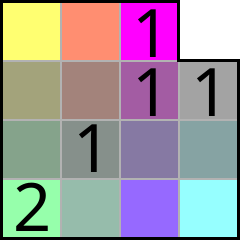}&k&\includegraphics[scale=0.45]{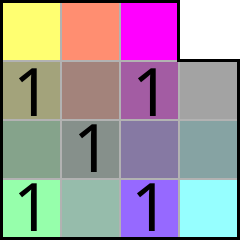}&l&\includegraphics[scale=0.45]{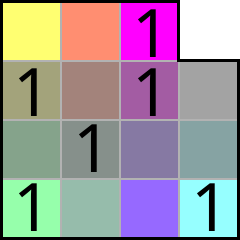}\\ \hline
	m&\includegraphics[scale=0.45]{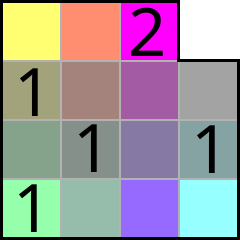}&n&\includegraphics[scale=0.45]{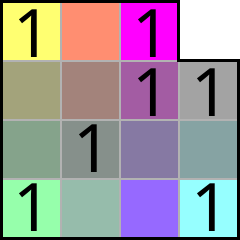}&o&\includegraphics[scale=0.45]{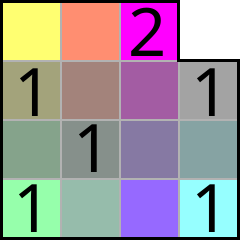}&p&\includegraphics[scale=0.45]{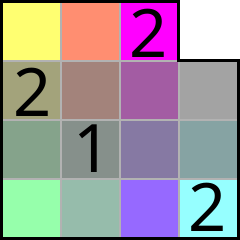}&q&\includegraphics[scale=0.45]{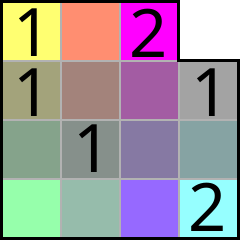}&& \\ \hline
\end{longtblr}

\begin{longtblr}[
	caption = {All cases of distributing four numbers in the sets $S_1,...,S_9$ when there is no set that contains all of them, see Remark \ref{n12list} and Theorem \ref{n12}.},
	label = {no-four-set},
	]{	colspec = {|cc|cc|cc|cc|cc|cc|}, 
	}
	\hline
	1&\includegraphics[scale=0.45]{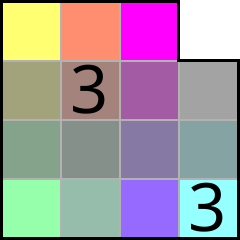}&2&\includegraphics[scale=0.45]{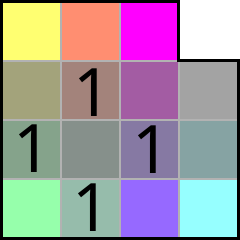}&3&\includegraphics[scale=0.45]{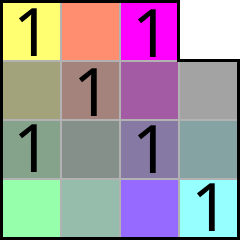}&4&\includegraphics[scale=0.45]{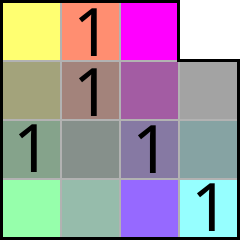}&5&\includegraphics[scale=0.45]{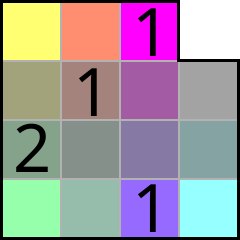}&6&\includegraphics[scale=0.45]{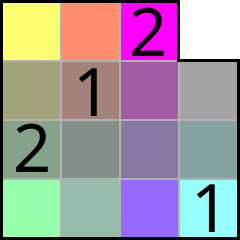}\\ \hline
	7&\includegraphics[scale=0.45]{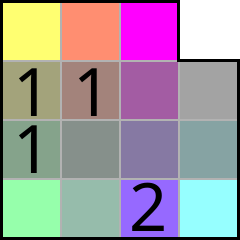}&8&\includegraphics[scale=0.45]{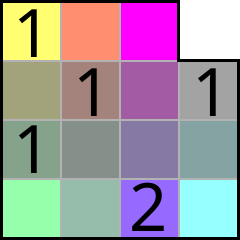}&9&\includegraphics[scale=0.45]{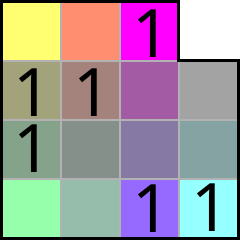}&10&\includegraphics[scale=0.45]{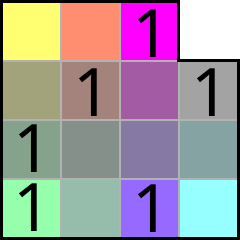}&11&\includegraphics[scale=0.45]{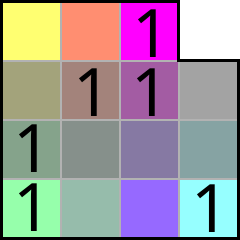}&12&\includegraphics[scale=0.45]{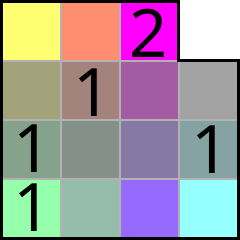}\\ \hline
	13&\includegraphics[scale=0.45]{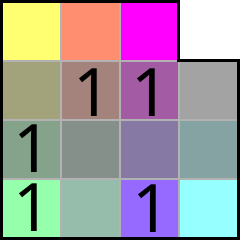}&14&\includegraphics[scale=0.45]{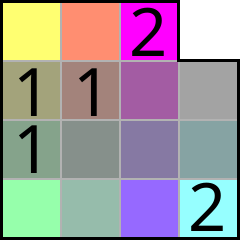}&15&\includegraphics[scale=0.45]{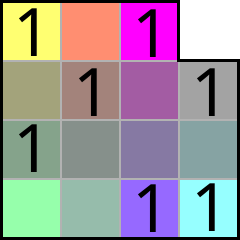}&16&\includegraphics[scale=0.45]{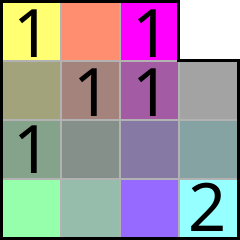}&17&\includegraphics[scale=0.45]{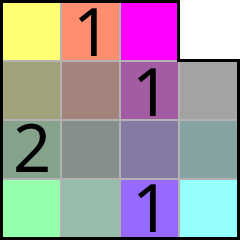}&18&\includegraphics[scale=0.45]{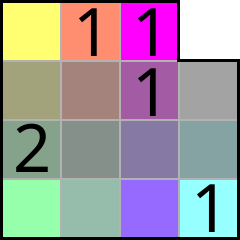}\\ \hline
	19&\includegraphics[scale=0.45]{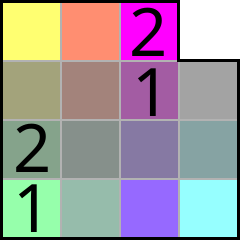}&20&\includegraphics[scale=0.45]{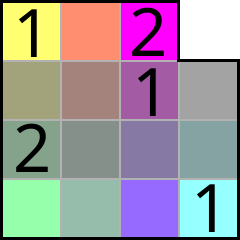}&21&\includegraphics[scale=0.45]{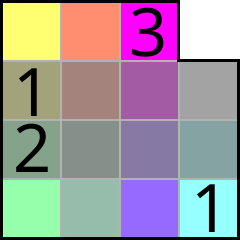}&22&\includegraphics[scale=0.45]{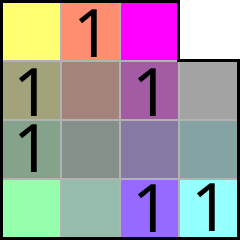}&23&\includegraphics[scale=0.45]{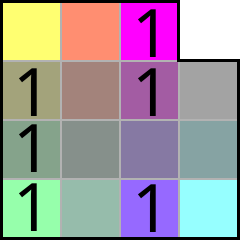}&24&\includegraphics[scale=0.45]{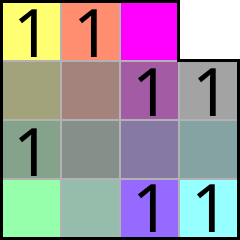}\\ \hline
	25&\includegraphics[scale=0.45]{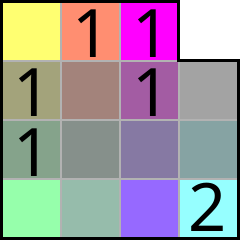}&26&\includegraphics[scale=0.45]{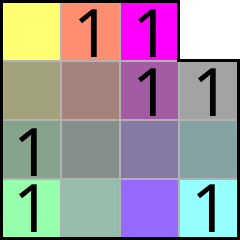}&27&\includegraphics[scale=0.45]{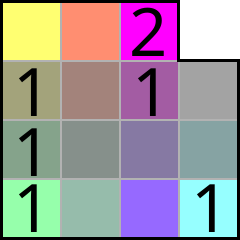}&28&\includegraphics[scale=0.45]{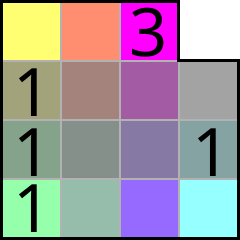}&29&\includegraphics[scale=0.45]{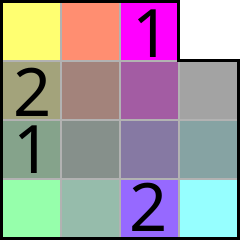}&30&\includegraphics[scale=0.45]{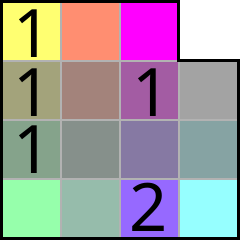}\\ \hline
	31&\includegraphics[scale=0.45]{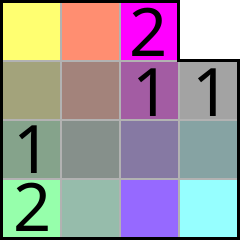}&32&\includegraphics[scale=0.45]{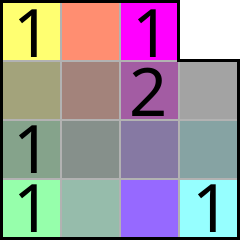}&33&\includegraphics[scale=0.45]{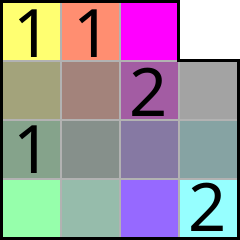}&34&\includegraphics[scale=0.45]{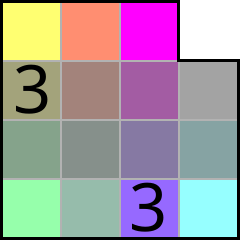}&35&\includegraphics[scale=0.45]{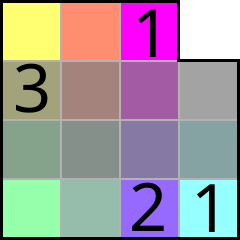}&36&\includegraphics[scale=0.45]{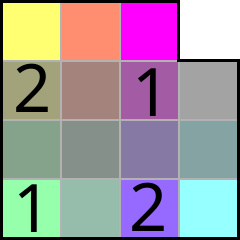}\\ \hline
	37&\includegraphics[scale=0.45]{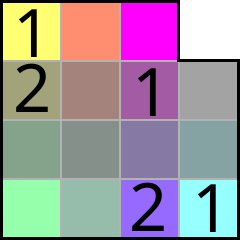}&38&\includegraphics[scale=0.45]{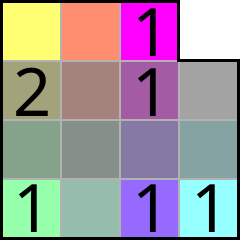}&39&\includegraphics[scale=0.45]{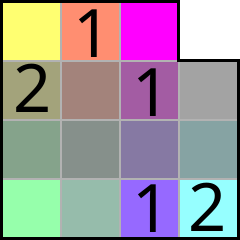}&40&\includegraphics[scale=0.45]{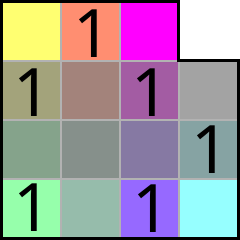}&41&\includegraphics[scale=0.45]{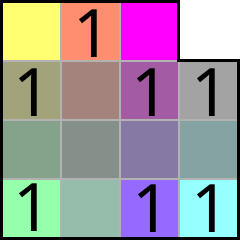}&&\\ \hline
\end{longtblr}
\section{Dual and Inverted Tropical Curves}\label{1stint}
This chapter presents a different approach on how to construct all tropical curves that fulfill a set of cross-ratio conditions given by a triangulation with all of them being interpreted in the dual way, see Definition \ref{defthreeint} and thus, also for obtaining the degree $d_\U$ as in Definition \ref{mappi}. Here, no new results are presented but rather than using a recursive argument, we provide a direct method of constructing the demanded tropical curves. 
\begin{definition}[Dual tropical curve]\label{defdual}
	For a triangulation of an $n$-gon as in Definition \ref{defthreeint}, there is an $n$-marked abstract rational tropical curve, corresponding to a point in a top-dimensional cone of $\mathcal{M}_{0,n}^{\mathrm{trop}}$, dual to it, which is given as follows. The underlying graph of the tropical curve can be constructed by first taking $n-3$ vertices corresponding to the $n-3$ triangles of the triangulation, where two vertices are connected if the corresponding triangles share an edge. If a triangle is adjacent to a an outer edge of the $n$-gon with marking $a$, we adjoin an unbounded edge with this marking to the corresponding vertex. In this way, the bounded edges of the underlying graphs correspond to the diagonals of the $n$-gon and the ends correspond to the marked outer edges of the $n$-gon. The lengths of the bounded edges of the tropical curve are given by the lengths of the corresponding diagonals of the triangulation.
\end{definition}
\begin{proposition}\label{dualcurve} Let $T$ be a triangulation of an $n$-gon as in Definition \ref{defthreeint}, with the cross-ratios $C_1,...,C_{n-3}$ all interpreted in the dual way. Then the dual tropical curve of $T$ is in $\pi_{\U(T)}^{-1}(C_1,...,C_{n-3}),$ see Definition \ref{mappi}.
\end{proposition}
\begin{proof}
	We show that every bounded edge of the tropical curve only contributes to the cross-ratio of the corresponding diagonal in the $n$-gon. Consider a diagonal $\mathcal{D}$ and the corresponding cross-ratio $C=(\lambda,\cros{a}{d}{b}{c}).$ This diagonal is adjacent to two triangles, we call them $V$ and $W$ and the corresponding vertices in the tropical curve $v$ and $w$. These two triangles also each have two other sides, that are either diagonals or boundary edges of the $n$-gon. By definition of the cross-ratio by a diagonal interpreted in the dual way, the markings $a,$ $b,$ $c$ and $d$ must be behind the four sides of the triangles $V$ and $W$ that are not the diagonal $\mathcal{D}$ such that behind each of these four sides is exactly one of these markings. An example can be seen in Figure \ref{Dualproof}.
	
	So, only the bounded edge of the tropical curve corresponding to the diagonal $\mathcal{D}$ can contribute to the cross-ratio $C,$ as the paths from $a$ to $c$ and from $b$ to $d$ have this edge in common but part at the vertices $v$ and $w.$ 
	
	In this way, we see that every edge corresponds to exactly one cross-ratio defined by $T$ and all these cross-ratios are fulfilled, so the dual tropical curve is in $\pi_{\U(T)}^{-1}(C_1,...,C_{n-3}).$
	
	\begin{figure}[h]
		\centering
		\includegraphics[scale=0.5]{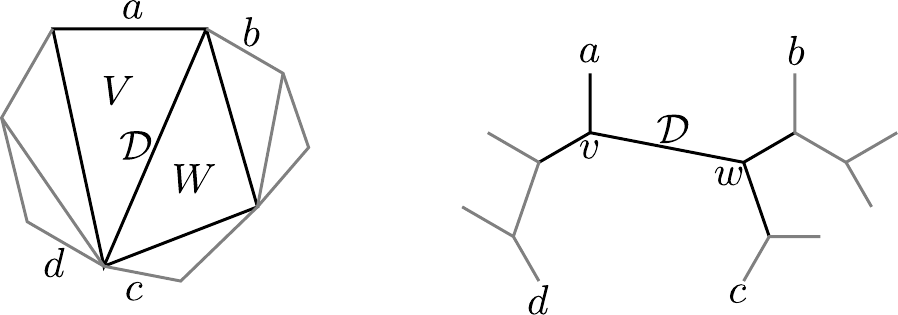}
		\caption{A diagonal in a triangulation and the corresponding edge in the dual tropical curve, see the proof of Proposition \ref{dualcurve}}
		\label{Dualproof}
	\end{figure}
	
\end{proof}

In general, there also can be other tropical curves fulfilling these same conditions. We now discuss how to find and construct all of these tropical curves.

\begin{example}\label{ex6a} Consider the triangulation $T$ depicted in Figure \ref{ex6aim}. The three cross-ratios defined by $T$ are $C_1=(\lambda_1,\cros{1}{2}{3}{6})$, $C_2=(\lambda_2,\cros{3}{4}{2}{5})$ and $C_3=(\lambda_3,\cros{5}{6}{1}{4})$. No matter how long the lengths of these cross-ratios are, as long as all are positive, there are always exactly two tropical curves  in $\pi_{\U(T)}^{-1}(C_1,C_2,C_3),$ see Definition \ref{mappi}. Both can be seen in Figure \ref{ex6aim}. The right one is the one dual to the triangulation. We call the left one the inverted one.  As here each edge contributes to exactly one cross-ratio, the lengths of the bounded edges are given by the lengths of these cross-ratios. Both tropical curves are of multiplicity 1 which can be seen easily, so we checked that, indeed, Theorem \ref{allcases} holds in this case.
	\begin{figure}
		\centering
		\includegraphics[scale=0.4]{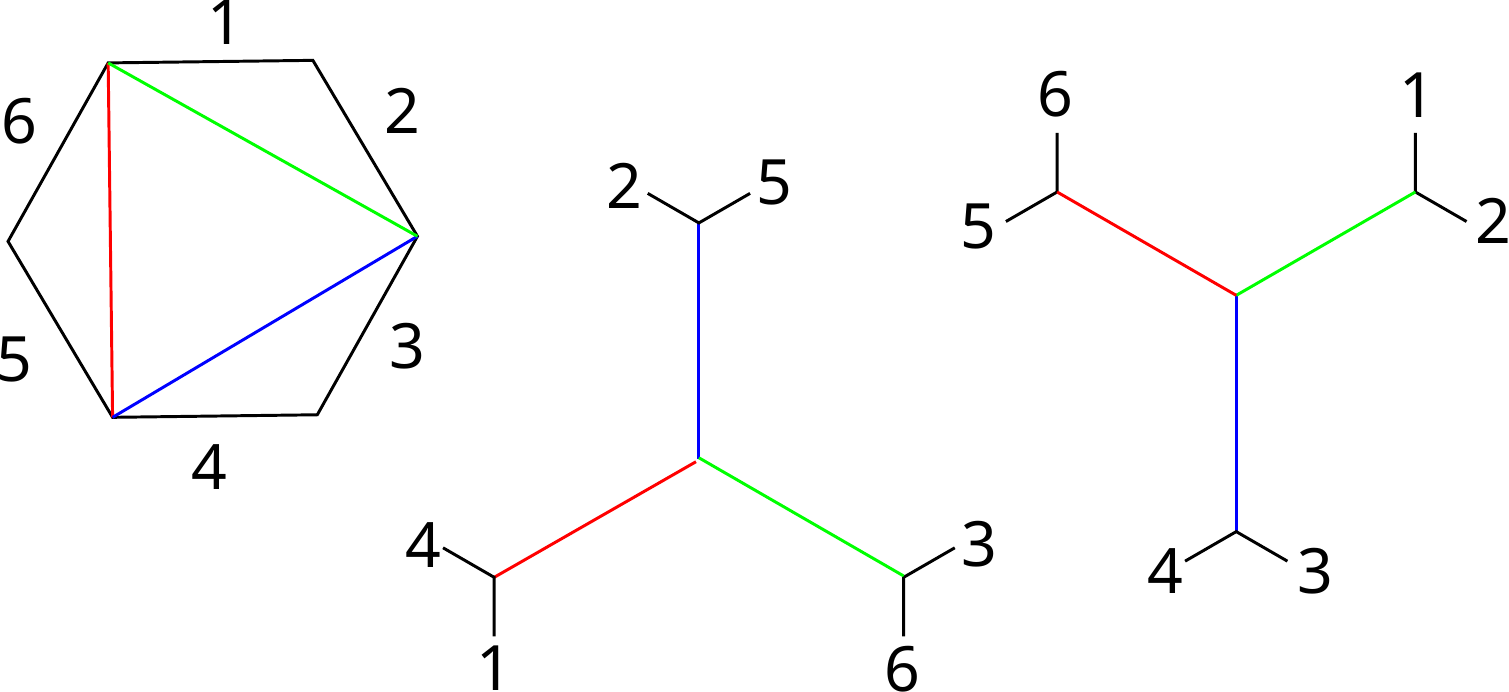}
		\caption{The two tropical curves that fulfill the cross-ratio conditions given by that triangulation of a hexagon in the dual interpretation, see Example \ref{ex6a}}
		\label{ex6aim}
	\end{figure}
\end{example}
\begin{example}\label{ex8}
	Now we consider a slightly larger example of a triangulation $T$ of an octagon, given in Figure \ref{ex8im}. The five cross-ratios defined by $T$ are $C_1=(\lambda_1,\cros{1}{2}{3}{8})$, $C_2=(\lambda_2,\cros{3}{4}{2}{5})$, $C_3=(\lambda_3,\cros{5}{6}{4}{7})$, $C_4=(\lambda_4,\cros{7}{8}{6}{1})$ and $C_5=(\lambda_5,\cros{1}{4}{5}{8})$. For any positive lengths of these cross-ratios, there are four abstract rational tropical curves in $\pi_{\U(T)}^{-1}(C_1,C_2,C_3,C_4),$. The first one is the one dual to the triangulation. This one and the other three can be seen in Figure \ref{ex8im}, where the lengths of the bounded edges are determined by the lengths of the diagonals. In the first and in the third row, we again have a one-to-one correspondence between bounded edges and cross-ratio lengths. In the other cases, the two branches that are depicted with the same shading pattern are glued together as far as determined by the cross-ratio lengths. For example, if $\lambda_2,$  depicted in light blue, is smaller than $\lambda_3,$ depicted in green, then the ends $4$ and $7,$ or $5$ and $6$ in the other case, meet in a vertex and the end $3$ or $2,$ respectively is attached to the branch of these ends. Otherwise, if $\lambda_2>\lambda_3$, the ends $3$ and $4$, and the ends $2$ and $5$ meet in a vertex and the ends $7$ and $6$ are attached to the branch of them. Analogously, it works for $\lambda_1$ and $\lambda_4.$ For $\lambda_2<\lambda_3$ and $\lambda_1<\lambda_4$, we obtain the tropical curves depicted in Figure \ref{ex8bim}.
	\begin{figure}
		\centering
		\includegraphics[scale=0.35]{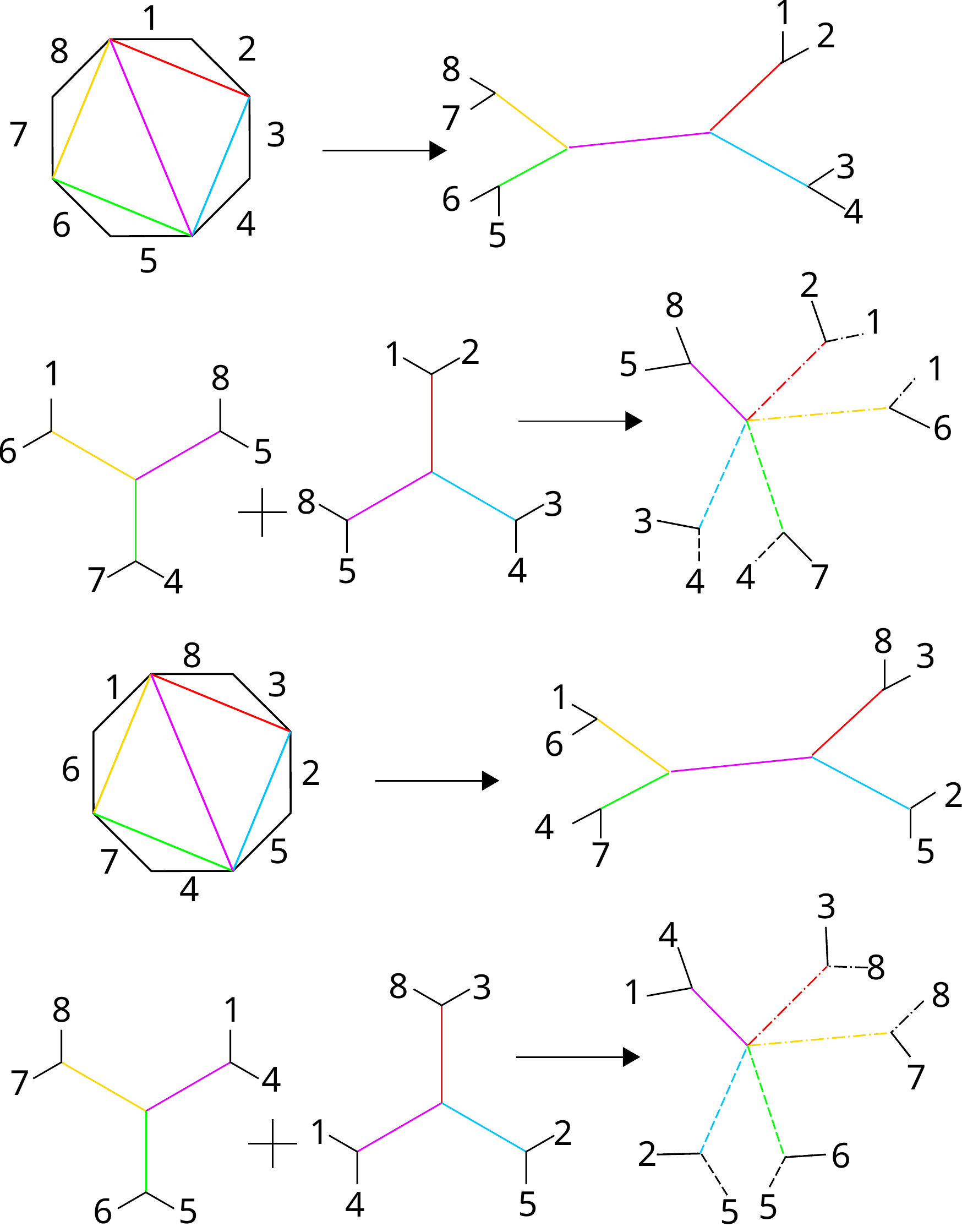}
		\caption{Constructing all tropical curves that fulfill the cross-ratio conditions given by an octagon with two inner triangles, see Example \ref{ex8}}
		\label{ex8im}
	\end{figure}
	\begin{figure}
		\centering
		\includegraphics[scale=0.35]{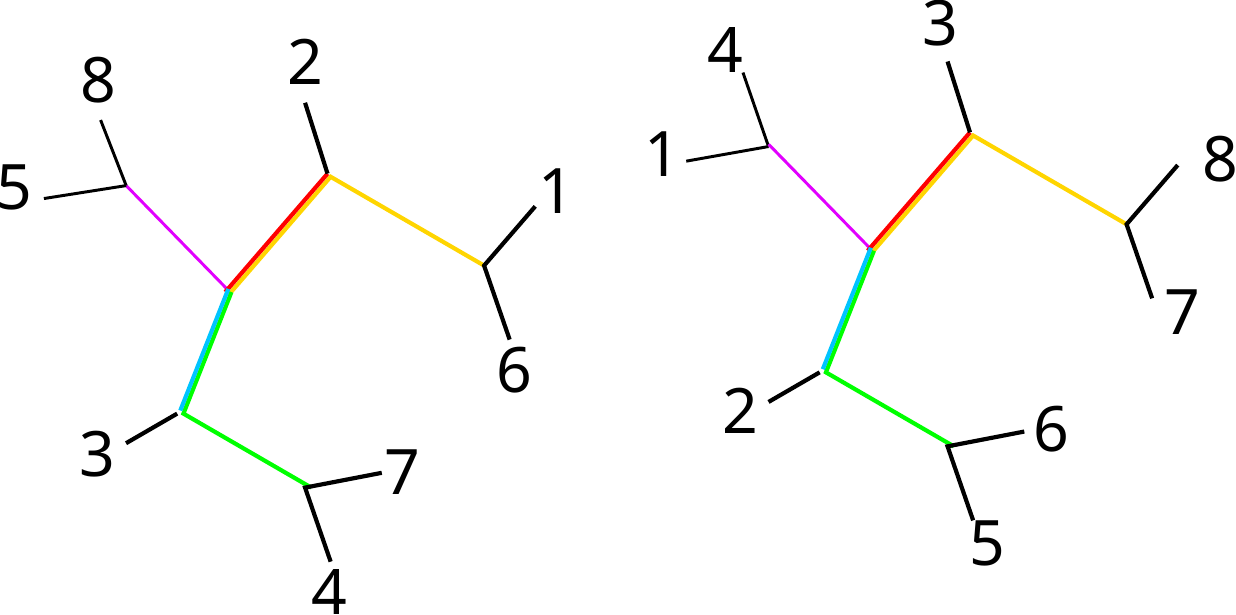}
		\caption{Two examples of tropical curves that fulfill the cross-ratio conditions of Example \ref{ex8}}
		\label{ex8bim}
	\end{figure}
\end{example}
\begin{construction}[Inverting $n$-gons without outer triangles]\label{constotallyinverted}
	For a triangulation without outer triangles, we can construct a tropical curve that we call \emph{totally inverted} by the following process. For all vertices of the $n$-gon, where a diagonal is ending, we change the markings of the two adjacent edges. As there are no outer triangles, this means that exactly every second corner has diagonals ending in it, so all edge markings swap places with either their right or left neighbor.
	Now, we draw the dual tropical curve to this triangulation with the new markings. The lengths of the bounded edges of the totally inverted tropical curve are again given by the lengths of the corresponding diagonals. As the underlying triangulation when ignoring the markings stays the same, the combinatorial type without the markings, and the lengths of the edges of the dual tropical curve and the totally inverted tropical curve are the same. 
	
\end{construction}
\begin{example}
	Examples of these totally inverted tropical curves are seen in the middle of Figure \ref{ex6aim}, in the third row on the right of Figure \ref{ex8im} and at the top right of Figure \ref{Inv_n-eck}. 
\end{example}
\begin{lemma}\label{lemmatotal}
	Let $T$ be a triangulation of an $n$-gon as defined in Definition \ref{defthreeint}. The dual tropical curve of Definition \ref{defdual} and the totally inverted tropical curve of Construction \ref{constotallyinverted} fulfill the same cross-ratio conditions, in the dual interpretation, see Definition \ref{defthreeint}, which are given by $T$.
\end{lemma}
\begin{proof}
	Analogously to the proof of Proposition \ref{dualcurve}, we see that for the totally inverted tropical curve every bounded edge contributes exactly to one cross-ratio. Also, as the triangulation when swapping the markings of the boundary edges does not change, the tropical curve without the marked ends does not change and we obtain a bijection between the totally inverted tropical curves and the dual tropical curves, where every bounded edge of the dual tropical curve and the corresponding one of the totally inverted tropical curve contribute to the same cross-ratio because when changing the markings of the edges of the $n$-gon as described, a cross-ratio of the form $(\lambda,\cros{a}{b}{c}{d})$ changes to $(\lambda,\cros{c}{d}{a}{b}),$ which is tropically the same.
\end{proof}
\begin{definition}[Dual and inverted tropical curves of an inner triangle]
	Let $T$ be a triangulation of an $n$-gon without outer triangles, where all cross-ratios are interpreted in the dual way. For an inner triangle consisting of the three cross-ratios $C_1=(\lambda_1,\cros{a}{b}{c}{f})$, $C_2=(\lambda_2,\cros{c}{d}{b}{e})$ and $C_3=(\lambda_3,\cros{e}{f}{a}{d}),$
	we consider the forgetful map $ft_{[n]\setminus \{a,b,c,d,e,f\}}.$ We say that the image under this map of the dual tropical curve of this triangulation is the \emph{dual tropical curve to this inner triangle} and that the image under this map of the totally inverted tropical curve of this triangulation is the \emph{inverted tropical curve to this inner triangle}. 
\end{definition}
\begin{remark}
	The dual and inverted tropical curves of an inner triangle look like the dual and the totally inverted tropical curve of a triangulation of a hexagon with markings $\{a,b,c,d,e,f\}$ as in Example \ref{ex6a} and Figure \ref{ex6aim}.
\end{remark}
\begin{construction}[Partially inverted tropical curve]\label{partialcurve}
	We can construct an $n$-marked tropical curve by gluing together the dual or inverted tropical curves to all inner triangles of this triangulation as follows. 
	
	First, we pick for each inner triangle if we want to take its inverted or its dual tropical curve. We say that two inner triangles have the same orientation if we chose the same for both. Now, consider two adjacent inner triangles --- these share a diagonal, and consider the two inverted or dual tropical curves to these inner triangles. As the inner triangles share a diagonal, these two 6-marked tropical curves share a bounded edge $e$. Now, we glue these two tropical curves together by identifying $e$ so that it is compatible with the four markings of the corresponding cross-ratio $C$ (remember, for these 6-marked tropical curves, every bounded edge contributes to exactly one cross-ratio), and then gluing the paths from the adjacent vertices of $e$ to the ends with markings in $C$. If both triangles have the same orientation, we obtain the dual or totally inverted tropical curve to an octagon with two inner triangles.
	
	This can also be seen in Figure \ref{ex8im}, where the branches with the same dashed pattern are glued together.
	
	Now, we continue to glue together all other dual or inverted tropical curves of the other inner triangles in order to obtain an $n$-marked tropical curve. This process also works when several $6$-marked tropical curves are already glued together because we can still identify the edge (or now possibly more edges) contributing to the common cross-ratio. We call this resulting tropical curve a \emph{partially inverted} tropical curve.
\end{construction}
\begin{remark}\label{2dcurves}
	If a triangulation of an $n$-gon without outer triangles has $d$ inner triangles, we obtain that $d=\frac{n}{2}-2$ and as for each inner triangle we can choose if we want to consider its dual or its inverted tropical curve, we obtain $2^d$ different tropical curves with the just described construction. These are all different tropical curves as they have different images under the forgetful maps $ft_{[n]\setminus \{a,b,c,d,e,f\}}.$ 
	
	If we choose the dual tropical curve for every inner triangle, we obtain the dual tropical curve of this triangulation and if we choose the inverted tropical curve for every inner triangle, we obtain the totally inverted tropical curve of this triangulation.
\end{remark}
\begin{proposition}\label{partialfulfills}
	Let $T$ be a triangulation of an $n$-gon without outer triangles, where all cross-ratios $C_1,...,C_{n-3}$ are interpreted in the dual way. A partially inverted tropical curve as constructed in Construction \ref{partialcurve} is in $\pi_{\U(T)}^{-1}(C_1,...,C_{n-3}),$ see Definition \ref{mappi}.
\end{proposition}
\begin{proof}
	By gluing together some of the smaller tropical curves, cross-ratio conditions that were before fulfilled by these tropical curves are still fulfilled by the larger glued tropical curve because no lengths of bounded edges change, they can just be subdivided into several edges that sum up to the same length and also no markings are affected. Thus, as all cross-ratio conditions that are fulfilled by one or two of the 6-marked dual or inverted tropical curves of the inner triangles, all are still fulfilled after gluing them together. Also, all cross-ratios are used because all are part of some inner triangle.
	So, the resulting tropical curve is in $\pi_{\U(T)}^{-1}(C_1,...,C_{n-3}).$
\end{proof}

\begin{example}[Revisited \ref{ex6a}]
	In Figure \ref{ex8im}, we now observe for which inner triangles we chose the inverted or the dual tropical curve. With choosing the dual tropical curve for both inner triangles, we obtain the dual tropical curve for the triangulation, which is seen in the first row. In the second row, we choose the dual tropical curve for the right inner triangle and the inverted one for the left one. In the third row, for both inner triangles, the inverted tropical curve is chosen and thus, the resulting tropical curve is the totally inverted tropical curve of the triangulation. And finally, in the last row, for the left inner triangle the dual tropical curve is chosen and for the right one the inverted one.
\end{example}
\begin{example}\label{ex12}
	\begin{figure}
		\centering
		\includegraphics[scale=0.45]{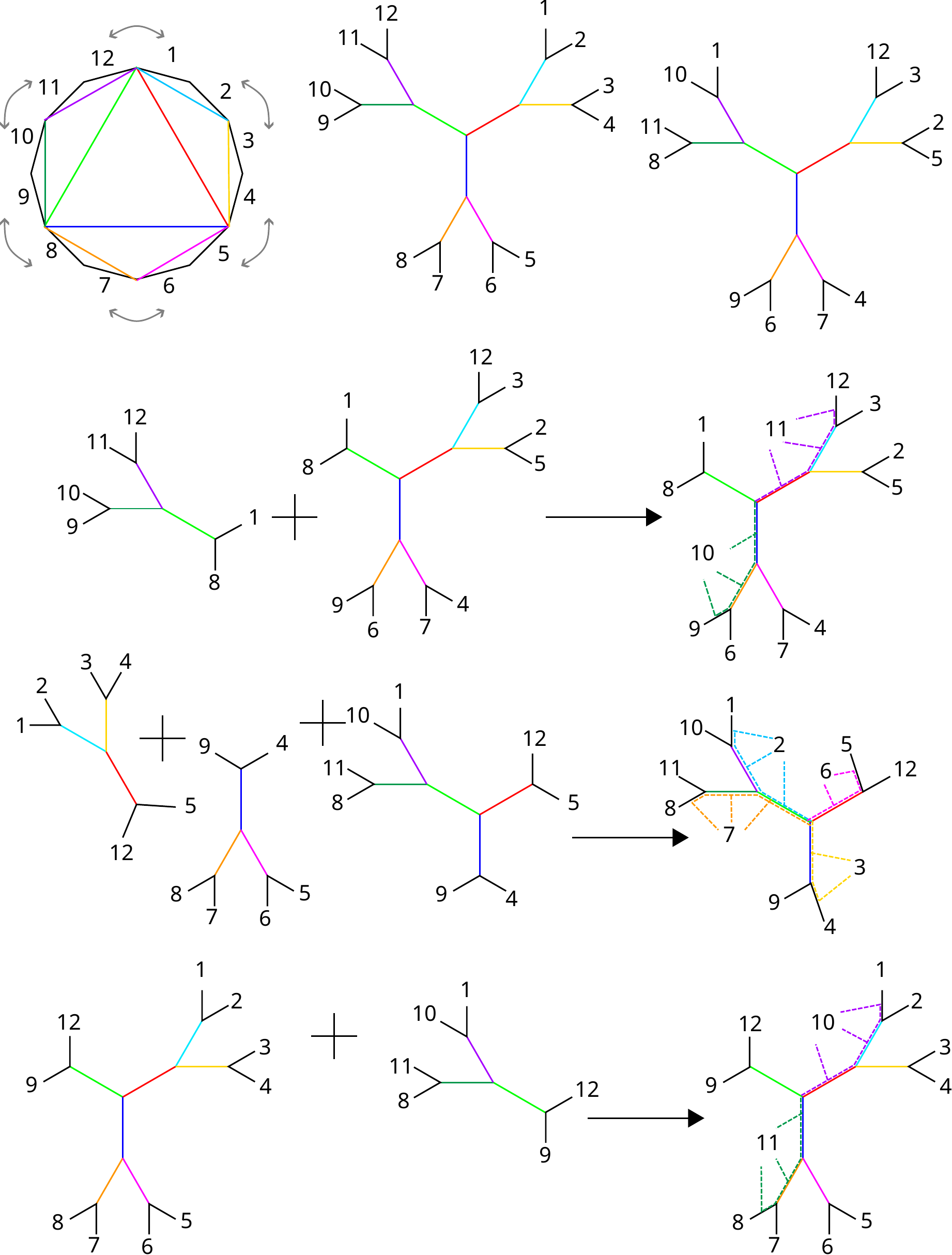}
		\caption{Constructing some partially inverted tropical curves that fulfill the cross-ratio conditions given by a triangulation of a dodecagon, see Example \ref{ex12}}
		\label{Inv_n-eck}
	\end{figure}
	Another sketch of constructing partially inverted tropical curves is found in Figure \ref{Inv_n-eck}, where in the first row we see the dual and the totally inverted tropical curve of the triangulation.
	In the other rows, we see three examples of partially inverted tropical curves, where different orientations are chosen. On the left of these sketches, for adjacent triangles with the same orientation, the corresponding dual or inverted tropical curves are already glued together and then in the last step, seen on the right, the remaining gluing is done. The dashed lines represent different positions on where the ends can be, depending on the lengths of the cross-ratios.
\end{example}

\begin{theorem}
	Let $T$ be a triangulation of an $n$-gon with $d$ inner triangles, where all cross-ratios $C_1,...,C_{n-3}$ are interpreted in the dual way, as in Definition \ref{defthreeint}. Then, there are always $2^d$ different abstract rational tropical curves of multiplicity 1 in $\pi_{\U(T)}^{-1}(C_1,...,C_{n-3}),$ see Definition \ref{mappi}.
\end{theorem}
\begin{proof}
For any triangulation of an $n$-gon, we can split this $n$-gon up into smaller polygons as described in the proof of Corollary \ref{trop26} until there are no outer triangles left any more. At the end, these can be glued together again, as described in Lemma \ref{bijection}, so we can also w.l.o.g assume that $T$ has no outer triangles.	

Using Remark \ref{2dcurves} and Proposition \ref{partialfulfills}, we already know that in this setting there are $2^d$ tropical curves in $\pi_{\U(T)}^{-1}(C_1,...,C_{n-3})$. To prove this theorem we now only have to show that all of them are of multiplicity 1.

	For a set of cross-ratios in general position, we always obtain a tropical curve with only three-vertices. This tropical curve also has multiplicity 1, because the multiplicity of the dual or inverted tropical curves of the inner triangles is 1 as each edge contributes to exactly one cross-ratio and when gluing together two graphs of multiplicity 1 in the described way, we can observe in the same way as in Lemma \ref{bijection} that the multiplicity is not affected by performing a matrix operation with two columns corresponding to adjacent bounded edges on the path from the identified edges to the ends of the corresponding cross-ratio. 
	
	So, for each orientation of the inner triangles, there is exactly one tropical curve of multiplicity 1 fulfilling the demanded conditions and there are $2^d$ choices of orientation which tells us that for a triangulation on an $n$-gon with $d$ inner triangles, we obtain $2^d$ different tropical curves of multiplicity 1 each is in $\pi_{\U(T)}^{-1}(C_1,...,C_{n-3}).$
\end{proof}

\begin{remark}
	Let $T$ be a triangulation of an $n$-gon without outer triangles that defines the cross-ratios $C_1,...,C_{n-3},$ see Definition \ref{defthreeint}.
	Similarly to the method that was just described, we can also directly construct the tropical curves in $\pi_{\U(T)}^{-1}(C_1,...,C_{n-3})$ for all other interpretations of a cross-ratio, just in the case where all cross-ratios are interpreted in the dual way, we obtain the easier method of taking the dual or inverted tropical curves. 
	
	The tropical curve can be constructed as follows. For each inner triangle, with cross-ratios $C_i,$ $C_j$ and $C_k$ using the markings $a,b,c,d,e,f,$ we can take one of the two (or the one) $6$-marked tropical curves in $\pi_{\U(T')}^{-1}(C_i, C_j,C_k),$ where $T'$ is the triangulation of the hexagon with markings $a,b,c,d,e,f$ and the diagonals given by $C_i,$ $C_j$ and $C_k.$ How these $6$-marked tropical curves look like can also be seen in Table \ref{table} for any interpretation of the cross-ratios. Then we identify the shared edges as in the proof of the preceding Theorem. 
\end{remark}

\begin{remark}[Deflating $n$-gons]\label{inflate} 
	We look more closely at the case where all cross-ratios are interpreted in the neighboring way. Here, for every inner triangle that defines the cross-ratios $C_i,$ $C_j$ and $C_k$, we can take the edge graph that the corresponding three diagonals in the triangulation are forming. Then we add ends with markings to the ends of the diagonals and imagine the triangles to deflate and the diagonals to grow to the length that they are assigned to. If an inner triangle fulfills the triangle inequalities, its three diagonals become three bounded edges of the tropical curve where each diagonal is now part of two edges. This process can be seen in the first row of Figure \ref{inflateim}.
	If it does not fulfill the triangle inequalities, the two diagonals with the shorter lengths get cut at their common vertex and get glued onto the third one, beginning with the opposite vertices. In this case, there are two possibilities of which of the two markings at the common vertex of the shorter diagonals now stays with which of them, as seen in the middle and the bottom of Figure \ref{inflateim}.	
	
	These are now the unique 6-marked tropical curves in $\pi_{\U(T')}^{-1}(C_i, C_j,C_k),$ and now, we can construct the tropical curves in $\pi_{\U(T')}^{-1}(C_1, ...,C_{n-3})$ by identifying these 6-marked tropical curves along the edges of their common cross-ratios as described before in the case where all cross-ratios are interpreted in the dual way.
	\begin{figure}
		\centering
		\includegraphics[scale=0.4]{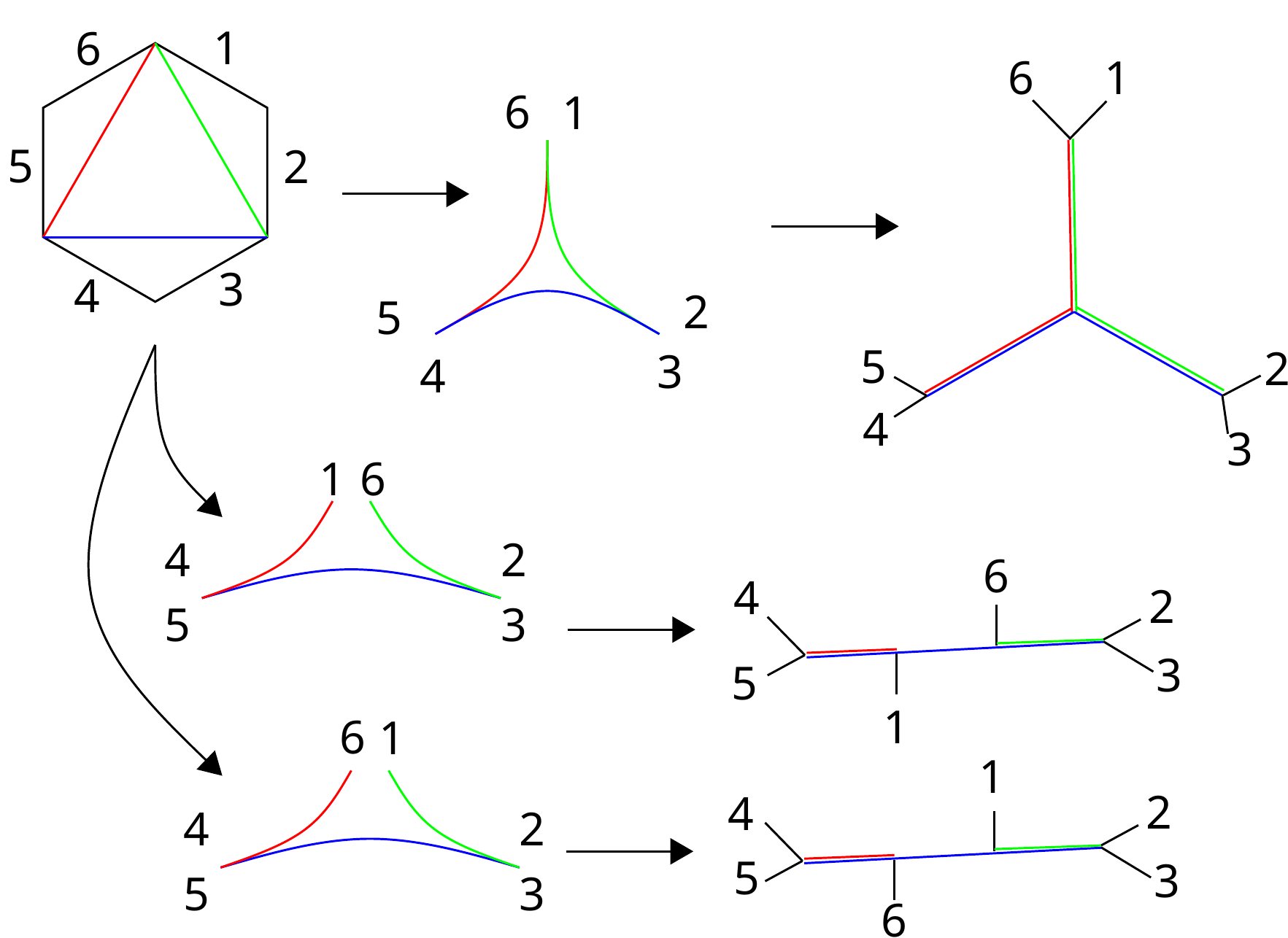}
		\caption{The tropical curves that fulfill the cross-ratio conditions given by that triangulation of a hexagon in the neighboring interpretation, see Example \ref{ex6b} and Remark \ref{inflate}}
		\label{inflateim}
	\end{figure}
\end{remark}
\end{document}